\documentclass[12 pt]{article}
\usepackage{amsmath,amsfonts,amsthm,amssymb,array,mathrsfs,latexsym,paralist, xypic, color,tabu}
\usepackage{tikz, soul}
\usepackage{lipsum}

\usepackage{dsfont}

\usetikzlibrary{arrows,automata}

\tikzstyle{V}=[fill=black,circle,scale=0.2, outer sep = 4pt]

\newtheorem{thm}{Theorem}[section]
\newtheorem{prop}[thm]{Proposition}
\newtheorem{cor}[thm]{Corollary}

\theoremstyle{remark}
\newtheorem{rmk}[thm]{Remark}
\newtheorem{example}[thm]{Example}
\theoremstyle{definition}
\newtheorem{defn}[thm]{Definition}

\newtheorem{hypothesis}[thm]{Hypothesis}

\DeclareMathOperator{\diam}{diam}

\renewcommand{\1}{{\bf 1}}

\newcommand{\bi}{\begin{itemize}}
\newcommand{\ei}{\end{itemize}}
\newcommand{\be}{\begin{enumerate}}
\newcommand{\ee}{\end{enumerate}}

\newcommand{\C}{\mathbb{C}}

\renewcommand{\H}{\mathcal{H}}

\newcommand{\K}{\mathcal{K}}

\newcommand{\R}{\mathbb{R}}
\newcommand{\N}{\mathbb{N}}

\providecommand{\keywords}[1]{{\textit{Key words and phrases:}} #1}
\providecommand{\classification}[1]{{\textit{2010 Mathematics Subject Classification:}} #1}

\def\IoIIdimdots(#1/#2/#3,#4){\node at (#1,#4) {$.$};\node at (#2,#4) {$.$};\node at (#3,#4) {$.$};}
\def\IIoIIdimdots(#1,#2/#3/#4){\node at (#1,#2) {$.$};\node at (#1,#3) {$.$};\node at (#1,#4) {$.$};}

\def\IoIIIdimdots(#1/#2/#3,#4,#5){\node at (#1,#4,#5) {$.$};\node at (#2,#4,#5) {$.$};\node at (#3,#4,#5) {$.$};}
\def\IIoIIIdimdots(#1,#2/#3/#4,#5){\node at (#1,#2,#5) {$.$};\node at (#1,#3,#5) {$.$};\node at (#1,#4,#5) {$.$};}
\def\IIIoIIIdimdots(#1,#2,#3/#4/#5){\node at (#1,#2,#3) {$.$};\node at (#1,#2,#4) {$.$};\node at (#1,#2,#5) {$.$};}

\usepackage[margin=0.77in]{geometry}
\AtEndDocument{\bigskip{\small%
  \textsc{Carla Farsi, Judith Packer : Department of Mathematics, University of Colorado at Boulder, Boulder, Colorado, 80309-0395, USA.} \par
  \textit{E-mail address}: \texttt{carla.farsi@colorado.edu, packer@euclid.colorado.edu} \par
  \addvspace{\medskipamount}
  \textsc{Elizabeth Gillaspy : Mathematisches Institut der Universit\"at M\"unster, Einsteinstrasse 62, M\"unster, 48149, Germany.}\par
  \textit{E-mail address}: \texttt{gillaspy@uni-muenster.de}\par
  \addvspace{\medskipamount}
  \textsc{Antoine Julien : Nord University Levanger, H\o gskoleveien 27, 7600 Levanger, Norway. }\par
  \textit{E-mail address}: \texttt{antoine.julien@nord.no}\par
   \addvspace{\medskipamount}
  \textsc{Sooran Kang : Department of Mathematics, Sungkyunkwan University, Seobu-ro 2066, Jangan-gu, Suwon, 16419, Republic of Korea.} \par
  \textit{E-mail address}, \texttt{sooran@skku.edu}
}}

\begin{document}

\title{Wavelets and spectral triples for higher-rank graphs}

\author{Carla Farsi, Elizabeth Gillaspy, Antoine Julien, Sooran Kang, and Judith Packer}

\date{\today}

\maketitle

\begin{abstract}

In this paper, we present two new ways to associate a spectral triple to a higher-rank graph $\Lambda$.  Moreover, we prove that these spectral triples are intimately connected to the wavelet decomposition of the infinite path space of $\Lambda$ which was introduced by Farsi, Gillaspy, Kang, and Packer in 2015.  We first introduce the concept of stationary $k$-Bratteli diagrams, to associate a family of ultrametric Cantor sets to a finite, strongly connected higher-rank graph $\Lambda$.  Then we show that {under mild hypotheses,} the  Pearson-Bellissard  spectral triples of  such Cantor sets have a regular $\zeta$-function, whose abscissa of convergence agrees with the Hausdorff dimension of the Cantor set, and that {the measure $\mu$ induced by the associated Dixmier trace  agrees with the measure $M$ on the infinite path space $\Lambda^\infty$ of $\Lambda$ which was introduced by an Huef, Laca, Raeburn, and Sims.  Furthermore, we prove that $\mu = M$ is a rescaled version of the Hausdorff measure of the ultrametric Cantor set.}

 From work of Julien and Savinien, we know that for $\zeta$-regular Pearson-Bellissard spectral triples, the eigenspaces of the associated Laplace-Beltrami operator constitute an orthogonal decomposition of $L^2(\Lambda^\infty, \mu)$; we show that this orthogonal decomposition refines the wavelet decomposition of Farsi et al. In addition, we generalize a spectral triple of Consani and Marcolli from Cuntz-Krieger algebras to higher-rank graph $C^*$-algebras, and prove that the wavelet decomposition of Farsi et al.~describes the eigenspaces of its Dirac operator. 
\end{abstract}

\classification{46L05, 46L87, 58J42.}

\keywords{Spectral triple, wavelets, higher-rank graph, Laplace-Beltrami operator,  Hausdorff measure, $\zeta$-function, Dixmier trace, $k$-Bratteli diagram, ultrametric Cantor set.}

\tableofcontents

\section{Introduction}
Both spectral triples and wavelets are algebraic structures which encode geometrical information.  However, to our knowledge our earlier paper \cite{FGJKP1} was the first to highlight a connection between wavelets and spectral triples.  In this paper, we expand the correspondence established in \cite{FGJKP1} between wavelets and spectral triples for the Cuntz algebras $\mathcal O_N$ to the setting of higher-rank graphs.  We also introduce a new spectral triple for the $C^*$-algebra $C^*(\Lambda)$ of a higher-rank graph $\Lambda$, and establish its compatibility with the wavelet decomposition of \cite{FGKP}.

The objective of the initial work on wavelet analysis, pioneered by    Mallat \cite{mallat},  Meyer \cite{meyer}, and  Daubechies \cite{daubechies} in the late 1980's, was to identify orthonormal bases or frames for     $L^2(\mathbb R^n)$  which behaved well under compression algorithms, for the purpose of signal or image storage.  A few years later,   Jonsson \cite{jonsson} and Strichartz \cite{strichartz} began to study orthonormal bases of more general Hilbert spaces $L^2(X, \mu)$ which arise from dilations and translations of finite sets of functions in $L^2(X, \mu)$.  
 When $X$ is a fractal space, Jonsson and Strichartz' wavelets reflect the self-similar structure of $X$; these fractal wavelets were the inspiration for the wavelet decompositions associated to graphs \cite{marcolli-paolucci} and higher-rank graphs \cite{FGKP} which are {the wavelets we focus on} in this paper. 

 Wavelet analysis has many applications in various areas of mathematics, physics and engineering. For example, it  has been used to study $p$-adic spectral analysis \cite{kozyrev}, pseudodifferential operators and dynamics on ultrametric spaces \cite{khrennikov-kozyrev,khrennikov-kozyrev2}, and the theory of quantum gravity \cite{ellis-mavromatos-nanopoulos-sakharov,battle-federbush-uhilg}.

 The idea of a spectral triple was  introduced by Connes in \cite{connes} as the noncommutative generalization  of  a compact Riemannian manifold.  A spectral triple consists of a representation $\pi$ of a pre-$C^*$-algebra $\mathcal{A}$ on Hilbert space $\mathcal{H}$, together with a Dirac-type operator $D$ on $\mathcal{H}$, which satisfy certain commutation relations.
In the case when $\mathcal{A} = C^\infty(X)$ is the algebra of smooth  functions on a compact spin manifold, Connes showed \cite{connes-reconstruction} that the algebraic structure of the associated spectral triple suffices to reconstruct the Riemannian metric on $X$. 

In addition to spin manifolds, Connes studied spectral triples for the triadic Cantor set and Julia set in \cite{connes}.   Shortly thereafter, Lapidus \cite{lapidus} suggested studying spectral triples $(\mathcal{A}, \H, D)$ where $\mathcal{A} $ is a commutative algebra of functions on a fractal space $X$, and 
investigating which aspects of the geometry of $X$ are recovered from the spectral triple.  Such spectral triples usually recover the fractal dimension of $X$, and in particularly nice cases \cite{pearson-bellissard, christensen-ivan-lapidus, guido-isola, lapidus-sarhad} they also recover the metric structure of the fractal space.

As we show in Theorem \ref{thm:JS-wavelets} below, the spectral triples $(\mathcal{A}, \H, D)$  of Pearson and Bellissard \cite{pearson-bellissard} are closely related to the wavelets of \cite{marcolli-paolucci,FGKP} (see Equation \eqref{eq:wavelet-intro} below).
 Here,  $\mathcal{A}  = C_{Lip}(X)$ is the algebra of Lipschitz continuous functions on an ultrametric Cantor set $X$ constructed from a weighted tree. Since the wavelet decompositions of \cite{marcolli-paolucci, FGKP} involve the infinite-path space $\Lambda^\infty$ of a graph or higher-rank graph $\Lambda$, we first explain in Section \ref{sec:background}  below how to construct a family of ultrametric Cantor sets $\{ (\Lambda^\infty, d_\delta): \delta \in (0,1)\}$ from a given higher-rank graph $\Lambda$. A $k$-dimensional generalization of directed graphs, higher-rank graphs (also called $k$-graphs) were introduced by Kumjian and Pask in \cite{kp} to provide computable, combinatorial examples of $C^*$-algebras.
The combinatorial character of $k$-graph $C^*$-algebras has facilitated the analysis of their structural properties,  such as  simplicity and ideal structure \cite{rsy2, robertson-sims, davidson-yang-periodicity, kang-pask, ckss}, quasidiagonality \cite{clark-huef-sims} and KMS states \cite{aHLRS, aHLRS1, aHKR}.  In particular, results such as \cite{spielberg-kirchberg,bnrsw,bcfs,prrs} show that  higher-rank graphs often provide concrete examples of $C^*$-algebras which are relevant to Elliott's  classification program for simple separable nuclear $C^*$-algebras.
 
Having laid the groundwork in Section \ref{sec:background} for associating a Pearson-Bellissard spectral triple to a $k$-graph, we study the $\zeta$-function and Dixmier trace of this spectral triple in Section \ref{sec:zeta-regular}. Inspired by Julien and Savinien \cite{julien-savinien}, we use Bratteli diagrams (more precisely, the stationary $k$-Bratteli diagrams which we introduce in Definition \ref{def-k-brat-diagrm}) to facilitate this analysis.  
  Theorem \ref{thm:abscissa-conv} establishes that the $\zeta$-function of the spectral triple associated to the ultrametric Cantor set $(\Lambda^\infty, d_\delta)$ has abscissa of convergence $\delta$; we later show in Corollary \ref{cor:Hausdorff-dim} that $\delta$ is also the Hausdorff dimension of $(\Lambda^\infty, d_\delta)$.  Weaker analogues of both of these results were obtained in \cite{pearson-bellissard}; we relied heavily on the Bratteli diagram structure to obtain our stronger results.
  
The structure of the stationary $k$-Bratteli diagrams also enables us to prove (in Theorem \ref{thm:dixmier-measure}) that  in our setting, the Pearson-Bellissard spectral triples satisfy the crucial hypothesis of $\zeta$-regularity invoked in \cite{pearson-bellissard}.  This implies the existence, for each $\delta \in (0,1)$, of a Dixmier trace $\mu_\delta$ on $\mathcal A = C_{Lip}(\Lambda^\infty)$, and an associated probability measure (also denoted $\mu_\delta$) on $\Lambda^\infty$.  Corollary \ref{cor:dixmier-aHLRS} shows that the probability measures $\mu_\delta$ agree with the Borel probability measure $M$ on $\Lambda^\infty$ which was identified in Proposition 8.1 of \cite{aHLRS} and which we used in \cite{FGKP} to construct a wavelet decomposition of $L^2(\Lambda^\infty, M)$.  Section \ref{sec:zeta-regular} concludes with Theorem \ref{thm:anHuef-Hausdorff-meas}, which proves that, after rescaling, the Hausdorff measure of $(\Lambda^\infty, d_\delta)$ also agrees with $M$ and $\mu_\delta$.

  Section \ref{sec-wavelets-as-eigenfunctions} presents the promised connection between the Pearson-Bellissard spectral triples and the wavelet decomposition of $L^2(\Lambda^\infty, M)$ from \cite{FGKP}.  In \cite{FGKP}, four of the authors of the current paper constructed an orthogonal decomposition 
  \begin{equation}
\label{eq:wavelet-intro}  
  L^2(\Lambda^\infty, M) = \mathscr V_0 \oplus \bigoplus_{n\geq 0} \mathcal W_n,\end{equation}
  where each subspace\footnote{The subspaces denoted in this paper by $\mathcal W_n$ were labeled $\mathcal{W}_{j,\Lambda}$ for $j\in \N$ in Theorem~4.2 of \cite{FGKP}.} $\mathcal W_n = \{ S_\lambda f: f \in \mathcal W_0, \lambda \in \Lambda^{(n, \ldots, n)}\}$ is constructed from $\mathcal W_0$ by means of generalized ``scaling and translation'' operators which reflect the (higher-rank) graph structure of $\Lambda$.  The geometric structure of this orthogonal decomposition led the authors of \cite{FGKP} to label it a wavelet decomposition, following Marcolli and Paolucci \cite{marcolli-paolucci}, Jonsson \cite{jonsson} and Strichartz \cite{strichartz}.
  
  As Julien and Savinien show in Theorem~4.3 of \cite{julien-savinien}, the Pearson-Bellissard spectral triples give rise to another orthogonal decomposition of $L^2(\Lambda^\infty, M)$.  To be precise, given an ultrametric Cantor set $(X, d)$ whose associated Pearson-Bellissard spectral triple is  $\zeta$-regular, write $\mu$ for the associated Dixmier trace.  Pearson and Bellissard constructed in \cite{pearson-bellissard} an associated family $\{\Delta_s: s \in \R\}$ of Laplace-Beltrami operators on $L^2(X, \mu)$.  Theorem~4.3 of \cite{julien-savinien} shows that the eigenspaces of $\Delta_s$ are independent of $s \in \R$ and they form an orthogonal decomposition of $L^2(X, \mu)$; moreover, when $X$ arises from a  Bratteli diagram, the eigenspaces $E_\gamma$ of $\Delta_s$ are labeled by the finite paths $\gamma$ in the Bratteli diagram.
  
  Theorem \ref{thm:JS-wavelets} of the current paper shows that these two orthogonal decompositions are compatible.  More precisely, it proves that 
  \[ \mathcal W_n = \bigoplus_{nk \leq |\gamma| < (n+1)k} E_\gamma,\]
  so the eigenspaces of the Laplace-Beltrami operators refine the wavelet decomposition of \cite{FGKP}.  The remainder of Section \ref{sec-wavelets-as-eigenfunctions} presents some variations of the wavelet decomposition of \cite{FGKP} which are also related to the Pearson-Bellissard spectral triples and their associated Laplace-Beltrami operators.
  
  Inspired by this close relationship between spectral triples and wavelets on the infinite path space of a $k$-graph, one might (naturally) ask whether there is a connection between the wavelets of \cite{marcolli-paolucci} or \cite{FGKP} and any of the known noncommutative spectral triples \cite{pask-rennie, consani-marcolli, goffeng-mesland, goffeng-mesland-ON}
associated to graph $C^*$-algebras. 
Of the spectral triples listed above, only the one given by  Consani and Marcolli  in \cite{consani-marcolli} uses $L^2(\Lambda^\infty, M)$ for the Hilbert space. 

We conclude this paper in Section \ref{sec-Consani-Marcolli-spectral-triples-for-k-graphs} by establishing a link between  Consani-Marcolli type spectral triples for graph or $k$-graph $C^*$-algebras and \cite{FGKP}'s wavelet decomposition of $L^2(\Lambda^\infty, M)$. 
More precisely, given a $k$-graph $\Lambda$, we  construct in Theorem~\ref{thm-Consani-Marcolli-spectral-triples-k-graphs} a spectral triple $(\mathcal{A}_\Lambda, L^2(\Lambda^\infty, M), D)$, where $\mathcal A_\Lambda$ is a dense (noncommutative) subalgebra of $C^*(\Lambda)$.  Theorem \ref{thm:CM-Dirac-wavelets} then establishes that the eigenspaces of the Dirac operator $D$ of this spectral triple agree with the wavelet decomposition of \cite{FGKP}.

\subsection*{Acknowledgments} We would like to thank Palle Jorgensen for helpful discussions about Bratteli diagrams associated to higher-rank graphs.  E.G.~was partially supported by the SFB 878 ``Groups, Geometry, and Actions'' of the Westf\"alische-Wilhelms-Universit\"at M\"unster. J.P.~was  partially supported by a grant from the Simons Foundation (\#316981).

\section{Higher-rank graphs and ultrametric Cantor sets}
\label{sec:background}
 In this section, we review the basic definitions and results that we will need about directed graphs, higher-rank graphs, (weighted/stationary) Bratteli diagrams, infinite path spaces, and (ultrametric) Cantor sets.  Throughout this article, $\N$ will denote the non-negative integers.

\subsection{Bratteli diagram}

A  \emph{directed graph} is given by a quadruple $E = (E^0, E^1, r, s)$, where $E^0$ is the set of vertices of the graph, $E^1$ is the set of edges, and $r, s: E^1 \to E^0$ denote the range and source of each edge.
A vertex $v$ in a directed graph $E$ is a \emph{sink} if $s^{-1}(v) = \emptyset;$ we say $v$ is a \emph{source} if $r^{-1}(v) = \emptyset$.

\begin{defn}\label{def:bratteli-diagram}\cite{bezuglyi-jorgensen}
A \emph{Bratteli diagram} $\mathcal{B}=(\mathcal{V}, \mathcal{E})$ is a directed graph with vertex set $\mathcal{V}= \bigsqcup_{n \in \N} \mathcal{V}_n$, and edge set $\mathcal{E} = \bigsqcup_{n\geq 1} \mathcal{E}_n$, where $\mathcal{E}_n$ consists of edges whose source vertex lies in $\mathcal{V}_{n}$ and whose range vertex lies in $\mathcal{V}_{n-1}$, and  $\mathcal{V}_n$ and $\mathcal{E}_n$ are finite sets for all $n$.

For a Bratteli diagram $\mathcal{B}=(\mathcal{V}, \mathcal{E})$, define a sequence of adjacency matrices $A_n=(f^n(v,w))_{v,w}$ of $\mathcal{B}$ for $n\ge 1$, where
\[
f^n(v,w)=\#\Big( \{ e\in \mathcal{E}_{n} : r(e)=v\in \mathcal{V}_{n-1}, \,s(e)=w \in \mathcal{V}_{n}\} \Big),
\]
where by $\#(Q)$ we denote the cardinality of the set $Q$.
A Bratteli diagram is \emph{stationary} if $A_n=A_1=:A$ are the same for all $n\ge 1$.  We say that $\eta$ is a \emph{finite} path of $\mathcal{B}$ if there exists $m\in \N$ such that $\eta=\eta_1\dots \eta_m$ for $\eta_i\in \mathcal{E}_i$, and in that case the \emph{length} of $\eta$, denoted by $| \eta|$, is $m$.
\end{defn}

\begin{rmk}
In the literature, Bratteli diagrams traditionally have $s(\mathcal{E}_n) = \mathcal{V}_{n}$ and $r(\mathcal{E}_n) = \mathcal{V}_{n+1}$; our edges point the other direction for consistency with the standard conventions for higher-rank graphs and their $C^*$-algebras.  

It is also common in the literature to require $| \mathcal{V}_0 | = 1$ and to call this vertex the \emph{root} of the Bratteli diagram; we will NOT invoke this hypothesis in this paper.
\end{rmk}

\begin{defn}
\label{def-k-brat-diagrm-inf-path-space} 
Given a Bratteli diagram $\mathcal{B}=(\mathcal{V}, \mathcal{E})$,  denote by $X_{\mathcal{B}}$ the set of all of its infinite paths:
\[ X_{\mathcal{B}} = \{(x_n)_{n \ge 1} : x_n \in \mathcal{E}_n \text{ and } s(x_n) = r(x_{n+1})\;\;\text{for $n\ge 1$}\}.
\]
 For each finite path $\lambda = \lambda_1 \lambda_2 \cdots \lambda_\ell$  in $\mathcal{B}$ with $r(\lambda) \in \mathcal{V}_0$, $\lambda_i\in \mathcal{E}_i$ and $|\lambda| = \ell$, 
 define the \emph{cylinder set} $[\lambda]$ by 
\[ [ \lambda ] = \{ x = (x_n)_{n \ge 1} \in X_{\mathcal{B}}:  x_i = \lambda_i \;\;\text{for}\;\;  1 \leq i \leq \ell\}.\]
The collection $\mathcal{T}$ of all cylinder sets forms a compact open sub-basis for a locally compact Hausdorff topology on $X_{\mathcal{B}}$ and cylinder sets are clopen; we will always consider $X_{\mathcal{B}}$ with this topology.
\end{defn}

The following proposition will tell us
when $X_{\mathcal{B}}$ is a {\em Cantor set}; that is, a totally disconnected, compact, perfect topological space.

\begin{prop} (Lemma 6.4. of  \cite{amini-elliott-golestani})
\label{pr:inf-path-cantor}
 Let $\mathcal{B} = (\mathcal{V},\mathcal{E})$ be a Bratteli diagram such that $\mathcal{B}$ has no sinks outside of $\mathcal{V}_0$, and no sources. Then $X_{\mathcal{B}}$ is a totally disconnected compact Haudorff space, and the following statements are equivalent:
\begin{enumerate}
\item The infinite path space $ X_{\mathcal{B}}$ of $\mathcal{B}$  is a  Cantor set;
\item For each infinite path $x = (x_1, x_2, ....)$ in $ X_{\mathcal{B}}$ and each $n \geq 1$ there is
an infinite path $y = (y_1, y_2, ....)$  with 
\[
x \not= y \ \text{and}\  x_k = y_k \text{ for } 1 \leq k \leq n;
\]
\item  For each $n \in \N$ and each $v \in \mathcal{V}_n$ there is $m \geq n$ and $w \in \mathcal{V}_m$ such that there is a path from $w$ to $v$ and 
\[
\#(r^{-1}(\{w\})) \geq 2.
\]
\end{enumerate}
\end{prop}

\subsection{Higher-rank graphs and stationary $k$-Bratteli diagrams}
\begin{defn}
\label{def-k-brat-diagrm}
Let $A_1, A_2, \cdots, A_k$ be  $N \times N$ matrices with non-negative integer entries.  The \emph{stationary $k$-Bratteli diagram} associated to the matrices $A_1, \ldots, A_k$, which we will call  $ \mathcal{B}_{(A_j)_{j=1,...,k} }$, is given by a filtered set of vertices $\mathcal{V} = \bigsqcup_{n\in \N} \mathcal{V}_n$ and a filtered set of edges $\mathcal{E} = \bigsqcup_{n\geq 1} \mathcal{E}_n$, where the edges in $\mathcal{E}_n$ go from $\mathcal{V}_{n}$ to $\mathcal{V}_{n-1}$, such that:
\begin{itemize}
\item[(a)] For each $n \in \N$, $\mathcal{V}_n$ consists of $N$ vertices, which we will label $1, 2, \ldots, N$. 

\item[(b)]  When $ n \equiv i \pmod{k}$, there are $A_i(p,q)$ edges whose range is the vertex $p$ of $\mathcal{V}_{n-1}$ and whose source is the vertex $q$ of $\mathcal{V}_{n}$. 
\end{itemize} 

\end{defn}

In other words, the matrix $A_1$ determines the edges with source  in $\mathcal{V}_1$ and range  in $\mathcal{V}_0$; then the matrix $A_2$ determines the edges with source in $\mathcal{V}_2$ and range  in $\mathcal{V}_1$; etc.  The matrix $A_k$ determines the edges with source in $\mathcal{V}_{k}$ and range in $\mathcal{V}_{k-1}$, and the matrix $A_1$ determines the edges with range in $\mathcal{V}_{k}$ and source in $\mathcal{V}_{k+1}$.

Note that a stationary 1-Bratteli diagram is often denoted a \emph{stationary Bratteli diagram} in the literature (cf.~\cite{bezuglyi-jorgensen, julien-savinien}).

Just as a directed graph has an associated adjacency matrix $A$ which also describes a stationary Bratteli diagram $\mathcal{B}_A$, the higher-dimensional generalizations of directed graphs known as \emph{higher-rank graphs} or $k$-graphs give us $k$ commuting matrices $A_1, \ldots, A_k$ and hence a stationary $k$-Bratteli diagram.  

  We use the standard terminology and notation for higher-rank graphs, which we review below for the reader's convenience.  

\begin{defn}
\label{def:k-graph}\cite{kp}
    A \emph{$k$-graph} is a countable small category $\Lambda$ equipped with a degree functor\footnote{We view $\N^k$ as a category with one object, namely $0$, and with composition of morphisms given by addition.} $d: \Lambda \to \N^k$ satisfying the \emph{factorization property}: whenever $\lambda$ is a morphism in $\Lambda$ such that $d(\lambda) = m+n$, there are unique morphisms $\mu, \nu \in \Lambda$ such that $d(\mu) = m, d(\nu) = n$, and $\lambda = \mu \nu$.   
    
We use the arrows-only picture of category theory; thus, $\lambda \in \Lambda$ means that $\lambda$ is a morphism in $\Lambda$.  
For $ n \in \N^k$, we write 
\[ \Lambda^n := \{ \lambda \in \Lambda: d(\lambda) = n\} .\]
When $n=0$,  $\Lambda^0$ is the set of objects of $\Lambda$, which we also refer to as the \emph{vertices} of $\Lambda$. 

Let  $r, s: \Lambda \to \Lambda^0$  identify the range and source of each morphism, respectively.  For $v \in \Lambda^0$ a vertex, we define 
\[v\Lambda^n := \{\lambda \in \Lambda^n : r(\lambda) = v\} \text{ and } \Lambda^n w := \{ \lambda \in \Lambda^n: s(\lambda) = w\}.\]
We say that $\Lambda$ is \emph{finite} if $\#(\Lambda^n )< \infty$ for all $n \in \N^k$, and  we say $\Lambda$ is \emph{source-free} or \emph{has no sources} if $\#(v\Lambda^n ) > 0$ for all $v \in \Lambda^0$ and   $n \in \N^k$.

For $1 \leq i \leq k$, write $e_i$ for the $i$th standard basis vector of $\N^k$, and define a matrix $A_i \in M_{\Lambda^0}(\N)$ by 
\[ A_i(v, w) = \#( v \Lambda^{e_i} w ).\]
We call $A_i$ the \emph{$i$th adjacency matrix} of $\Lambda$.  Note that the factorization property implies that the matrices $A_i$ commute.
\end{defn}

Despite their formal definition as a category, it is often useful to think of $k$-graphs as $k$-dimensional generalizations of directed graphs.  In this interpretation, $\Lambda^{e_i}$ is the set of ``edges of color $i$'' in $\Lambda$.  The factorization property implies that each $\lambda \in \Lambda$ can be written as a concatenation of edges in the following sense: A morphism $\lambda \in \Lambda$ with $d(\lambda) = (n_1, n_2, \ldots, n_k)$ can be thought of as a $k$-dimensional hyper-rectangle of dimension $n_1 
\times n_2 \times \cdots \times n_k$.  Any minimal-length lattice path in $\N^k$ through the rectangle lying between 0 and $(n_1, \ldots, n_k)$ corresponds to a choice of how to order the edges making up $\lambda$, and hence to a unique decomposition or ``factorization'' of $\lambda$.  For example, the lattice path given by walking in straight lines from $0$ to $(n_1, 0, \ldots, 0)$ to $(n_1, n_2, 0, \ldots, 0)$ to $(n_1, n_2, n_3, 0, \ldots, 0)$, and so on,   corresponds to the factorization of  $\lambda$ into edges of color 1, then edges of color 2, then edges of color 3, etc.

For any directed graph $E$, the category of its finite paths $\Lambda_E$ is a 1-graph; the degree functor $d: \Lambda_E \to \N$ takes a finite path $\lambda$ to its length $|\lambda|$.  Example \ref{ex:Omega-k} below gives a less trivial example of a $k$-graph.  The $k$-graphs $\Omega_k$ of Example \ref{ex:Omega-k}
are also fundamental to the definition of the space of infinite paths in a $k$-graph.

\begin{example}
For $k\ge 1$, let $\Omega_k$ be the small category with 
\[ \text{Obj}\, (\Omega_k) = \N^k, \  \text{Mor} \, (\Omega_k) = \{ (m, n) \in \N^k \times \N^k: m \leq n \}, \quad \ r(m,n) = m, \ s(m,n) = n.\]
If we define $d:\Omega_k\to \N^k$ by $d(m,n)=n-m$, then $\Omega_k$ is a $k$-graph with degree functor $d$.
\label{ex:Omega-k}
\end{example}

\begin{defn} 
\label{def:inf-path-space}
Let $\Lambda$ be a $k$-graph.
An \emph{infinite path} of $\Lambda$ is a $k$-graph morphism 
\[
x: \Omega_k \to \Lambda;
\]
 we write $\Lambda^\infty$ for the set of infinite paths in $\Lambda$.  For each $p \in \N^k$, we have a map $\sigma^p: \Lambda^\infty \to \Lambda^\infty$ given by 
  \[
  \sigma^p(x)(m,n) = x(m+p, n+p)
  \]
  for $x\in \Lambda^\infty$ and $(m,n)\in \Omega_k$.
  
\end{defn}
 
 \begin{rmk} \label{rmk:S2_shift}
 \begin{itemize}
\item[(a)] Given $x\in \Lambda^\infty$, we often write $r(x):=x(0)=x(0,0)$ for the terminal vertex of $x$. This convention means that an infinite path has a range but not a source. 

We equip $\Lambda^\infty$ with the topology generated by the sub-basis $\{[\lambda]: \lambda \in \Lambda\}$ of compact open sets, where 
\[ [\lambda] = \{x\in \Lambda^\infty: x(0, d(\lambda)) = \lambda\}.\]
Note that we use the same notation for a cylinder set of $\Lambda^\infty$ and a cylinder set of $X_{\mathcal{B}}$ in Definition~\ref{def-k-brat-diagrm-inf-path-space} since $\Lambda^\infty$ is homeomorphic to $X_{\mathcal{B}_\Lambda}$ for a finite, source-free $k$-graph $\Lambda$. See the details in Proposition~\ref{pr:kgraph-bratteli-inf-path-spaces}.
Remark 2.5 of \cite{kp} establishes that, with this topology, $\Lambda^\infty$ is a locally compact Hausdorff space.
 
\item[(b)] For any $\lambda \in \Lambda$ and any $x \in \Lambda^\infty$ with $r(x) = s(\lambda)$, we write $\lambda x$ for the unique infinite path $y \in \Lambda^\infty$ such that $y(0, d(\lambda)) = \lambda$ and $\sigma^{d(\lambda)}(y) = x$.  
If $d(\lambda) = p$, the maps $\sigma^{p}$ and $\sigma_\lambda := x \mapsto \lambda x$ are local homeomorphisms which are mutually inverse: 
\[\sigma^p \circ \sigma_\lambda = id_{[s(\lambda)]}, \quad \sigma_\lambda \circ \sigma^p = id_{[\lambda]},\]
although the domain of $\sigma^p$ is $\Lambda^\infty \supsetneq [\lambda]$. 

Informally, one should think of $\sigma^p$ as ``chopping off'' the initial segment of length $p$, and the map $x \mapsto \lambda x$ as ``gluing $\lambda$ on'' to the front of $x$.  By ``front'' and ``initial segment'' we mean the range of $x$, since an infinite path has no source.

\end{itemize}
\end{rmk}

We can now state precisely the connection between $k$-graphs and stationary $k$-Bratteli diagrams.

\begin{prop}
\label{pr:kgraph-bratteli-inf-path-spaces}
Let $\Lambda$ be a finite, source-free $k$-graph with adjacency matrices $A_1, \ldots, A_k$.  Denote by $\mathcal{B}_\Lambda$ the stationary $k$-Bratteli diagram associated to the matrices $\{A_i\}_{i=1}^k$.  Then $X_{\mathcal{B}_\Lambda}$ is homeomorphic to $\Lambda^\infty$.
\end{prop}
\begin{proof}
Fix $x\in \Lambda^\infty$ and write $\1 := (1, 1, \ldots, 1) \in \N^k$. Then the factorization property for $\Lambda^\infty$ implies that there is a unique sequence 
\[ (\lambda_i)_i \in \prod_{i=1}^\infty \Lambda^{\1}\]
such that $x = \lambda_1 \lambda_2 \lambda_3 \cdots$ with  $\lambda_i = x((i-1)\1, i \1)$. (See the details in Remark~2.2 and Proposition~2.3 of \cite{kp}).  Since there is a unique way to write  $\lambda_i = f_1^i f_2^i \cdots f_k^i$ as a composable sequence of edges with $d(f_j^i) = e_j$, we have
\[x = f_1^1 f_1^2 \cdots f_1^k f_2^1 f_2^2 \cdots f_2^k f_3^1 \cdots\; , \]
 where the $n k + j$th edge has color $j$. Thus, for each $j$, $f^i_j$ corresponds to an entry in $A_j$, and hence 
\[f_1^1 f_1^2 \cdots f_1^k f_2^1 f_2^2 \cdots f_2^k f_3^1 \cdots \in X_{\mathcal{B}_\Lambda}.\]

Conversely, given $y= (g_\ell)_\ell \in X_{\mathcal{B}_\Lambda}$, we  construct an associated $k$-graph infinite path
$\tilde y \in \Lambda^\infty$ as follows.  
To $y = (g_\ell)_\ell $ we associate a sequence $(\eta_n)_{n\geq 1}$ of finite paths in $\Lambda$, where 
\[\eta_n = g_1 \cdots g_{n k}\]
is the unique morphism in $\Lambda$ of degree $(n, \ldots, n)$ represented by the sequence of composable edges $g_1 \cdots g_{n k}$.  

Recall from \cite{kp} Remark 2.2 that  a morphism $\tilde y : \Omega_k \to \Lambda$ is uniquely determined by $\{\tilde y(0, n\1) \}_{n \in \N}$.  Thus, the sequence $(\eta_n)_n$ determines $\tilde y$:
\[ \tilde y(0, 0) = r(y) = r(g_1), \qquad \tilde y(0, n \1) := \eta_n \ \forall \ n \geq 1.\]
 The map $y \mapsto \tilde{y}$ is easily checked to be a bijection which is inverse to the map $x \mapsto  f_1^1 f_1^2 \cdots f_1^k f_2^1 f_2^2 \cdots f_2^k f_3^1 \cdots $.
 
Moreover, for any $i \in \N$,  $ 0 \leq j\leq k-1$, and any 
\[\lambda = f_1^1 f_1^2 \cdots f_1^k f_2^1 f_2^2 \cdots f_2^k f_3^1 \cdots f^i_j\]
 with $d(\lambda) = i \1  + (\overbrace{1, \ldots, 1}^j, 0, \ldots, 0)$, both of these bijections preserve the cylinder set $[\lambda]$.  In particular, these bijections preserve the ``square'' cylinder sets $[\lambda]$ associated to paths $\lambda$ with $d(\lambda) =  i \1 $ for some $i \in \N$.  (If $i =0$ then we interpret $d(\lambda) = 0 \cdot \1$ as meaning that $\lambda$ is a vertex in $\mathcal{V}_0 \cong \Lambda^0$.)
 From the proof of Lemma~4.1 of \cite{FGKP}, any cylinder set can be written as a disjoint union of square cylinder sets, and therefore the square cylinder sets generate the topology on $\Lambda^\infty$. We deduce that $\Lambda^\infty$ and $X_{\mathcal{B}_\Lambda}$ are homeomorphic, as claimed.
\end{proof}

\begin{rmk}
\label{rmk:no-bij-finite-paths}
\begin{itemize}
\item[(a)] Thanks to Proposition \ref{pr:kgraph-bratteli-inf-path-spaces}, we will usually identify the infinite path spaces $X_{\mathcal B_\Lambda}$ and $\Lambda^\infty$, denoting this space by the symbol which is most appropriate for the context.  In particular, the Borel structures on $X_{\mathcal B_\Lambda}$ and $\Lambda^\infty$ are isomorphic, and so any Borel measure on $\Lambda^\infty$ induces a unique Borel measure on $X_{\mathcal B_\Lambda}$ and vice versa.
\item[(b)] The bijection of Proposition \ref{pr:kgraph-bratteli-inf-path-spaces}
 between infinite paths in the $k$-graph $\Lambda$ and in the associated Bratteli diagram $\mathcal{B}_\Lambda$ does not extend to finite paths.
While any finite path in the Bratteli diagram determines a finite path, or morphism, in $\Lambda$, not all morphisms in $\Lambda$ have a representation in the Bratteli diagram. For example, if $e_1$ is a morphism of degree $(1,0, \ldots, 0)\in \N^k$ in a $k$-graph ($k > 1$) with $r(e_1) = s(e_1)$, the composition $e_1 e_1$ is a morphism in the $k$-graph which cannot be represented as a path on the Bratteli diagram.
 However, the proof of Proposition \ref{pr:kgraph-bratteli-inf-path-spaces} above establishes that ``rainbow'' paths in $\Lambda$ -- morphisms of degree $(\overbrace{q+1, \ldots, q+1}^j, q, \ldots, q)$ for some $q \in \N$ and $1 \leq j \leq k$ --  can be represented uniquely as paths of length $kq+j$ in the Bratteli diagram. 
\end{itemize}
\end{rmk}

\subsection{Ultrametrics on $X_{\mathcal{B}}$}
\label{sec:ultrametrics}
Although the Cantor set is unique up to homeomorphism, different metrics on it can induce quite different geometric structures. In this section, we will focus on Bratteli diagrams $\mathcal B$ for which the infinite path space $X_{\mathcal B}$ is a Cantor set.  In this setting, we
 construct ultrametrics on $X_{{\mathcal B}}$ by using weights on $\mathcal{B}$.  To do so, we first need to introduce some definitions and notation.

\begin{defn}  A metric $d$ on a Cantor set $\mathcal{C}$ is called an {\it ultrametric}
if $d$ induces the 
Cantor set topology and  satisfies
\begin{equation}\label{eq:strong-inequality}
d(x,y ) \leq \max\{ d(x,z),d(y,z)\} \quad\text{for all $x,y,z \in \mathcal{C}$}.
\end{equation}
\end{defn}
The inequality of \eqref{eq:strong-inequality} often called the strong triangle inequality.

\begin{defn}
\label{def:finite-path-Bratteli}
Let $\mathcal{B}$ be a Bratteli diagram.  Denote by $F\mathcal{B}$ the set of finite paths in $\mathcal{B}$ with range in $\mathcal{V}_0$.  For any $n \in \N$, we write 
\[ F^n \mathcal{B} = \{ \lambda \in F\mathcal{B}: |\lambda| = n \}.\]   

Given two (finite or infinite) paths $\lambda, \eta$ in $\mathcal{B}$, we say $\eta$ is a {\em sub-path} of $\lambda$ if there is a sequence $\gamma$ of edges, with $r(\gamma) = s(\eta)$, such that $\lambda = \eta \gamma$.

For any two infinite paths $x, y \in X_{\mathcal{B}}$, we define $x \wedge y $ to be the longest path $ \lambda \in F\mathcal{B}$ such that $ \lambda $ is a sub-path of $x$ and $y$.  We write 
 $x \wedge y = \emptyset$ when no such path $\lambda$ exists.
\end{defn}

\begin{defn}
\label{def:weight}
 A \emph{weight} on a Bratteli diagram $\mathcal{B}$ is a function $w: F\mathcal{B} \to \R^+$ such that 
\begin{itemize}
\item For any vertex $v \in \mathcal{V}_0, \ w(v) <1$.
\item  $\lim_{n\to \infty} \sup \{ w(\lambda): \lambda \in F^n\mathcal{B} \} = 0 .$
\item If $\eta$ is a sub-path of $\lambda$, then $w(\lambda) < w(\eta)$.
\end{itemize}
A Bratteli diagram with a weight often called a weighted Bratteli diagram and denoted by $(\mathcal{B},w)$.
\end{defn}
Observe that the third condition  implies that for any path $x = (x_n)_n \in \mathcal{B}$  (finite or infinite), 
\[w\left( x_1 x_2 \ldots x_n \right) > w\left( x_1 x_2 \cdots x_{n+1} \right) \quad\text{for all $n$}.\]

The concept above of a weight was inspired by Definition 2.9 of \cite{julien-savinien}; indeed, if one denotes a weight in the sense of \cite{julien-savinien} Definition~2.9 by $w'$, and defines $w(\lambda) := w'(s(\lambda))$, then  $w$ is a weight on $\mathcal{B}$ in the sense of Definition~\ref{def:weight} above.

\begin{prop}\label{pro:weight-to-metric}
Let $(\mathcal{B},w)$ be a weighted Bratteli diagram such that $X_{\mathcal{B}}$ is a Cantor set.  The function $d_w: X_\mathcal{B} \times X_{\mathcal{B}} \to \R^+$ given by 
\[ 
d_w(x,y) = \begin{cases}
1 & \text{ if } x \wedge y = \emptyset, \\
 0  &  \text{ if } x=y, \\
w(x \wedge y) & \text{ else.}
\end{cases}
\]
is an ultrametric on $X_{\mathcal{B}}$. 
\label{pr:weight-ultrametric}
\end{prop}
\begin{proof}
It is evident from the defining conditions of a weight that $d_w$ is symmetric and satisfies $d_w(x, y) = 0 \Leftrightarrow x=y$.
Since the inequality  \eqref{eq:strong-inequality}  is stronger than the triangle inequality, once we show that $d_w$ satisfies the ultrametric condition \eqref{eq:strong-inequality} it will follow that $d_w$ is indeed a metric.

To that end, first suppose that $d_w(x, y) = 1$; in other words, $x$ and $y$ have no common sub-path.  This implies that for any $z \in X_{\mathcal{B}}$, at least one of $d(x, z)$ and $d(y,z)$ must be 1, so 
\[ d_w(x, y) \leq \max \{ d_w(x, z), d_w(y,z)\},\]
as desired.  Now, suppose that $d_w(x, y) = w(x \wedge y ) < 1$. If $d_w(x, z) \geq d_w(x, y)$ for all $z \in X_\mathcal{B}$ then we are done.  On the other hand, if there exists $z \in X_\mathcal{B}$ such that $d_w(x, z) < d_w(x, y)$, then the maximal common sub-path of $x$ and $z$ must be  longer than that of $x$ and $y$.  This implies that 
\[ d_w(y, z) := w(y \wedge z) = w(y \wedge x) = d_w(x,y);\]
consequently, in this case as well we have $d_w(x, y) \leq \max \{ d(x, z), d_w(y,z)\}$.

Finally, we observe that the metric topology induced by $d_w$ agrees with the cylinder set topology.  To see this, fix $x = (x_n)_n \in X_{\mathcal{B}}$ and  $r >0$. Then the  conditions in Definition \ref{def:weight} imply that there is a smallest $n \in \N$ such that $w(x_1 \cdots x_n) < r$.  Then, 
\[B_r(x) = \{y \in X_\mathcal{B} : w(x \wedge y) < r \} =  [x_1 \cdots x_n],\]
so cylinder sets of $X_{\mathcal{B}}$ and open balls induced by the metric $d_w$ agree.  (If $n =0$ then we interpret $x_1 
\cdots x_n$ as $r(x)$.)
\end{proof}

\subsection{Strongly connected higher-rank graphs}
\label{sec:strongly-conn}
When $\Lambda$ is a finite $k$-graph whose adjacency matrices satisfy some additional properties, there is a natural family $\{w_\delta\}_{0 < \delta < 1}$ of weights on the associated Bratteli diagram $\mathcal{B}_\Lambda$ which induce  ultrametrics on the infinite path space $X_{\mathcal{B}_\Lambda}$.  We describe these additional properties on $\Lambda$ and the formula of the weights $w_\delta$  below.

\begin{defn}
A $k$-graph $\Lambda$ is \emph{strongly connected} if, for all $v, w \in \Lambda^0$, $v\Lambda w \not= \emptyset$.
\end{defn}

In Lemma 4.1 of \cite{aHLRS}, an Huef et al.~show that a finite $k$-graph $\Lambda$ is strongly connected if and only if the adjacency matrices $A_1, \ldots, A_k$ of $\Lambda$ form an \emph{irreducible family of matrices}.    Also,  Proposition 3.1 of \cite{aHLRS} implies that if $\Lambda$ is a finite strongly connected $k$-graph, then there is a unique positive vector $x^\Lambda \in (0, \infty)^{\Lambda^0}$ such that $\sum_{v \in \Lambda^0} x^\Lambda_v =1$ and  for  all $1 \leq i \leq k$,
\[ A_i x^\Lambda = \rho_i x^\Lambda,\]
where $\rho_i$ denotes the spectral radius of $A_i$.  We call $x^\Lambda$ the \emph{Perron-Frobenius eigenvector} of $\Lambda$. Moreover, an Huef et al.~constructed a Borel probability measure $M$ on $\Lambda^\infty$ in Proposition~8.1 of \cite{aHLRS} when $\Lambda$ is finite, strongly connected $k$-graph. 
The measure $M$ on $\Lambda^\infty$  is given by 

\begin{equation}\label{eq:kgraph_Measure}
M([\lambda])=\rho(\Lambda)^{-d(\lambda)}x^{\Lambda}_{s(\lambda)} \quad\text{for $\lambda\in\Lambda$,}
\end{equation}
where $x^\Lambda$ is the Perron-Frobenius eigenvector of $\Lambda$ and  $\rho(\Lambda)=(\rho_1,\dots, \rho_k)$, and for $n=(n_1, \ldots, n_k) \in \N^k$, 
\[ 
\rho(\Lambda)^n : = \rho^{n_1}_1 \cdots \rho_k^{n_k}.
\]

{We know from Remark \ref{rmk:no-bij-finite-paths} that every finite path $\lambda \in  {\mathcal B_\Lambda}$ corresponds to a unique morphism in $\Lambda$.  Using this correspondence and the homeomorphism $X_{\mathcal B_\Lambda} \cong \Lambda^\infty$ of Proposition \ref{pr:kgraph-bratteli-inf-path-spaces}, Equation \eqref{eq:kgraph_Measure} translates into the formula }
\begin{equation}\label{eq:M-measure}
M([\lambda]) = (\rho_1 \cdots \rho_t)^{-(q+1)} (\rho_{t+1} \cdots \rho_k)^{-q} x^\Lambda_{s(\lambda)}
\end{equation}
for $[\lambda] \subseteq X_{\mathcal B_\Lambda}$, where $\lambda\in F\mathcal{B}_\Lambda$ with $|\lambda| = qk +t$ and $x^\Lambda$ is the Perron-Frobenius eigenvector of $\Lambda$.

In the proof that follows, we rely heavily on the identification between $\Lambda^\infty$ and $X_{\mathcal{B}_\Lambda}$ of Proposition \ref{pr:kgraph-bratteli-inf-path-spaces}.  We also use the observation from Remark \ref{rmk:no-bij-finite-paths} that every finite path in $F \mathcal{B}_\Lambda$ corresponds to a unique finite path $\lambda \in \Lambda$.

\begin{prop}
\label{pr:spec-rad-cantor}
Let $\Lambda$ be a finite, strongly connected $k$-graph with adjacency matrices $A_i$. Then the infinite path space $\Lambda^\infty$ is a Cantor set whenever $\prod_i \rho_i > 1$.
\end{prop}

\begin{proof}
 We let $A = A_1 \ldots A_k$; it is a matrix whose entries are indexed by $\Lambda^0 \times \Lambda^0$, and its spectral radius is $\prod_i \rho_i$.
 We assume that $\Lambda^\infty$ is not a Cantor set, and will prove that the spectral radius of $A$ is at most $1$, hence proving the Proposition.
 
 Since $\Lambda^\infty$ is compact Hausdorff and totally disconnected, {but not a Cantor set, }it has an isolated point $x$. We write $\{\gamma_n\}_{n \in \N}$ for the increasing sequence of finite paths in $\mathcal B_\Lambda$ which are sub-paths of $x$.  If $n = \ell k +t$, then $|\gamma_n| = n$ and (thinking of $\gamma_n$ as an element of $\Lambda$) $d(\gamma_n)=(\ell +1, \ldots, \ell +1, \ell , \ldots, \ell)$ with $t$ occurrences of $\ell +1$.
 Since $x$ is an isolated point, there exists $N \in \N$ such that for all $n \geq N$, $[\gamma_n] = \{ x\}$.  Without loss of generality, we can assume that $N=dk$ is a multiple of $k$, so that $d(\gamma_N) = (d, \ldots, d)$.
 For $n \geq N$, we write $\gamma_n = \gamma_N \eta_n$, with $|\gamma_n|=n$ and $|\eta_n|=n -N = qk+t$, so that $d( \eta_n)=(q+1, \ldots, q+1, q, \ldots, q)$, with $t$ occurrences of $q+1$. 
  
Our hypothesis that $x$ is an isolated point implies that for all $n \geq N$, $\eta_{n}$ is the unique path of  degree $d(\eta_n)$ whose range is $s(\gamma_N) = r(\eta_n)$.
 This, in turn, implies that for all $n \geq N$, we have $A^q A_1 \ldots A_{t} (r(\eta_{n}), z)$ is equal to $1$ for a single $z$, and $0$ otherwise.
 In other words, if we consider the column vector $\delta_{v}$ which is $1$ at the vertex $v$ and $0$ else, we have that
 \[
  \big(\delta_{r(\eta_{n})} \big)^T \cdot A^q A_1 \ldots A_t = \big( \delta_{s(\eta_{n})} \big)^T.
 \]

Note that for each $n \geq N$ with $n - N=qk+t$, 
$s(\eta_{n+1})$ is the label of the only non-zero entry in row  $s(\eta_n)$ of the matrix $A_{t}$.
Since each entry in the sequence $(s(\eta_n))_{n \in \N}$ is completely determined by a finite set of inputs -- namely, the previous entry in the sequence, and the entries of the matrices $A_t$ -- and the set $\Lambda^0$ of vertices is finite, the sequence $(s(\eta_n))_{n \in \N}$  is eventually periodic. 
 Let $p$ be a period for this sequence. Then $kp$ is also a period, so there exists $J$ such that for all $n \geq J$ 
 we have
 \[(A^{p})^T \delta_{s(\eta_{n})} = \delta_{s(\eta_{n})}.\]
{We average along one period and define} 
 \[
  \vec v = \frac{1}{kp} \sum_{j=J+1}^{J + kp} \delta_{s(\eta_{j})},
 \]
 we compute
 \[
  A^T \vec v = \frac{1}{kp} \sum_{j=J+1}^{J+kp} \delta_{s(\eta_{j})}  = \vec{v},
 \]
 so $\vec v$ is an eigenvector of $A^T$ with eigenvalue $1$, with nonnegative entries.

Since $\Lambda$ is strongly connected by hypothesis, Lemma 4.1 of \cite{aHLRS} implies that there exists a matrix $A_F$ which is a finite sum of finite products of the matrices $A_i$ and which has positive entries. This matrix $A_F$ commutes with $A$, and therefore
 \[
  A^T A_F^T \vec v = A_F^T A^T \vec v = A_F^T \vec v,
 \]
 and so $\vec u := A_F^T \vec v$ is an eigenvector of $A^T$ with eigenvalue $1$.
 Since $A_F$ is positive and $\vec v$ is nonnegative, $\vec u$ is positive. Therefore, we can apply Lemma~3.2 of~\cite{aHLRS} and conclude that $\prod_i \rho_i= \rho(A) \leq 1$.
\end{proof}

\begin{rmk}
The proof of Proposition \ref{pr:spec-rad-cantor} simplifies considerably if we add the hypothesis that each row sum of each adjacency matrix $A_i$ is at least 2.  In this case, any finite path $\gamma$ in the Bratteli diagram has at least two extensions $\gamma  e$ and $\gamma  f$. In terms of neighbourhoods, this means that each clopen set  $[\gamma]$ contains at least two disjoint non-trivial sets $[\gamma e], [\gamma f]$. It is therefore impossible to have a cylinder set $[\gamma]$ consist of a single point. Therefore, there is no isolated point in $X_{\mathcal B_\Lambda}$, and the path space is a Cantor set.
\end{rmk}

\begin{prop}
\label{pr:delta-weight}
Let $\Lambda$ be a finite, strongly connected $k$-graph with adjacency matrices $A_i$. For $\eta\in F\mathcal{B}_\Lambda$ with $|\eta|=n\in \N$, write $n = qk + t$ for some $q, t \in \N$ with $	0 \leq t \leq k-1$.  For each $\delta \in (0,1),$ define $w_\delta: F\mathcal{B}_\Lambda \to \R^+$ by 
\begin{equation}
w_\delta(\eta) = \left(\rho_1^{q +1} \cdots \rho_t^{q+1} \rho_{t+1}^q \cdots \rho_k^q \right)^{-1/\delta} x^\Lambda_{s(\eta)},
\label{eq:w-delta}
\end{equation}
where $x^\Lambda$ is the unimodular Perron-Frobenius eigenvector for $\Lambda$.
If the spectral radius $\rho_i$ of $A_i$ satisfies $\rho_i > 1 \ \forall \ i$, then $w_\delta$ is a weight on $\mathcal{B}_\Lambda$.
\end{prop}
\begin{proof}
Recall that $x^\Lambda \in (0,\infty)^{\Lambda^0}$, $\sum_{v\in\Lambda^0} x^\Lambda_v=1$ and $A_i x^\Lambda=\rho_i x^\Lambda$ for all $1\le i\le k$; thus, for any $v \in \Lambda^0 \cong \mathcal{V}_0$, $w_\delta(v) = x^\Lambda_v <1$, and the first condition of Definition \ref{def:weight} is satisfied.  Since  $\rho_i > 1$ for all $i$ and $0 < \delta < 1$, 
\[ \lim_{q\to \infty} (\rho_i^q)^{-1/\delta} = \lim_{q \to \infty} \left( \frac{1}{\rho_i^{1/\delta}} \right)^q = 0.\]
Thus the second condition of Definition \ref{def:weight} holds.  To see the third condition, we observe that it is enough to show that $w_\delta(\lambda) > w_\delta(\lambda f)$ for any edge $f$ of any color with $s(\lambda)=r(f)$.    Note that if $|\lambda| = qk + j$ for $q\in \N$ and $0\le j\le k-1$, so that $s(\lambda) \in \mathcal{V}_{qk + j}$, then 
\begin{align*}
 \sum_{\stackrel{f: r(f) = s(\lambda)}{d(f) = e_{j+1}}}  w_\delta(\lambda f) &  = \left( ( \rho_1 \cdots \rho_k)^q \rho_1 \ldots \rho_{j +1}\right)^{-1/\delta} \sum_{v \in \Lambda^0} A_{j+1}(s(\lambda)), v) x^\Lambda_v \\
 &= \left( ( \rho_1 \cdots \rho_k)^q \rho_1 \ldots \rho_{j} \right)^{-1/\delta} \rho_{j+1}^{-1/\delta}\rho_{j+1}  x^\Lambda_{s(\lambda)} \\
 & < w_\delta(\lambda).
 \end{align*}
 Here the second equality follows since $x^\Lambda$ is an eigenvector for $A_{j+1}$ with eigenvalue $\rho_{j+1}$, and the final inequality holds because  $\rho_{j+1} > 1$ and $1/\delta > 1,$ and consequently 
 \[\rho_{j+1}^{1-1/\delta} =\frac{1}{\rho_{j+1}^{1/\delta-1}} < 1.\]
\end{proof}
Our primary application for the results of this section is the following.

\begin{cor}\label{cor:ultrametric-Cantor}
Let $\Lambda$ be a finite, strongly connected $k$-graph with adjacency matrices $A_i$ and let $\rho_i$ be the spectral radius for $A_i$, $1\le i\le k$. Suppose that $\rho_i >1$ for all $1\le i\le k$. Let $(\mathcal{B}_\Lambda, w_\delta)$ be the associated weighted $k$-stationary Bratteli diagram given in Proposition~\ref{pr:delta-weight}. Then the infinite path space $X_{\mathcal{B}_\Lambda}$ is an ultrametric Cantor set with the metric $d_{w_\delta}$ induced by the weight $w_\delta$.
\end{cor}
\begin{proof}
Combine Proposition \ref{pr:delta-weight}, Proposition \ref{pr:spec-rad-cantor}, and Proposition \ref{pr:weight-ultrametric}.
\end{proof}

\section{Spectral triples and Hausdorff dimension for ultrametric higher-rank graph Cantor sets}
\label{sec:zeta-regular}
{Proposition~8 of \cite{pearson-bellissard} (also see Proposition~3.1 of \cite{julien-savinien}) gives a recipe for constructing an even spectral triple for any ultrametric Cantor set induced by a weighted tree.  We begin this section by explaining how this construction works in the case of the ultrametric Cantor sets 
associated to a  finite strongly connected $k$-graph as in the previous section.  In Section \ref{sec:zeta-dixmier},  we investigate the  $\zeta$-function and Dixmier trace of these spectral triples, and Section \ref{sec:hausdorff} computes the Hausdorff measure and Hausdorff dimension of the underlying Cantor sets. }

{To be precise, consider the Cantor set $\Lambda^\infty \cong X_{\mathcal B_\Lambda}$ with the ultrametric induced by the weight $w_\delta$ of Equation \eqref{eq:w-delta}.}  (Because of Proposition \ref{pr:kgraph-bratteli-inf-path-spaces}, we will identify the infinite path spaces of $\Lambda$ and of $\mathcal B_\Lambda$, and use either $\Lambda^\infty$ or $X_{\mathcal B_\Lambda}$ to denote this space, depending on the context.)
 Under additional (but mild) hypotheses, Theorem \ref{thm:abscissa-conv} establishes that the $\zeta$-function of the associated spectral triple has abscissa of convergence $\delta$.  After proving in Theorem \ref{thm:dixmier-measure} that the Dixmier trace of the spectral triple induces a well-defined measure $\mu_\delta$ on $\Lambda^\infty$, Corollary \ref{cor:dixmier-aHLRS} establishes that $\mu_\delta$ agrees with the measure $M$ introduced in \cite{aHLRS} and used in  \cite{FGKP} to construct a  wavelet decomposition of $L^2(\Lambda^\infty, M)$.  Finally, Theorem \ref{thm:anHuef-Hausdorff-meas} shows that in many cases, both $M$ and $\mu_\delta$ agree with the Hausdorff measure on the ultrametric Cantor set $\Lambda^\infty.$ 

Analogues of Theorems \ref{thm:abscissa-conv} and \ref{thm:dixmier-measure} were proved in Section 3 of \cite{julien-savinien} for stationary Bratteli diagrams (equivalently, directed graphs) with primitive adjacency matrices.  However, even for directed graphs our results in this section are stronger than those of \cite{julien-savinien}, since in this setting, our hypotheses are equivalent to saying  that the adjacency matrix is merely irreducible.

We begin by recalling the definition of a spectral triple.
\begin{defn}
Given a  pre-$C^*$-algebra $\mathcal{A}$, a faithful  $*$-representation $\pi: \mathcal{A} \to B(\H)$, and an unbounded operator $D$ on $\H$ such that 
\[ (D^2 +1)^{-1} \in \K(\H) \quad \text{ and } \quad [ D, \pi(a)] \in B(\H) \ \forall \ a \in \mathcal A,\]
we say that $(\mathcal{A}, \H, D)$ is an \emph{(odd) spectral triple}.  If $\H$ has a grading operator  -- a self-adjoint unitary $\Gamma$ -- such that 
\[ [\Gamma, \pi(a)] = 0 \ \forall \ a \in \mathcal{A} \quad \text{ and } \quad \Gamma D = - D \Gamma,\]
we say that $(\mathcal{A}, \H, D, \Gamma)$ is an \emph{even spectral triple}.

Sometimes the representation $\pi$ is also included in the notation for a spectral triple.
\label{def:spectral-triple}
\end{defn}

To any spectral triple, even or odd, we  associate a $\zeta$-function and Dixmier trace  as follows.

\begin{defn}
\label{def:zeta-fcn-dixmier-trace}
The \emph{$\zeta$-function} associated to a spectral triple $(\mathcal{A}, \mathcal{H}, D)$ is given by

\[ \zeta(s) = \frac{1}{2}\text{Tr}\,(|D|^{-s})\quad\text{{for $s\in \C$.}}\]
\end{defn}
The $\zeta$-function of a spectral triple is always a Dirichlet series since $|D|^{-1}$ is compact and hence has a decreasing sequence of eigenvalues.  Thus, 
 Chapter 2 of \cite{hardy-riesz} tells us
that $\zeta$ either converges everywhere, nowhere, or in the complex half plane $\text{Re}(s) \geq s_0$ for some $s_0$, which we call the 
 \emph{abscissa of convergence} of $\zeta$. 
 
\begin{rmk}
\label{rmk:s-in-R}
To determine the abscissa of convergence of the $\zeta$-function, it suffices to evaluate $\zeta$ at points $s \in \R$.  Since we are primarily interested in the abscissa of convergence of $\zeta$, throughout this article, we will only consider real arguments for $\zeta$.
\end{rmk}

\begin{defn}
\label{def:dixmier-trace}
If  the abscissa of convergence $s_0$ of the $\zeta$-function exists, then
the \emph{Dixmier trace} of an element $a\in \mathcal{A}$ is given by
\begin{equation}\label{eq:Dixmier_general}
\mu(a)=\lim_{s\searrow s_0}\frac{\text{Tr}\, (|D|^{-s}\pi(a))} {\text{Tr}\, (|D|^{-s})},
\end{equation}
where we take the limit over $s \in \R, s > s_0$.
\end{defn}

We now review the construction of the spectral triple from \cite{pearson-bellissard} (see also  Section 3 of \cite{julien-savinien}).
Let $(\mathcal{B}, w)$ be a weighted Bratteli diagram such that the infinite path space $X_{\mathcal{B}}$ is a Cantor set.  Let $(X_{\mathcal{B}}, d_w)$ be the associated ultrametric Cantor space. A \emph{choice function} for $(X_{\mathcal{B}}, d_w)$ is a map $\tau:F\mathcal{B}\to  X_{\mathcal{B}}\times X_{\mathcal{B}}$ such that $\tau(\gamma)=(\tau_{+}(\gamma), \tau_{-}(\gamma))\in [\gamma]\times [\gamma]$ and $d_w(\tau_{+}(\gamma),\tau_{-}(\gamma))=\text{diam}\,[\gamma]$, where 
\begin{equation} \text{diam}[\gamma]=\sup\{d_w(x,y)\mid x,y\in [\gamma]\}.
\label{eq:diam}
\end{equation}
We denote by $\Upsilon$ the set of choice functions for $(X_{\mathcal{B}}, d_w)$. Note that  $\Upsilon$ is nonempty whenever $X_{\mathcal{B}}$ is a Cantor set, %
because Condition (3) of Proposition \ref{pr:inf-path-cantor} implies that for every finite path $\gamma$ of $\mathcal{B}$ we can find two distinct infinite paths $x,y\in [\gamma]$.

As in \cite{pearson-bellissard, julien-savinien}, let $C_{\text{Lip}}(X_{\mathcal{B}})$ be the pre-$C^*$-algebra of Lipschitz continuous functions on $(X_{\mathcal{B}}, d_w)$ and let $\mathcal{H}=\ell^2(F\mathcal{B})\otimes \C^2$. For $\tau\in \Upsilon$, we define a faithful $\ast$-representation $\pi_\tau$ of $C_{\text{Lip}}(X_{\mathcal{B}})$ on $\mathcal{H}$ by
\[
\pi_\tau(f)=\bigoplus_{\gamma\in F\mathcal{B}}\begin{pmatrix} f(\tau_{+}(\gamma)) & 0 \\ 0 & f(\tau_{-}(\gamma)) \end{pmatrix}.
\]
A Dirac operator $D$ on $\mathcal{H}$ is given by
\[
D=\bigoplus_{\gamma\in F\mathcal{B}}\frac{1}{\text{diam}[\gamma] }\begin{pmatrix} 0 & 1\\ 1 & 0\end{pmatrix}.
\]
The grading operator $\Gamma$ is given by 
\[
\Gamma=1_{\ell^2(F\mathcal{B})}\otimes \begin{pmatrix} 1&0 \\ 0 & -1\end{pmatrix}.
\]

Then by Proposition 8 of \cite{pearson-bellissard}, $(C_{\text{Lip}}(X_{\mathcal{B}}), \mathcal{H}, \pi_\tau, D, \Gamma)$ is an even spectral triple for all $\tau\in \Upsilon$. 

 For the rest of the paper, we assume the following:
 \begin{hypothesis}
 { Given a weighted Bratteli diagram $(\mathcal{B},w)$,  the weight $w$  satisfies }
\begin{equation}\label{eq:weight-diam}
{w(\lambda) = \text{diam}[\lambda]\quad \text{for all } \lambda\in F\mathcal{B}. }
\end{equation}
\label{hypoth}
\end{hypothesis}

\begin{rmk}
\label{rmk:weight-diam}
Observe that if $\mathcal B = \mathcal{B}_\Lambda$ for a $k$-graph $\Lambda$, then {Equation~\eqref{eq:weight-diam} holds if and only if $\Lambda^0$ receives at least two edges of each color, i.e.~$\sum_{b\in \Lambda^0} A_i(a,b)\ge 2$ for all $a\in \Lambda^0$ and $1\le i\le k$. Since the spectral radius of a nonnegative matrix is at least the minimum of its row sums, Equation~\eqref{eq:weight-diam} implies that $\rho_i \geq 2 >1$ for all $1\le i\le k$, and hence $\rho=\rho_1\ldots \rho_k>1$. Therefore, if the function $w_\delta$ given in Equation \eqref{eq:w-delta} satisfies Equation~\eqref{eq:weight-diam}, then $w_\delta$ is a weight and it gives rise to an ultrametric Cantor set $X_{\mathcal{B}_\Lambda}$ by Corollary~\ref{cor:ultrametric-Cantor}.}
\end{rmk}
When Equation \eqref{eq:weight-diam} holds,
then the $\zeta$-function is given  by the  formula
\begin{equation}\label{eq:zeta_w}
\zeta_w(s) = \frac{1}{2}\text{Tr}\,(|D|^{-s}) 
 =\sum_{\lambda \in F\mathcal{B}} w(\lambda)^s.
\end{equation}

If, moreover, the abscissa of convergence $s_0$ of the zeta function $\zeta_w$ is finite  and  the Dixmier trace in \eqref{eq:Dixmier_general} exists, then  it induces a  measure on the infinite path space $X_{\mathcal{B}}$, whose explicit formula on cylinder sets is given by

\begin{equation}\label{eq:Dixmier_w}
\mu_w([\gamma]):=\mu_w(\chi_{[\gamma]}) =\lim_{s \searrow s_0} \frac{\sum_{\lambda \in F_\gamma \mathcal{B}} w(\lambda)^s}{\zeta_w(s)},
\end{equation}
where $F_\gamma \mathcal{B} = \{ \alpha \in F\mathcal{B}: \gamma \text{ is a sub-path of } \alpha \}$
is the set of finite paths which extend a finite path $\gamma$. By abuse of notation, we use the same notation $\mu_w$ for the induced measure as for the Dixmier trace.

Before we begin our analysis of the spectral triples associated to the ultrametric Cantor sets $(X_{\mathcal{B}_\Lambda}, d_{w_\delta})$, we first present sufficient conditions for the Dixmier trace to give a well-defined measure on $X_{\mathcal{B}}$.

\begin{prop}
\label{pr:dixmier-trace-measure}
 Let $(\mathcal{B}, w)$ be a weighted Bratteli diagram and let $(X_{\mathcal{B}}, d_w)$ be the associated ultrametric Cantor set. 
Suppose that {the weight $w$ satisfies Equation~\eqref{eq:weight-diam}}, and that
the $\zeta$-function $\zeta_w(s)$ in \eqref{eq:zeta_w} has  abscissa of convergence $s_0$.  
If $\mu_w ([\gamma]) < \infty$ for all cylinder sets $[\gamma]\in X_{\mathcal{B}}$,  then $\mu_w$ determines a unique finite measure on $X_{\mathcal{B}}$. 
\end{prop}
\begin{proof} This proof relies on Carath\'eodory's theorem.  Notice that 
\[\mathcal{F}:=\{ [\lambda]: \lambda \in F\mathcal{B} \} \]
is closed under finite intersections (if $[\lambda] \cap [\gamma] \not= \emptyset$, then either $\lambda$ is a sub-path of $\gamma$ or vice versa, and thus $[\lambda] \cap [\gamma] = [\gamma]$), and 
\[ [\lambda]^c = \bigsqcup_{|\lambda_i| = |\lambda|, \lambda_i \not= \lambda} [\lambda_i].\]
In other words, the complement of any element of $\mathcal{F}$ can be written as a finite disjoint union of elements of $\mathcal{F}$.

Since $\mathcal{F}$ generates the topology on $X_{\mathcal{B}}$, and $\mu_w([\gamma])$ is finite for all $[\gamma] \in \mathcal{F}$ by hypothesis, Carath\'eodory's theorem tells us that in order to show that $\mu_w$ determines a measure on $X_{\mathcal{B}}$, we merely need to check that $\mu_w$ is $\sigma$-additive on $\mathcal{F}$.  In fact, since the cylinder sets $[\gamma]$ are clopen, the fact that $X_{\mathcal{B}}$ is compact means that it is enough to check that $\mu_w$ is finitely additive on $\mathcal{F}$.  

We remark that since $s_0$ is the abscissa of convergence of $\zeta_w(s)$, we must have 
\begin{equation}
\label{eq:lim-delta-zeta}
\lim_{s \searrow s_0} \zeta_w(s) = \infty.
\end{equation}
Consequently, since $\mu_{w}([\gamma]) = \lim_{s\searrow s_0} \frac{\sum_{\lambda\in F_\gamma\mathcal{B} } w(\lambda)^s}{\zeta_w(s)} $, 
 in calculating $\mu_w([\gamma])$ we can ignore finitely many initial terms in the sum in the numerator.  In other words,  for any $L \in \N$, 
\begin{equation}
\label{eq:mu-without-finite-prefix}
\mu_w([\gamma]) = \lim_{s \searrow s_0} \frac{\sum_{\substack{\lambda \in F_\gamma \mathcal B, |\lambda| \geq L}} w(\lambda)^s}{\zeta_w(s)} 
\end{equation}
Thus, suppose that $[\gamma] = \bigsqcup_{i=1}^N [\lambda_i]$.  Write $L = \max_i |\lambda_i|$,  and for each $i$, write $[\lambda_i] = \bigsqcup_{\ell} [\lambda_{i, \ell}]$ where $|\lambda_{i, \ell}| = L$.  
If $\lambda\in F_\gamma \mathcal{B}$ with $|\lambda| \geq L$, then $\lambda_i$ is a sub-path  of $\lambda$ for precisely one $i$, and hence
\begin{equation*}
\begin{split}
\mu_w([\gamma]) &= \lim_{s \searrow s_0} \frac{\sum_{\substack{\lambda \in F_\gamma\mathcal{B}\\ |\lambda| \geq L}} w(\lambda)^s}{ \zeta_w(s)}= \lim_{s \searrow s_0} \sum_i \frac{\sum_{\lambda \in F_{\lambda_i} \mathcal{B}} w(\lambda)^s}{ \zeta_w(s)}  \\
&= \sum_i \mu_{w}([\lambda_i]) = \sum_{i, \ell} \mu_w([\lambda_{i, \ell}])  
\end{split}
\end{equation*}
For each fixed $i$, $\bigsqcup_\ell [\lambda_{i, \ell}] = [\lambda_i]$, so the same argument will show that $\mu_w([\lambda_i]) = \sum_{\ell} \mu_w([\lambda_{i, \ell}])$.  Thus, 
\[ \mu_w([\gamma]) = \sum_{i, \ell} \mu_w([\lambda_{i, \ell}]) = \sum_i \mu_w([\lambda_i]).\]
Since $\mu_w$ is finitely additive on $\mathcal{F}$, Carath\'eodory's theorem allows us to conclude that it gives a well-defined finite measure on $X_{\mathcal{B}}$.
\end{proof} 

\subsection{Properties of the $\zeta$-function and Dixmier trace}
\label{sec:zeta-dixmier}

{In this section, which involves some of the most intricate proofs in this paper,} we return to our focus on the even spectral triples $(C_{\text{Lip}}(X_{\mathcal{B}_\Lambda}), \mathcal{H}, \pi_\tau, D, \Gamma)$ associated to the weighted stationary $k$-Bratteli diagrams $(\mathcal{B}_\Lambda, w_\delta)$  of  Proposition~\ref{pr:delta-weight} above.
From now on, the $\zeta$-function and Dixmier trace of these spectral triples will be 
denoted by $\zeta_\delta$ and $\mu_\delta$ to emphasize that they depend on the choice of $\delta\in (0,1)$. {Similarly, we write $d_\delta$ for the ultrametric associated to $w_\delta$.}

We begin by showing that $\zeta_\delta$ has abscissa of convergence $\delta$.

\begin{thm}
\label{thm:abscissa-conv}
Let $\Lambda$ be a finite, strongly connected $k$-graph.  Fix $\delta \in (0,1)$ and suppose that Equation \eqref{eq:weight-diam} holds for the weight $w_\delta$ of Equation \eqref{eq:w-delta}. Then 
\[ \zeta_\delta(s) < \infty \quad\text{if and only if}\quad s > \delta.\]
\end{thm}
\begin{proof}
In order to explicitly compute $\zeta_\delta(s)$, we first observe that we can rewrite 
\[\zeta_\delta(s) = \sum_{\lambda \in F\mathcal{B}_\Lambda} w_\delta(\lambda)^s = \sum_{n \in \N} \sum_{\lambda \in F^n\mathcal{B}_\Lambda} w_\delta(\lambda)^s = \sum_{q\in \N} \sum_{t=0}^{k-1} \sum_{\lambda \in F^{qk+t} \mathcal{B}_\Lambda} w_\delta (\lambda)^s.\]
Now, write $A := A_1 \cdots A_k$ {for the product of the adjacency matrices of $\Lambda$}.  If $t\in \{0,1, \ldots ,k-1 \}$ is fixed and  $n = qk + t$, then the number  of paths in $F^n(B_\Lambda)$ with source vertex $b$ and range vertex $a$  is given by
\[  A^q A_1 \cdots A_t (a,b) ,\]
where $F^n(\mathcal{B}_\Lambda)$ is the set of finite paths of $\mathcal{B}_\Lambda$ with length $n$.
Thus, writing $\rho := \rho_1 \cdots \rho_k$ {for the spectral radius of $A$, the formula for $w_\delta$ given in Equation \eqref{eq:w-delta} implies that}

\begin{equation}
\label{eq-similar-num}
\zeta_\delta(s) = \sum_{t=0}^{k-1} \frac{1}{{(\rho_1 \cdots \rho_t)^{s/\delta} }}\sum_{q \in \N} \sum_{a,b\in \mathcal{V}_0} A^q A_1 \cdots A_t (a,b) \frac{(x^\Lambda_b)^s }{\rho^{qs/\delta} } .\\
\end{equation}
Since all terms in this sum are non-negative, the series $\zeta_\delta(s)$ converges iff it converges absolutely; hence, rearranging the terms in the sum does not affect the convergence of $\zeta_\delta(s)$.
Thus, we can rewrite 
\begin{equation}
\label{eq-explicit-zeta}
\zeta_\delta(s)= \sum_{t=0}^{k-1} \sum_{a,b, z\in \mathcal{V}_0} \frac{A_1 \cdots A_t(z,b)}{(\rho_1 \cdots \rho_t)^{s/\delta}} \sum_{q \in \N} A^q(a,z)\frac{(x^\Lambda_b)^s}{\rho^{qs/\delta}} \\  .
\end{equation}
In order to show that $\zeta_\delta(s)$ converges for $s > \delta$, we begin by considering the sum 
\[ \sum_{q \in \N} \frac{A^q(a,z) }{(\rho^{s/\delta})^q}.\]
Since $A$ has a positive right eigenvector of eigenvalue $\rho$ (namely $x^\Lambda$), 
Corollary 8.1.33 of \cite{horn-johnson-matrix-analysis} implies that 
\[ \frac{A^q(a,z)}{\rho^q} \leq \frac{\max \{ x^\Lambda_b\}_{b \in \mathcal{V}_0}}{\min \{ x^\Lambda_b\}_{b\in \mathcal{V}_0}} \ \forall \ q \in \N\backslash \{0\}.\]
Consequently, 
\[\sum_{q\in \N} \frac{A^q(a,z)}{\rho^q \rho^{(s/\delta-1)q}} \leq \delta_{a,z} + \frac{\max \{ x^\Lambda_b\}_{b\in \mathcal{V}_0}}{\min \{ x^\Lambda_b\}_{b\in \mathcal{V}_0}} \sum_{q \geq 1} \frac{1}{\rho^{(s/\delta-1)q}}.\]
If $s>\delta$, then our hypothesis that $\rho >1$ implies that $1/\rho^{(s/\delta -1)} \in (0,1),$ and thus 
$\sum_{q \geq 1}{\rho^{(1-s/\delta)q}}$ converges to $ (1-\rho^{(s/\delta-1)})^{-1}-1$.  
Consequently, 
\[ \sum_{q \in \N} \frac{A^q(a,z) }{(\rho^{s/\delta})^q} < \infty,\]
and hence $\zeta_\delta(s) < \infty$, for any $s > \delta$ since $\mathcal{V}_0$ is a finite set.

To see that $\zeta_\delta(s) = \infty$ whenever $s \leq \delta$, we have to work harder.
Theorem 8.3.5 of \cite{horn-johnson-matrix-analysis} implies that the Jordan form of $A$ is 
\[J= \begin{pmatrix} \rho &0&0&0&0&\ldots&0&0&0 & 0 & 0 & 0 &0\\ 
0 & \ddots & 0 & 0 & 0&   \ldots & 0 & 0 &0 & 0 &0 &0 &0 \\
0 & 0 & \rho & 0 & 0 & \ldots & 0 & 0 & 0 & 0 &0 &0 &0\\
0 & 0 & 0 &\omega_1 \rho &0&\ldots&0&0&0& 0 &0 &0&0  \\ 
0 & 0 & 0 & 0 & \ddots & 0 & \ldots & 0 & 0 & 0  &0 &0&0 \\
0 & 0 & 0 & 0& 0& \omega_1 \rho & 0 & \ldots & 0 & 0 & 0 &0 &0\\
0 & 0 &  0 & 0 & 0 &0&\omega_2 \rho &0&\ldots&0 & 0&0 &0\\ 
 \vdots &\ldots&\ldots& \ldots & \ldots  & \ldots & \ldots &\ddots &\ldots&\ldots &\ldots&\ldots & \vdots\\ 
  0 & 0 &0 & 0 & 0 &0&0&0&\omega_{p-1} \rho &0  &\ldots&0&0\\ 
   0& 0&0 &0 &0&0&0 & 0 &0&J_{p+1}&0& 0&0\\
   \vdots &\ldots&\ldots& \ldots & \ldots & \ldots&\ldots&\ldots&\ldots  &\ldots &\ddots &\ldots  &\vdots \\ 
 0 & 0 &0&0&0 & 0 &0&\ldots& \ldots & \ldots  &0&J_{m-1}&0\\
 0 & 0 &0&0&0 & 0 &0&\ldots&\ldots &\ldots  &0&0&J_{m}\\
 \end{pmatrix},
\]
where $p$ is the period of $A$, $\omega_i$ is a $p$th root of unity for each $i$,  each  eigenvalue $\omega_i \rho$ is repeated along the diagonal $m_i$ times, and  $J_i$, $i=p+1, \ldots, m$ are Jordan blocks -- that is, upper triangular matrices whose constant diagonal is given by an eigenvalue $\alpha_i$ of $A$ (with $|\alpha_i| < \rho)$ and which have a  superdiagonal of 1s as  the only other nonzero entries.
 Thus, for each $1 \leq a, b \leq |\mathcal{V}_0| ,$
  \begin{equation}
  J^q(a,b)  \in \{ 0 \} \cup \{\rho^q\} \cup \{  \rho^q \omega_i^q: 1 \leq i \leq p-1 \} \cup \left\{ \frac{1}{\alpha_i^\ell} \binom{q}{\ell} \alpha_i^q : 0 \leq \ell \leq \dim J_i \right\}.\label{eq:jordan-powers}
 \end{equation}
Consequently, 
  \[ \left| \frac{1}{\rho^q } J^q(a, b) \right| \in \{ 0 , 1 \} \cup \left\{ \beta_i  \frac{1}{ | \alpha^\ell_i|} \binom{q}{\ell}: \beta_i =\frac{ |\alpha_i|}{\rho}  < 1 , \ 0 \leq \ell \leq \dim J_i \right\}.\]

Thanks to \cite{rothblum81} and \cite[Chapter 2]{BP94}, 
we know that since $A$ has a positive eigenvector (namely $x^\Lambda$) of eigenvalue $\rho$, $\lim_{m \to \infty} \frac{1}{\rho^{mp + j}} A^{mp + j}$ exists for all $0 \leq j \leq p-1$, where $p$ denotes the period of $A$.  Moreover, if we write 
\begin{equation} A^{(j)} =\lim_{m \to \infty} \frac{1}{\rho^{mp + j}} A^{mp + j}\label{eq:A-j}
\end{equation}
 for this limit, and $\tau$ for the maximum modulus of the eigenvalues $\alpha_i$ of $A$ with $|\alpha_i| < \rho$, 
\[ \forall \ \left( \frac{\tau}{\rho} \right)^p < \beta < 1, \ \exists \ M_{\beta, j} \in \R^+ \ \text{s.t. } \forall \ m \in \N, \  \left| \frac{A^{mp+j}(a,b)}{\rho^{mp+j}} - A^{(j)}(a,b) \right| \leq M_{\beta, j} \beta^m .\]
 Thus, for all $ m\in \N$ and all $0 \leq j \leq p-1$, and all such $\beta$,
\begin{equation} \frac{A^{mp + j}(a,b)}{\rho^{mp +j}} \geq A^{(j)}(a,b) - M_{\beta,j} \beta^m \qquad \text{ for all } m \in \N. \label{eq:zeta-lower-bd}
\end{equation}

Reordering the summands of $\sum_{q \in \N} A^q(a,b) (\rho^{-s/\delta})^q$, we see that 
\[\sum_{q \in \N} A^q(a,b) (\rho^{-s/\delta})^q = \sum_{j =0}^{p-1} \sum_{m \in \N} A^{mp + j}(a,b) (\rho^{-s/\delta})^{mp + j}.\]

Now, fix $j \in \{ 0, \ldots, p-1\}$ and consider the sum 
\begin{align*} \sum_{m \in \N} A^{mp + j}(a,b) (\rho^{-s/\delta})^{mp + j} & = \sum_{m \in \N} \frac{A^{mp + j}(a,b)}{\rho^{mp+j}} \left(\frac{1}{\rho^{s/\delta-1}}\right)^{mp + j} \\
& > \frac{1}{\rho^{(s/\delta - 1)j}} \sum_{m \in \N} ( A^{(j)}(a,b) - M_{\beta,j} \beta^m )  \left(\frac{1}{\rho^{s/\delta-1}}\right)^{pm} .
\end{align*}
If $A^{(j)}(a,b) >0$, the fact that $\beta < 1$ and $M_{\beta, j} > 0$ implies that there exists $M$ such that for $m > M$, $A^{(j)}(a,b) > M_{\beta, j} \beta^m$.  Consequently, if we define 
\[ K = \frac{1}{\rho^{(s/\delta - 1)j}} \sum_{m=0}^{M} \frac{A^{(j)}(a,b) - M_{\beta, j} \beta^m}{\rho^{(s/\delta-1)pm}},\]
and write $\nu  = A^{(j)}(a,b) - M_{\beta, j} \beta^M > 0$, the fact that $\{ M_{\beta, j} \beta^m\}_{m \in \N}$ is a decreasing sequence implies that  
\[ \sum_{m \in \N} A^{mp + j}(a,b) (\rho^{-s/\delta})^{mp + j} > K + \frac{\nu}{\rho^{(s/\delta - 1)j}} \sum_{m > M}  \left(\frac{1}{\rho^{s/\delta-1}}\right)^{pm} .\]
Since $\rho > 1$ and $s \leq \delta$,  $\rho^{(1-s/\delta)p} \geq 1;$ consequently, the series $\sum_{m> M}  (\rho^{(1-s/\delta)p})^m$ diverges to infinity.  The fact that $K, \nu$ are finite now implies that  $\sum_{m \in \N} A^{mp + j}(a,b) (\rho^{-s/\delta})^{mp + j} $ also diverges to infinity if $A^{(j)}(a,b)  > 0$.

Now, we show that for each  $j$, there must exist  some $(a, b) \in \mathcal{V}_0$ such that $A^{(j)}(a,b) > 0$. 
Recall that $x^\Lambda$ is an eigenvector for $A$, and consequently for $A^{mp+j}$.  Thus,
\[\sum_{b \in \mathcal{V}_0} A^{mp +j}(a,b) x^\Lambda_b = \rho^{mp+j} x^\Lambda_a.\]
Since $x^\Lambda$ is a positive eigenvector, there exists $\alpha > 0$ such that $x^\Lambda_a > \alpha$ for all $a \in \mathcal{V}_0$.  Moreover, $x^\Lambda$ is a unimodular eigenvector, so $0 < x^\Lambda_b < 1$ for all $b \in \mathcal{V}_0$. Thus the above equation becomes 
\[ \rho^{mp+j} \alpha < \sum_{b \in \mathcal{V}_0} A^{mp +j}(a,b) x^\Lambda_b < \sum_{b \in \mathcal{V}_0} A^{mp+j}(a,b).\]
Consequently, for each $a \in \mathcal{V}_0$ and each $m \in \N$ there exists at least one vertex $b$ such that 
\[\frac{A^{mp+j}(a,b)}{\rho^{mp+j}} >\frac{\alpha}{\#(\mathcal{V}_0)}.\]  
Moreover, since $\#(\mathcal{V}_0) < \infty$, the definition of the limit $A^{(j)}$ implies that there exists $N\in \N$ such that whenever $m \geq N$ we have 
\[ A^{(j)}(a,b) > \frac{A^{mp+j}(a,b)}{\rho^{mp+j}} - \frac{\alpha}{2\#(\mathcal{V}_0)}\]
for all $a, b \in \mathcal{V}_0$.
Thus, if we fix $a$ and $m \geq N$, and choose $b$ such that $\frac{A^{mp+j}(a,b)}{\rho^{mp+j}} >\frac{\alpha}{\#(\mathcal{V}_0)},$
we have 
\begin{equation}
\label{eq:Aj-positive}
\forall \, \ 1 \leq j \leq p, \ \forall \,  a \in \mathcal V_0, \ \exists \ b \in \mathcal V_0 \ \text{s.t.}\qquad 
A^{(j)}(a,b)  > \frac{\alpha}{2\#(\mathcal{V}_0)} > 0.\end{equation}

Finally, recalling that the matrices $A_i$ commute, we observe that 
\[\sum_{z\in \mathcal{V}_0} A^{mp+j}(a,z) A_1 \cdots A_t(z,b) = (A_1 \cdots A_t) A^{mp +j}(a,b) = \sum_{z\in \mathcal{V}_0}A_1 \cdots A_t(a,z) A^{mp+j}(z,b).\]
Using this, we rewrite
\[ \zeta_\delta(s) = \sum_{a, b, z \in \mathcal{V}_0} \sum_{t=0}^{k-1} \frac{ A_1 \cdots A_t (a,z) (x^\Lambda_b)^s}{(\rho_1 \cdots \rho_t)^{s/\delta}} \sum_{j=0}^{p-1} \sum_{m\in \N} \frac{A^{mp + j}(z,b)}{\rho^{(mp+j)s/\delta}}.\]

It now follows from our arguments above that $\zeta_\delta(s)$
diverges whenever $s \leq \delta$.  To convince yourself of this, it may help to recall that $x^\Lambda_b$ is positive for all vertices $b$, and that (since $A_1 \cdots A_t(a,z)$ represents the number of paths of degree $(\overbrace{1, \ldots, 1}^t, 0, \ldots, 0)$ with source $z$ and range $a$) $\sum_a A_1 \cdots A_t(a,z) $ must be strictly positive for each $t$ since $\Lambda$ is source-free.  In other words, $\zeta_\delta(s)$ is computed by taking a bunch of sums that diverge to infinity when $s \leq \delta$, possibly adding some other positive numbers, multiplying the lot by some positive scalars, and adding the results.
 
   Consequently, $\delta$ is the abscissa of convergence of the $\zeta$-function $\zeta_\delta(s)$, as claimed.
\end{proof}
In the terminology of \cite{pearson-bellissard}, the following Theorem establishes the $\zeta$-regularity of the Pearson-Bellissard spectral triple associated to the ultrametric Cantor sets $(X_{\mathcal B_\Lambda}, d_{\delta})$.
\begin{thm}
\label{thm:dixmier-measure}
Let $\Lambda$ be a finite, strongly connected $k$-graph and fix $\delta \in (0,1)$. Suppose moreover that Equation \eqref{eq:weight-diam} holds on the associated ultrametric Cantor set $(X_{\mathcal B_\Lambda}, d_{\delta})$.
Then the associated Dixmier trace induces a finite measure $\mu_\delta$ on $X_{\mathcal{B}_\Lambda}$, where
\[
\mu_\delta([\gamma]):=\lim_{s \searrow \delta} \frac{\sum_{\lambda \in F_\gamma \mathcal{B}} w_\delta(\lambda)^s}{\zeta_\delta(s)}\quad \text{for $\gamma\in F\mathcal{B}_\Lambda$.}
\]
\end{thm}
\begin{proof}  Thanks to 
Proposition \ref{pr:dixmier-trace-measure}, it is enough to show that the limit $\mu_\delta([\gamma])$ is finite for all $ \gamma \in F\mathcal{B}_\Lambda$.  To this end,
we begin by computing a more explicit expression for $\zeta_\delta(s)$ when $s >\delta$.
Recall from our computations in \eqref{eq:jordan-powers}  of the Jordan form $J$ of $A$  that  for any $z,v \in \mathcal{V}_0$ we can find constants $c_i^{z,v}$ and polynomials $P_i^{z,v}$ such that 
for any $n\in \N$, we have 
\begin{equation}
\label{eq-analog-oflemma-3-6-of-JS-all-cases-0}
A^n(z,v) = c_{1}^{z,v}\rho^n
 + c_{2}^{z,v} \omega_1^n \rho^n
  + \cdots +  c_{p}^{z,v} \omega_{p-1}^{n} \rho^n
    +  \sum_{i=p+1}^m P_i^{z,v}(n) \alpha_i^n,
\end{equation}
where $\omega_i$ is a  $p$th root of unity for all $i$ (denoting by $p$ the period of $A$) and 
each $\alpha_i$ is an eigenvalue of $A$ with $|\alpha_i| <\rho$.  In a bit more detail, writing $A = C^{-1} J C$ for some invertible matrix $C$,
we have 
\[ c_i^{z,v} = \sum_{j=m_0 + \cdots + m_{i-1}+1}^{m_0 + \cdots + m_i}C^{-1}(z,j) C(j, v),\]
 and 
\[ P_i^{z,v}(n) = \sum_{(a,b): J_i^n(a,b)\not= 0} C^{-1}(z,a) C(b,v) \frac{1}{\alpha_i^{b-a}} \binom{n}{b-a}.\]
Recall that since $J_i$ is a Jordan block, $J_i^n(a,b) = 0$ unless $a \leq b$.  

Equivalently, setting $c_{z,v;n} = c_{1}^{z,v}
 + c_{2}^{z,v} \omega_1^n 
  + ....+  c_{p}^{z,v} \omega_{p-1}^{n} $, we have

\begin{equation}
\label{eq-analog-oflemma-3-6-of-JS-all-cases}
A^n(z,v) = c_{z,v;n}\rho^n
    +  \sum_{i=p+1}^m P_i^{z,v}(n) \alpha_i ^n.
\end{equation}

Therefore, 
 \begin{align*}
\zeta_\delta(s) &= \sum_{t=0}^{k-1}\sum_{a,b,z\in\mathcal{V}_0}\frac{A_1 \cdots A_t(z,b)}{{(\rho_1 \cdots \rho_t)^{s/\delta} }}  \sum_{q \in \N} A^q(a,z) \frac{(x^\Lambda_b)^s}{\rho^{qs/\delta}} \\ 
 &= \sum_{t=0}^{k-1} \sum_{a,z,b\in \mathcal{V}_0} \frac{A_1 \cdots A_t(z,b)}{(\rho_1 \cdots \rho_t)^{s/\delta}} (x^\Lambda_b)^s\sum_{q \in \N}\frac{c_{a,z;q} }{\rho^{(s/\delta-1)^q}} \\
 & \quad + \sum_{t=0}^{k-1} \sum_{a,b, z\in \mathcal{V}_0} \frac{A_1 \cdots A_t(z,b)}{{(\rho_1 \cdots \rho_t)^{s/\delta} }} \sum_{q \in \N}  \left(\sum_{i=p+1}^m P_i^{a,z}(q) \alpha_i^q\right)\frac{(x^\Lambda_b)^s}{\rho^{qs/\delta}} .
 \end{align*}
 
 Now, define 
 \[
 \tilde{c}_{q}(s) =\sum_{t=0}^{k-1}\sum_{a,b, z\in \mathcal{V}_0} \frac{A_1 \cdots A_t(z,b)}{{(\rho_1 \cdots \rho_t)^{s/\delta} }}c_{a,z;q}  {(x^\Lambda_b)^s},\quad \tilde{Q}_i(q,s)= \sum_{t=0}^{k-1} \sum_{a,b, z\in \mathcal{V}_0}\frac{A_1 \cdots A_t(z,b)}{{(\rho_1 \cdots \rho_t)^{s/\delta} }} P_i^{a,z}(q)  {(x^\Lambda_b)^s}.
\]
Since $\mathcal{V}_0$ is a finite set, we can rewrite $\zeta_\delta(s)$ as
 \begin{equation}\label{eq:zeta_1}
 \zeta_\delta(s) =  \sum_{q\in \N} \rho^{q(1-s/\delta)}  \tilde{c}_{q}(s)  + \sum_{q\in \N} {\sum_{i=p+1}^m} \tilde{Q}_i(q,s) \left(\frac{\alpha_i}{\rho^{s/\delta}}\right)^q.
 \end{equation}
 
We now observe that the series $\sum_{q\in \N} \tilde{c}_q(s,t) \rho^{q(1-s/\delta)} $ converges by the Root Test.  To be more precise,     the fact that each $\omega_i$ is a  $p$th root of unity implies that 
\[ c_{a,z; q} = c_{a,z; q+p} \ \forall\ q \in \N.\]
Moreover, if we consider the limit 
$A^{(j)}(z,v) = \lim_{p \to \infty} \frac{A^{mp+j}(z,v)}{\rho^{mp +j}}, $
Equation \eqref{eq-analog-oflemma-3-6-of-JS-all-cases} implies that 
\begin{equation}
\label{eq:Aj-c-coeffs}
A^{(j)}(z,v) = c_{z,v;j}.
\end{equation}
Consequently, $c_{z,v;j}$ is a non-negative real number for all $z, v \in \mathcal V_0$.  Moreover, Equation \eqref{eq:Aj-positive} implies that for all $1 \leq j \leq m$ and all  $z \in \mathcal V_0$, there exists $v \in \mathcal V_0$ such that $c_{z,v; j} \not= 0$.

Thus,
\begin{align*}
\sum_{q\in\N} c_{a,z;q} \rho^{q(1-s/\delta)} &= \sum_{j=0}^{p-1} c_{a,z; j} \sum_{l\in \N} \rho^{(lp+j)(1-s/\delta)} = \sum_{j=0}^{p-1} c_{a,z; j} \rho^{j(1-s/\delta)} \sum_{l\in \N} \rho^{lp(1-s/\delta)}\\
&= \sum_{j=0}^{p-1} c_{a,z; j} \rho^{j(1-s/\delta)} \frac{1}{1-\rho^{p(1-s/\delta)}};
\end{align*}
the last equality follows because 
$\rho^{1-s/\delta} < 1$ whenever $s > \delta$, and hence $\sum_l \left(\rho^{p(1-s/\delta)})\right)^l $ is a geometric series with ratio $\rho^{p(1-s/\delta)} < 1$.  Thus, the first term of $\zeta_\delta(s)$ in \eqref{eq:zeta_1},
 \begin{align*}  \sum_{q\in \N} \tilde{c}_q(s) \rho^{q(1-s/\delta)} & =  \sum_{t=0}^{k-1}\sum_{a,b,z\in \mathcal{V}_0} \frac{A_1 \cdots A_k(z,b)}{(\rho_1 \cdots \rho_t)^{s/\delta}} (x^\Lambda_b)^s \sum_{j=0}^{p-1}c_{a,z; j} \rho^{j(1-s/\delta)}  \frac{ 1}{1-\rho^{p(1-s/\delta)}}\end{align*}
 is finite for any $s > \delta$. Moreover, the fact that $\zeta_\delta(s) $ is finite implies that, for any  $s > \delta$, 
 \begin{align*}
 {R}_{\zeta}(s)&  :=  \zeta_\delta(s) -  \sum_{q \in \N} \tilde{c}_q(s)\rho^{q(1-s/\delta)}  \\
 &= \sum_{q\in \N} {\sum_{i=p+1}^m} \tilde{Q}_i(q,s) \left(\frac{\alpha_i}{\rho^{s/\delta}}\right)^q 
\end{align*}
 is also finite.  Thus, we have 
 \begin{equation}
 \label{eq:rewritten-zeta}
 \begin{split}
 (1-\rho^{p(1-s/\delta)}) \zeta_\delta(s)&= \sum_{a,b,z\in \mathcal{V}_0}\sum_{t=0}^{k-1}  \frac{A_1 \cdots A_t(z,b)}{{(\rho_1 \cdots \rho_t)^{s/\delta} }}(x^\Lambda_b)^s  \sum_{j=0}^{p-1} c_{a,z;j} \rho^{j(1-s/\delta)}  + (1-\rho^{p(1-s/\delta)})  {R}_{\zeta}(s).
 \end{split} 
 \end{equation}

Recall that $F_\gamma\mathcal{B}_\Lambda$ is the set of $\lambda \in F\mathcal{B}_\Lambda$ such that $\lambda$ is an extension of $\gamma$, and  the formula for the measure $\mu_\delta$ is given on cylinder sets by 
\[ \mu_\delta([\gamma]) =  \lim_{s \searrow \delta} \frac{\sum_{\lambda \in F_\gamma \mathcal{B}_\Lambda} w_\delta(\lambda)^s}{\zeta_\delta(s)}=: \lim_{s \searrow \delta}  \frac{\mathcal{N}_\gamma(s)}{\zeta_\delta(s)} \]

To compute $\mu_\delta([\gamma])$, then, we will begin by rewriting
$
\mathcal{N}_\gamma (s)$
along the same lines as the expression of  $\zeta_\delta(s)$ in Equation \eqref{eq:rewritten-zeta}. This will simplify the computation of $\mu_\delta([\gamma]).$
 
Assume that $\gamma$ has $n_0=q_0k+t_0$ edges, for some $q_0, t_0 \in \N,$ $0 \leq t_0 <k$. For bookkeeping's sake, it is now useful to distinguish two cases: $t_0 =0$ and $t_0 \not=0$. We detail the case $t_0=0$ below; the case $t_0 > 0$ is similar but requires more bookkeeping.

For the case $t_0=0$, observe that any path $\eta \in F_\gamma\mathcal{B}_\Lambda $ can be uniquely realized as $\eta = \gamma \eta'$.  We will group the paths $\eta$ according to the number of edges in  $\eta'$.  Following the same arguments that led us to the formula \eqref{eq-explicit-zeta} for $\zeta_\delta(s)$, we see that 
\begin{align*}
\mathcal{N}_\gamma(s) &=\sum_{\eta\in F_\gamma\mathcal{B}_\Lambda} w_\delta(\eta)^s=\sum_{n\in \N}\sum_{\eta\in F^{q_0 k+n}_\gamma \mathcal{B}_\Lambda} w_\delta(\eta)^s \\
& = \sum_{t=0}^{k-1}\sum_{q \in \N} \sum_{z, b \in \mathcal{V}_0} A^q (s(\gamma), z) (x^\Lambda_{b})^s  \frac{A_1 \cdots A_t(z,b)}{(\rho_1 \cdots \rho_t)^{s/\delta}}\frac{1}{\rho^{(q+q_0)s/\delta} } \\
 & =\frac{1}{\rho^{q_0 s/\delta }}\sum_{t=0}^{k-1} \sum_{b,z\in\mathcal{V}_0}\frac{A_1 \cdots A_t(z,b)}{(\rho_1 \cdots \rho_t)^{s/\delta}}(x^\Lambda_b)^s\sum_{q \in \N} A^q(s(\gamma), z)   \frac{1}{\rho^{qs/\delta} }.\\
\end{align*}
Note that the only differences between the expression above and the formulation of $\zeta_\delta(s)$ given in Equation \eqref{eq-similar-num} are the initial factor of $\rho^{-q_0 s/\delta}$ and  that in the expression for $\mathcal{N}_\gamma(s)$, there is no summation on $s(\gamma)$ (this was $a \in \mathcal{V}_0$ in Equation \eqref{eq-similar-num}).
  
 Therefore, by  using the expression for $A^q(s(\gamma), z)$ given by Equation \eqref{eq-analog-oflemma-3-6-of-JS-all-cases} and going through the same 
formal derivations as we did for  $\zeta_\delta(s)$, we get

\begin{equation}\label{eq:num_1}
\begin{split}
(1-\rho^{p(1-s/\delta)}) \mathcal{N}_\gamma(s)&=  \frac{1}{\rho^{q_0 s/\delta}}\sum_{b,z\in \mathcal{V}_0}\sum_{t=0}^{k-1} (x^\Lambda_b)^s \frac{A_1 \cdots A_t(z,b)}{{(\rho_1 \cdots \rho_t)^{s/\delta} }}  \sum_{j=0}^{p-1} c_{s(\gamma),z;j} \rho^{j(1-s/\delta)}  \\
& \qquad + (1-\rho^{p(1-s/\delta)})  {R}_{\mathcal{N}}(s(\gamma),s),
\end{split}
\end{equation}
where for any $a \in \mathcal{V}_0$, we set

\begin{equation}
\label{eq-similar-Num}
\begin{split}
R_{\mathcal{N}}(a,s) 
&= \rho^{-q_0(s/\delta)} \sum_{b,z\in \mathcal{V}_0}\sum_{t=0}^{k-1} \frac{A_1 \cdots A_t(z,b) }{(\rho_1 \cdots \rho_t)^{s/\delta}} (x^\Lambda_b)^s \sum_{i=p+1}^m\sum_{q\in \N} P^{a,z}_i(q) \frac{\alpha_i^q}{\rho^{qs/\delta}}.
\end{split} 
\end{equation}
Note that, as before, for any fixed $a \in \mathcal{V}_0$ and any $s > \delta$, ${R}_{\mathcal{N}}(a,s)$ is finite.

Now, for $s > \delta$ and $a \in \mathcal{V}_0$, let 
\begin{align*}
Y_{a; 0}(s) & =\sum_{t=0}^{k-1} \sum_{z,b\in \mathcal{V}_0} \frac{A_1 \cdots A_t(z,b)}{(\rho_1 \cdots \rho_t)^{s/\delta}}  (x^\Lambda_b)^s \sum_{j=0}^{p-1} c_{a,z;j} \rho^{j(1-s/\delta)}.
\end{align*}
Observe that $Y_{a; 0}(s)$ is positive for every $a \in \mathcal V_0$, since the fact that $\Lambda$ is source-free implies that $\sum_{b \in \mathcal V_0}A_1 \cdots A_t(z, b)$ is nonzero for each $z \in \mathcal V_0,  \ 0\leq t \leq k-1 $, and Equation \eqref{eq:Aj-c-coeffs} implies that $c_{a,z;j} = A^{(j)}(a,z) \in \R_{\geq 0}$ is nonzero for at least one $z$.  If we also define
\begin{align*}
  Y_s = \sum_{a\in \mathcal{V}_0} Y_{a;0}(s) = 
       \sum_{t=0}^{k-1} \sum_{a,z,b\in \mathcal{V}_0}  \frac{A_1 \cdots A_t(z,b)}{(\rho_1 \cdots \rho_t)^{s/\delta}} (x^\Lambda_b)^s \sum_{j=0}^{p-1} c_{a,z;j} \rho^{j(1-s/\delta)},
\end{align*}
then we have $\zeta_\delta(s) =  Y_s(1-\rho^{p(1-s/\delta)})^{-1} + R_\zeta(s)$. Similarly  we have from \eqref{eq:num_1}
\[
\mathcal{N}_\gamma(s) =\frac{ Y_{s(\gamma); 0} (s)}{\rho^{q_0 s/\delta}(1-\rho^{p(1-s/\delta)})} + R_\mathcal{N}(s(\gamma); s)
\] 
 It follows that 
\begin{align*}
\mu_\delta([\gamma]) &= \lim_{s\searrow \delta} \frac{\mathcal{N}_\gamma(s)}{\zeta_\delta(s)} = \lim_{s \searrow \delta} \frac{(1-\rho^{p(1-s/\delta)}) \mathcal{N}_\gamma(s)}{(1-\rho^{p(1-s/\delta)}) \zeta_\delta(s)} \\
& = \lim_{s\searrow \delta}  \frac{\rho^{-q_0s/\delta}d_{s(\gamma);0}(s) + (1-\rho^{p(1-s/\delta)}){R}_{\mathcal{N}}(s(\gamma),s)}{Y_s + (1-\rho^{p(1-s/\delta)}){R}_{\zeta}(s)} .
\end{align*}

The next step in calculating $\mu_\delta([\gamma])$ is to compute 
\[ \lim_{s \searrow \delta} (1-\rho^{p(1-s/\delta)}){R}_{\zeta}(s) \text{ and } \lim_{s \searrow \delta} (1-\rho^{p(1-s/\delta)}){R}_{\mathcal{N}}(s(\gamma),s).\] 
For each $i$, the infinite series 
\begin{equation}
  \sum_{q \in \N} P_i^{a,z}(q) \left( \frac{\alpha_i}{\rho^{s/\delta}} \right)^q\label{eq:R-inf-series}
\end{equation}
used in the definition of 
$R_{\mathcal{N}}(a, s)$ and $R_\zeta(s)$
satisfies 
\begin{equation}
\label{eq:R-upper-bd}
\begin{split}
 \left| \sum_{q \in \N} P_i^{a,z}(q) \left( \frac{\alpha_i}{\rho^{s/\delta}} \right)^q \right| &\leq \sum_{q \in \N}|  P_i^{a,z}(q) | \left| \frac{\alpha_i}{\rho^{s/\delta}}\right|^q \\
   & < \sum_{q\in \N} | P_i^{a,z}(q) | \left( \frac{|\alpha_i|}{\rho}\right)^q 
 \end{split}
\end{equation}
whenever $s> \delta$, since $\rho > 1$.  The eigenvalue $\alpha_i$ also satisfies $|\alpha_i| < \rho$ by construction.  Since $|\alpha_i|/\rho < 1$, it now follows that the series \eqref{eq:R-inf-series} converges by the Root Test.  Moreover, the fact that the final series in \eqref{eq:R-upper-bd} converges independently of $s$ implies that both $\lim_{s\searrow \delta} R_{\mathcal{N}}(a, s)$ and $\lim_{s\searrow \delta}R_\zeta(s)$ are finite.
Since $\lim_{s\searrow \delta} (1- \rho^{p(1-s/\delta)}) = 0$, we have 
\[ \lim_{s\searrow \delta} (1-\rho^{p(1-s/\delta)}) R_\mathcal{N}(a, s) = 0 = \lim_{s \searrow \delta} (1-\rho^{p(1-s/\delta)}) R_\zeta(s).\]
All of the terms in the formulas for $Y_s$ and $Y_{a;0}(s)$ are  continuous in the limit $s \searrow \delta$; thus,
\begin{equation}
\begin{split}
 \mu_\delta([\gamma]) & = \frac{ \rho^{-q_0} Y_{s(\gamma); 0}(\delta)}{ Y_\delta}
\label{eq:mu-first-formula}
\end{split}
\end{equation}
whenever $\gamma \in F\mathcal{B}_\Lambda$ has $|\gamma| = q_0 k$.
Since the sums defining $Y_{s(\gamma);0}(\delta)$ and $Y_\delta$ are finite, $\mu_\delta([\gamma])$ is finite in this case.

If $|\gamma| = q_0 k + t_0$ for some $t_0 \not= 0$, we can consider separately the paths in $F_\gamma \mathcal{B}_\Lambda$ extending $\gamma$ which have  length less than $(q_0 + 1)k$, and those which are longer.
This perspective gives 
\begin{align*}
  \mathcal{N}_\gamma(s) & = \sum_{a \in \mathcal{V}_0} \sum_{r=0}^{k-t_0-1} \frac{A_{t_0 + 1} \cdots A_{t_0 + r}(s(\gamma), a)}{ \rho^{q_0 s/\delta} (\rho_1 \cdots \rho_{t_0 +r})^{s/\delta}} (x_a^\Lambda)^s \\
    & \qquad + \sum_{a, z, b \in \mathcal{V}_0} \sum_{t=0}^{k-1} {A_{t_0+1} \cdots A_k(s(\gamma), a) } (x^\Lambda_b)^s \frac{A_1 \cdots A_t(z,b)}{ (\rho_1 \cdots \rho_t)^{s/\delta}} \sum_{q\in \N} A^q(a,z) \rho^{-(q_0 + 1 + q) s/\delta}
\end{align*}
The first sum, being finite and continuous in the limit as $s \searrow \delta$, will vanish when we multiply $\mathcal{N}_\gamma(s)$ by $(1-\rho^{p(1-s/\delta)})$.  For the second term, the same manipulations that we performed above on the sum $\sum_{q\in \N} A^q(a,z) \rho^{-q(s/\delta)}$ tell us that if we let (for $v\in \mathcal{V}_0$)
\[  \widetilde{Y_{v;t_0}}(s) :=  \sum_{a,z,b\in \mathcal{V}_0} A_{t_0 +1} \cdots A_k(v, a) \sum_{t=0}^{k-1} \frac{A_1 \cdots A_t(z,b)}{(\rho_1 \cdots \rho_t)^{s/\delta}} (x_b^\Lambda)^s\sum_{j=0}^{p-1} c_{a,z; j} \rho^{j(1-s/\delta)} ,\]
we have
 \begin{align*}
\mu_\delta([\gamma]) &= \lim_{s\searrow \delta} \frac{\mathcal{N}_\gamma(s)}{\zeta_\delta(s)} = \lim_{s \searrow \delta} \frac{(1-\rho^{p(1-s/\delta)}) \mathcal{N}_\gamma(s)}{(1-\rho^{p(1-s/\delta)}) \zeta_\delta(s)} \\
& = \frac{\widetilde{Y_{s(\gamma);t_0}}(\delta)}{\rho^{q_0+1} {Y_\delta}} ,
\end{align*}
where for any $v \in \mathcal{V}_0$, we have 
\begin{align*}
\widetilde{Y_{v; t_0}}(\delta) &= \sum_{a,b,z \in \mathcal{V}_0} \sum_{t=0}^{k-1} \frac{A_1 \cdots A_t(z,b) }{\rho_1 \cdots \rho_t} (x_b^\Lambda)^\delta A_{t_0+1} \cdots A_k(v, a) \sum_{j=0}^{p-1} c_{a,z; j} \\ 
&= \sum_{a \in \mathcal{V}_0} A_{t_0 +1} \cdots A_k(v,a) Y_{a; 0}(\delta).
\end{align*}
Observe that
 $\mu_\delta([\gamma])$ only depends on the vertex $s(\gamma)$ and on $|\gamma| = q_0 k + t_0.$  Moreover, $\mu_\delta([\gamma])$ is a quotient of finite sums and is hence finite.  The fact that  the Dixmier trace determines a measure $\mu_\delta$ now follows from Proposition \ref{pr:dixmier-trace-measure}.
\end{proof}

We now proceed to compare the measures $\mu_\delta$ on $X_{\mathcal{B}_\Lambda}$ of Theorem \ref{thm:dixmier-measure} with the unique scale-invariant Borel measure $M$ on $\Lambda^\infty\cong X_{\mathcal{B}_\Lambda}$ identified in Proposition 8.1 of \cite{aHLRS}, which was described in Equation \eqref{eq:M-measure} above.

\begin{cor}
\label{cor:dixmier-aHLRS}
Let $\Lambda$ be a finite, strongly connected $k$-graph with adjacency matrices $A_i$ {such that  $A = A_1 \cdots A_k$ is  irreducible. }For any $\delta \in (0,1)$ such that 
Equation \eqref{eq:weight-diam} holds on the associated ultrametric Cantor set $(X_{\mathcal{B}_\Lambda}, d_{\delta}),$
 {the  measure $\mu_\delta$ induced by the Dixmier trace} agrees with the measure $M$ introduced in Proposition~8.1 of \cite{aHLRS}:  for any  $\lambda \in F\mathcal{B}_\Lambda$,
\[ 
\mu_\delta([\lambda]) = (\rho_1 \cdots \rho_t)^{-(q+1)} (\rho_{t+1} \cdots \rho_k)^{-q} x^\Lambda_{s(\lambda)} = M([\lambda]).
\]
  In particular, $\mu_\delta$ is a probability measure which is independent of $\delta$.
\end{cor}
\begin{proof}
To see that $\mu_\delta$ is a probability measure, we compute 
\begin{align*}
\mu_\delta(X_{\mathcal{B}_\Lambda}) &= \sum_{v \in \mathcal{V}_0} \mu_\delta([v]) = \sum_{v\in \mathcal{V}_0} \frac{Y_{v;0}(\delta)}{Y_\delta} =1,
\end{align*}
since $Y_s = \sum_{v \in \mathcal{V}_0} Y_{v;0}(s)$ for all $s$.  Recall that $Y_{v;0}(\delta) > 0$ for all $v \in \mathcal V_0$, and hence $Y_{\delta}$ is also finite and nonzero.

Moreover, for any path $\gamma \in F_v\mathcal{B}_\Lambda$ with $|\gamma| \geq k$,  write $\gamma = \gamma_0 \gamma'$ with $|\gamma_0| = k$.  Since $r(\gamma') \in \mathcal{V}_k = \mathcal{V}_0$, we can identify $\gamma'$ with a path in $F\mathcal{B}_\Lambda$. Then Proposition \ref{pr:delta-weight} tells us that 
\[w_\delta(\gamma) = \rho^{-1/\delta} w_\delta(\gamma').\]
Consequently, 
\begin{align*}\mu_\delta([v]) &  = \lim_{s \searrow \delta} \frac{\sum_{\gamma \in F_v \mathcal{B}_\Lambda}w_\delta(\gamma)^s}{\zeta_\delta(s)} = \lim_{s \searrow \delta} \left(  \frac{\sum_{r(\gamma) = v, |\gamma| < k} w_\delta(\gamma)^s}{\zeta_\delta(s)} + \frac{\sum_{r(\gamma) = v, |\gamma| \geq k } w_\delta(\gamma)^s}{\zeta_\delta (s)} \right)  \\
&= \lim_{s \searrow \delta} \left( \frac{\sum_{r(\gamma) = v, |\gamma| < k} w_\delta(\gamma)^s}{\zeta_\delta(s)} + \frac{\sum_{n=1}^\infty \sum_{t=0}^{k-1} \sum_{r(\gamma) = v, \, |\gamma| = nk + t} w_\delta(\gamma)^s}{\zeta_\delta(s)} \right) \\
&= \lim_{s \searrow \delta} \left( \frac{\sum_{r(\gamma) = v, |\gamma| < k} w_\delta(\gamma)^s}{\zeta_\delta(s)} + \frac{\rho^{-s/\delta} \sum_{n=0}^\infty \sum_{t=0}^{k-1} \sum_{z \in \Lambda^0}\sum_{r(\gamma') = z, \, |\gamma'| = nk + t}A(v, z) w_\delta(\gamma')^s}{\zeta_\delta(s)} \right) \\
&= \frac{1}{\rho} \sum_{z \in \Lambda^0} A(v, z) \mu_\delta([z]).
\end{align*}
The penultimate equality holds because of the formula \eqref{eq:w-delta}
for the weight $w_\delta$; the last equality holds because $\lim_{s \to \delta} \zeta_w(s) = \infty$, and hence the first sum (having a finite numerator) tends to zero as $s$ tends to $\delta$.

Thus, $(\mu_\delta([v]))_v$ is a positive eigenvector for $A$ with $\ell^1$-norm 1 and eigenvalue $\rho$, and hence must agree with $x^\Lambda$ by the irreducibility of $A$.

Recall from Equation \eqref{eq:mu-first-formula} that if $|\gamma| = q_0 k$ (equivalently, if we think of $\gamma \in \Lambda$, we have $d(\gamma) = (q_0, \ldots, q_0)$) we have 
\[ \mu_\delta([\gamma]) = \frac{Y_{s(\gamma); 0}(\delta)}{Y_\delta \rho^{q_0 }} . \]
Thus, for such $\gamma$, we have 
\[ \mu_\delta([\gamma]) = \mu_\delta([s(\gamma)]) \rho^{-q_0} = \mu_\delta([s(\gamma)]) (\rho_1 \cdots \rho_k)^{-q_0} = x^\Lambda_{s(v)} (\rho_1 \cdots \rho_k)^{-q_0}.\]
Comparing this formula with Equation \eqref{eq:M-measure} tells us that whenever $|\gamma | = q_0 k$,
\[\mu_\delta([\gamma]) = M([\gamma]).\]
Since $\mu_\delta$ agrees with $M$ on the square cylinder sets $[\lambda]$ with $d(\lambda) = (q_0, \ldots, q_0)$, and we know from the proof of Lemma 4.1 of \cite{FGKP} that
these sets generate the Borel $\sigma$-algebra of $\Lambda^\infty$, $\mu_\delta$ must agree with $M$ on all of $\Lambda^\infty.$ 
\end{proof}

\subsection{Hausdorff measure}
\label{sec:hausdorff}
In this section, we show that the same hypotheses we needed to guarantee the existence of the Dixmier trace also imply that the measure induced by the Dixmier trace agrees with the Hausdorff measure of the ultrametric Cantor set $(\Lambda^\infty, d_\delta)$.  A similar conclusion was also obtained, under different hypotheses, in Theorem 3.8 of \cite{FGKP}.
\begin{defn}\cite[Definition 16]{rogers}
\label{def:hausdorff-measure}
Let $(X, d)$ be a metric space and fix $s \in \R_{\geq 0}$.
The \emph{Hausdorff measure} of dimension $s$ of a  compact subset $Z$ of $X$ is 
\[ H^s(Z) = \lim_{\epsilon \to 0}\,  \inf \left\{ \sum_{U_i \in F} (\text{diam } U_i)^s: |F| < \infty, \ \cup_i U_i = Z, \ \text{diam } U_i < \epsilon \text{ for all } i\right\}.\]
It is standard  to show that $H^s(Z)$ is a decreasing function of $s$, and that there is a unique $s \in \R$ such that $H^t(X) = \infty$ for all $t > s$ and that $H^t(X) = 0$ for all $t < s$.  This value of $s$ is called the \emph{Hausdorff dimension} of $X$.
\end{defn}

\begin{thm}
\label{thm:Hausd-meas-nontriv}Let $\Lambda$ be a finite, strongly connected $k$-graph  and fix $\delta \in (0,1)$.     If Equation \eqref{eq:weight-diam} holds for the ultrametric Cantor set $(X_{\mathcal B_\Lambda}, d_{\delta})$, we have $0 < H^\delta(X_{\mathcal B_\Lambda}) < \infty$.
\end{thm}

\begin{proof}
 First, consider a cover $\{ U_i\}_i$ of $X_{\mathcal B_\Lambda}$  and the associated sum
 \[
  \sum_{i \in F} (\text{diam } U_i)^\delta.
 \]
 Given $U_i$ in this covering, let $x = x_1 x_2 \cdots \in U_i$, and define $B_{x,i} = B(x,\tfrac 1 2 \text{diam } U_i)$ (the open ball of center $x$ and diameter $\text{diam }U_i$). {Invoking Equation \eqref{eq:weight-diam} and the definition of the weight $w_\delta$, there is a smallest $ n\in \N$ such that 
 \[w_\delta(x_1 \cdots x_n) \leq \text{diam}\, U_i;\]
 consequently,   $B_{x, i} = [x_1 \cdots x_n]$.} 
 Moreover, the ultrametric property implies that $B_{x, i} = U_i$ for all $x \in U_i$.
 Therefore, without loss of generality, the  infimum defining $H^\delta$ can be taken over coverings of the form $\{[\gamma_i]\}_{i \in F}$.

 Assume, for a contradiction, that $H^\delta(X_{\mathcal B_\Lambda}) = 0$. This means that for all $n \in \N$, there is $\epsilon_n > 0$ such that
 \[
  \inf_{|F|<\infty} \left\{ \sum_{i \in F} (\text{diam } [\gamma_i])^s: \cup_i [\gamma_i] = X_{\mathcal B_\Lambda}, \ \text{diam } [\gamma_i] < \epsilon_n \text{ for all } i\right\} < \frac 1 {2n}.
 \]
 By definition of the weight (and thus, by Equation \eqref{eq:weight-diam}, of the diameter of the cylinder sets $[\gamma_i])$, it follows that for all $n \in \N$ there is $m_n > 0$ such that
 \[
  \inf_{|F|<\infty} \left\{ \sum_{i \in F} (\text{diam } [\gamma_i])^s: \cup_i [\gamma_i] = X_{\mathcal B_\Lambda}, \ |\gamma_i| > m_n \text{ for all } i\right\} < \frac 1 {2n}.
 \]
 (note that $m_n$ tends to infinity as $n$ tends to infinity.)
 By definition of the infimum, for all $n$, there is  a covering $\{\gamma_i^{(n)}\}$ with $|\gamma_i| >  {m_n}$  such that
 \[
  \sum_{i \in F} (\text{diam } [\gamma^{(n)}_i])^\delta < \frac 1 n.
 \]
 Write $M_n$ for the length of the longest of the paths $\gamma_i^{(n)}$.
 For a given $n$, there is always a $k>n$ such that $m_k > M_n$. It follows that the shortest path among the $\{\gamma_i^{(k)}\}_i$ is longer that the longest path among the $\{\gamma_i^{(n)}\}_i$.
 In other words, the cover $\mathcal C_k := \{\gamma_i^{(k)}\}_i$ is a refinement of the cover $\mathcal C_n := \{\gamma_i^{(n)}\}_i$.
 In conclusion, up to extracting a subsequence, there is a sequence of refining coverings of $X$ whose diameter tends to zero, such that the elements of each covering are of the form $[\gamma]$ for $\gamma$ a finite path, and such that
 \[
  \lim_{n \rightarrow \infty} \sum_{[\gamma] \in \mathcal C_n} (\text{diam } [\gamma])^\delta = 0.
 \]
  
Now, let $(\mathcal C_n)_{n\in \N}$ be a sequence of finite open coverings of $X_{\mathcal B_\Lambda}$ such that, for all $n \in \N$,
 \[
   \sum_{[\gamma] \in \mathcal C_n} (\text{diam } [\gamma])^\delta < \frac 1 n.
 \]
 
 Each $\mathcal C_n$ is made of clopen sets $\mathcal C_n = \{[\gamma_i^{(n)}] \}_{i \in F_n}$ for $F_n$ a finite set, and we let $m_n$ and $M_n$ be respectively the length of the shortest and longest path among the $\{\gamma_i^{(n)}\}_{i \in F_n}$.
 
 For each $n$, we now consider the set of finite clopen covers  $\mathcal C$ of $X$ which satisfy the following: 
 \begin{enumerate}
 \item The clopen sets in $\mathcal C$ are of the form $[\eta_i^{(n)}]$ ($i \in F'_n$)
 \item The cover  satisfies $\sum_{i \in F'_n} (\text{diam } [\eta^{(n)}_i])^\delta < \frac 1 n$.

 \item For all $i \in F'_n$, $|\eta_i^{(n)}| \leq M_n$.
 \end{enumerate}
 There is at least one covering satisfying these conditions, {namely} $\mathcal C_n$, and there are only finitely many since there are finitely many paths whose length are bounded by $m_n$ and $M_n$.
 Among all these coverings, we let $\mathcal C'_n$ be {one such that $m'_n:=\min_i |\eta^{(n)}_i|$ is maximal.} 
 Since $\mathcal C_n$ satisfies conditions 1--3 above, we have $m'_n \geq m_n$.
 Therefore, when we do this construction for all $n$, $\mathcal C'_{n+1}$ is still a refinement of $\mathcal C'_n$.
 
 We further alter the coverings $\mathcal C'_n$ to show that they can be self-similar to some extent.
 For a given $n$, let $m'_n$ be the length of the shortest path appearing in $\mathcal C'_n$.
 For each $\lambda$ of length ${m'_n}$, we consider:
 \[
  \Sigma_{\lambda} = \sum_{i \in F'_n,\ [\eta_i^{(n)}] \subseteq [\lambda]} (\diam [\eta_i^{(n)}])^\delta
 \]
 Let us fix $v \in \mathcal V_n$ for the moment, and let $\lambda_v$ be a path for which the sum above is minimal, more precisely:
 \begin{equation}\label{eq:hausdorff-sum-minimal}
  \Sigma_{\lambda_v} = \min \{ \Sigma_\lambda \ : \ \lambda \in F^{m'_n} (\mathcal B), \ s(\lambda) = v \}.
 \end{equation}
 We name explicitly the partition of $[\lambda_v]$ induced by $\mathcal C'_n$: the elements of $\mathcal C'_n$ which are subsets of $[\lambda_v]$ are:
 \[
  [\lambda_v \eta'_1], \ldots, [\lambda_v \eta'_l].
 \]
 We define a new partition of $X$ as follows:
 for all finite paths $\lambda$ of length $m'_n$ and source $v$, we replace the elements of $\mathcal C'_n$ which are subsets of $[\lambda]$ by
 \[
  [\lambda \eta'_1], \ldots, [\lambda \eta'_l].
 \]
  By Equation~\eqref{eq:hausdorff-sum-minimal} and the formula given in Proposition \ref{pr:kgraph-bratteli-inf-path-spaces} for the weight $w_\delta$, for all $\lambda$ with $|\lambda| = m'_n$ and $s(\lambda) = v$, 
 \begin{equation}\label{eq:hausdorff-sum-reduced}
  \sum_{i=1}^k (\diam([{\lambda \eta'_i}])^\delta = \Sigma_{\lambda_v} \leq \Sigma_\lambda.
 \end{equation}
 We perform these substitutions for all $v$, and obtain a new partition of $X$, which we expectedly call $\mathcal C''_n$.
 By construction, we have
 \[
  \sum_{[\gamma] \in \mathcal C'_n} (\text{diam } [\gamma])^\delta \leq  \sum_{[\gamma] \in \mathcal C''_n} (\text{diam } [\gamma])^\delta .
 \]
 
 We can make this change for all $n$, and obtain a new sequence of partitions $(\mathcal C''_n)_n$ which satisfy properties 1--3 above, and in addition satisfies the following self-similarity condition:
 \begin{itemize}
  \item if $[\gamma] = [\lambda \eta]$ is an element of the partition $\mathcal C''_n$, with $|\lambda| = m'_n$, then for all paths $\lambda'$ of length $m'_n$ with $s(\lambda) = s(\lambda')$, then $[\lambda' \eta] \in \mathcal C''_n$.
 \end{itemize}
 In addition, since $\mathcal C'_n$ was the partition satisfying 1--3 for which $m'_n$ was maximal, we deduce that the shortest path $\lambda$ for which $[\lambda] \in \mathcal C''_n$ has length at most $m'_n$. By construction, its length cannot be lower than $m'_n$. Therefore,
 \begin{itemize}
 \item there is a path $\lambda$ of length $m'_n$ such that $[\lambda] \in \mathcal C''_n$.
 \end{itemize}
 
 We now have all the ingredients to conclude the proof.
 Fix $\lambda$ a path such that $[\lambda] \in \mathcal C''_n$ and $|\lambda| = m'_n$. Let $v$ be the source of $\lambda$. Then, by the self-similarity property above, for all $\lambda'$ such that $s(\lambda')=v$ {and $|\lambda'|=m'_n$}, we have $[\lambda'] \in \mathcal C''_n$.
 We can therefore compute a lower bound of $ \left\{ \sum_{[\gamma] \in \mathcal C''_n} (\text{diam}\, [\gamma])^\delta \right\}_{n \in \N}$:
 \[ \begin{split}
\sum_{[\gamma] \in \mathcal C''_n} (\text{diam}\, [\gamma])^\delta  & \geq \sum_{\lambda \in F^{m'_n}(\mathcal B), \ s(\lambda) = v} (\diam [\lambda])^\delta = \sum_{\lambda \in F^{m'_n}(\mathcal B), \ s(\lambda) = v} w_\delta (\lambda)^\delta \\
      & = \sum_{z \in \mathcal V_0} \rho^{-q} (\rho_1 \cdots \rho_t)^{-1} ( x^\Lambda_v)^\delta(A^q A_1 \cdots A_t)(z,v), \qquad \quad m'_n = qk+t.
 \end{split}\]
{
 Assume for the moment that $t = 0$; then $w_\delta(\lambda)^\delta = \rho^{-q} (x^\Lambda_v)^\delta$.  Recall from Equation~\eqref{eq-analog-oflemma-3-6-of-JS-all-cases} that 
 \[ A^q(z,v) = c_{z,v;q} \rho^q +   \sum_{i=p+1}^m P_i^{z,v}(q) \alpha_i ^q, \]
 where for each $i$, $|\alpha_i| < |\rho|$ and $P_i^{z,v}$ is a polynomial in $q$.  Thus, since $A$ and $x^\Lambda$ have non-negative entries, 
\[ 
  w_\delta(\lambda)^\delta A^q(z,v) = |w_\delta(\lambda)^\delta A^q(z,v)| = (x^\Lambda_v)^\delta \left|c_{z,v;q} + \sum_{i=p+1}^m P_i^{z,v}(q) \left( \frac{\alpha_i}{\rho}\right)^q \right| .
\]
{We claim that for all $q\in \N$, $ w_\delta(\lambda)^\delta A^q(z,v)$ is uniformly bounded away from zero.
To see this, we use the characterization of a strongly connected $k$-graph from~\cite[Section 3]{aHLRS}: the family of matrices $\{A_i\}_{1 \leq i \leq k}$ is an irreducible family of matrices. This characterization implies immediately that the family $\{A_i^T\}_{1 \leq i \leq k}$ is also an irreducible family.
Proposition 3.1 of \cite{aHLRS} then tells us that $\{ A_i^T\}_{1 \leq i \leq k}$ has a common positive {right} eigenvector $y \in (0, \infty)^{\Lambda^0}$. 
Note that $y$ is also a positive eigenvector for $A^T$ with eigenvalue $\rho$, since the spectral radius of a matrix equals that of its transpose.

Now, the same arguments used 
to establish in Equation \eqref{eq:Aj-positive} that for all $j \in \{ 0, \ldots, p-1\}$ and $a\in \mathcal V_0$, there exists $b\in \mathcal V_0$ such that $A^{(j)}(a,b) > 0$ will also show that for all such $j$ and $b$, there exists $a$ such that $A^{(j)}(a,b) = (A^{(j)})^T(b,a)$ is positive.}  In other words, for each fixed $q$ and $v$, $c_{z, v; q}$
must be nonzero for at least one $z \in \mathcal V_0$. 
 Since $P_i^{z,v}$ is a polynomial in $q$ and $c_{z,v;q}$ is periodic in $q$ (of period $p$), we see that \[ \lim_{q \to \infty} w_\delta(\lambda)^\delta A^q(z,v)\] is bounded away from zero, for all $z, v \in \mathcal V_0$.

If $1 \leq t \leq k-1$, the fact that the adjacency matrices $A_i$ commute implies that
\[ A^q A_1 \cdots A_t(z, v) = \sum_{a \in \mathcal V_0} A_1 \cdots A_t(z,a) A^q(a, v) = \sum_{a \in \mathcal{V}_0} A_1 \cdots A_t(z,a) \left(c_{a,v;q} \rho^q +   \sum_{i=p+1}^m P_i^{a,v}(q) \alpha_i ^q\right) \]
and hence (since  $w_\delta(\lambda) = \frac{x^\Lambda_{s(\lambda)}}{(\rho^q \rho_1 \cdots \rho_t)^{1/\delta}}$)
\[\sum_{z\in \mathcal V_0} w_\delta(\lambda)^\delta (A^q A_1 \cdots A_t)(z, v) = \frac{(x^\Lambda_v)^\delta}{(\rho_1 \cdots \rho_t)} \sum_{a, z \in \mathcal{V}_0} A_1 \cdots A_t(z,a) \left(c_{a,v;q}  +   \sum_{i=p+1}^m P_i^{a,v}(q)\left(\frac{\alpha_i}{\rho}\right)^q\right).\]
Since $\Lambda$ is strongly connected, by Lemma 2.1 of \cite{aHLRS}, for all $a \in \mathcal V_0$ and all $t$, $A_1 \cdots A_t(z, a)$ must be nonzero for at least one $z$.  This, combined with the fact that for each $q$ and $v$, there exists $a$ such that $\{c_{a, v; q}\}  > 0$, implies that the sum above is  bounded away from zero, uniformly in $q$ and $t$.  In other words, no matter what value of $m_n' = qk +t$ we have,
\[ \lim_{n\to \infty} \sum_{[\gamma] \in \mathcal C''_n} (\text{diam}\, [\gamma])^\delta > 0,\]
contradicting the hypothesis that $X_{\mathcal B_\Lambda}$ has zero Hausdorff measure.
 }

 To show that $H^s(X_{\mathcal B_\Lambda})$ is not infinite, simply consider for all $q \in \N$ the cover of $X_{\mathcal B_\Lambda}$ given by $\{[\gamma] \ : \ \gamma \in F^{qk}(\mathcal B) \}$.
 We have
 \[
  H^\delta (X_{\mathcal B_\Lambda}) \leq \lim_{q \rightarrow \infty} \sum_{\gamma \in F^{qk}(\mathcal B)} (\diam[\gamma])^\delta = \lim_{q \to \infty} \sum_{\gamma \in F^{qk}(\mathcal B)} w_\delta(\gamma)^\delta.
 \]
 Now, if $|\gamma| = qk$, then $w_\delta(\gamma)^\delta = O(\rho^{-q})$. In addition, the number of paths in $F^{qk}(\mathcal B)$ is equal to $\sum_ {v,w} A^q(v,w)$, which by Equation~\eqref{eq-analog-oflemma-3-6-of-JS-all-cases} is $O(1)$-dominated  by $ \rho^{q}$. Therefore $
  H^\delta(X_{\mathcal B_\Lambda})$ is finite.
\end{proof}

\begin{cor}
The Hausdorff dimension of $(X_{\mathcal B_\Lambda}, d_\delta)$ is $\delta$.
\label{cor:Hausdorff-dim}
\end{cor}
Note that this Corollary can also be proved from  Lemma~2.11 and Theorem~2.12 in~\cite{JS-embedding}.

\begin{thm}
\label{thm:anHuef-Hausdorff-meas}
Let $\Lambda$ be a finite, strongly connected $k$-graph. If $(X_{\mathcal B_\Lambda}, d_{w_\delta})$ is a Cantor set of Hausdorff dimension $\delta$ for which Equation \eqref{eq:weight-diam} holds,
its normalized Hausdorff measure $\frac{1}{H^\delta(X_{\mathcal B_\Lambda})} H^\delta$ agrees with the measure $M$  of Equation \eqref{eq:M-measure}.
\end{thm}
\begin{proof}
We will show that the Hausdorff measure $H^\delta$ (which is finite and nonzero by Theorem \ref{thm:Hausd-meas-nontriv} above) has the same scaling property as the measure $M$.  Since we know from Proposition 8.1 of \cite{aHLRS} that $M$ is the unique Borel probability measure on $\Lambda^\infty$ with 
\[M([\lambda]) = \rho(\Lambda)^{-d(\lambda)} M([s(\lambda)])\]
for all cylinder sets $[\lambda]$,
showing that $H^\delta([\lambda]) = \rho(\Lambda)^{-d(\lambda)} H^\delta([s(\lambda)])$ for all $\lambda \in \Lambda$ will complete the proof.

Since the cylinder sets $[\lambda] \subseteq \Lambda^\infty$ are compact open sets, standard arguments about Hausdorff measure show that 
\[ H^\delta ([\lambda])  = \lim_{\epsilon \to 0}\,  \inf\left\{ \sum_{\eta_i \in F} (\text{diam } [\lambda \eta_i])^\delta: \#(F) < \infty, [\lambda] = \cup_i [\lambda \eta_i] , \ \text{diam } [\lambda \eta_i] < \epsilon \text{ for all } i\right\}.\]
Moreover, recall from  Proposition \ref{pr:delta-weight} and Equation \eqref{eq:diam} that for the metric $d_\delta$, 
\begin{align*}
\text{diam } [\lambda \eta_i] &= \begin{cases}
\rho(\Lambda)^{-d(\gamma_i)/\delta} x^\Lambda_{s(\gamma_i)}, & \gamma_i \text{ is the shortest extension of } \lambda \eta_i \\
\quad & \quad \text{ such that  } |\gamma_i| = qk + t \text{ and } \#(s(\gamma_i) \Lambda^{e_{t+1}} )\not= 1\\
0, & \#[\lambda \eta_i] = 1
\end{cases} 
\end{align*}

Thus, $\{ \lambda \eta_i \}_i$ is a cover of $[\lambda]$ of maximum radius $\epsilon$, iff $\{ \eta_i \}_i$ is a cover of $[s(\lambda)]$ of maximum radius $\epsilon \rho(\Lambda)^{d(\lambda)/\delta}$.    Moreover, 
\[( \text{diam }[\lambda \eta_i])^\delta 
= \rho(\Lambda)^{-d(\lambda)} (\text{diam }[\eta_i])^\delta.\]
It follows that, as desired, 
\begin{align*}
H^\delta ([\lambda]) &= \lim_{ \epsilon \to 0} \inf \left\{ \sum_{\eta_i \in F} (\text{diam } [ \lambda \eta_i])^\delta: \#(F) < \infty,  [s(\lambda)] = \cup_i  [\eta_i] , \ \text{diam } [\eta_i] < \rho(\Lambda)^{d(\lambda) }\epsilon \text{ for all } i\right\}\\
&= \rho(\Lambda)^{-d(\lambda)} H^\delta([s(\lambda)]). \qedhere
\end{align*}
\end{proof}
\begin{rmk} 
Note that Theorem \ref{thm:anHuef-Hausdorff-meas} does not contradict the uniqueness of the nontrivial Hausdorff measure.  Indeed, it proves a stronger uniqueness result: Not only do all of our ultrametric Cantor sets $(X_{\mathcal B_\Lambda}, d_\delta)$ possess a unique nontrivial Hausdorff measure $H^\delta$, but (up to scaling)  this measure is independent of $\delta$. 
 In other words, the Hausdorff measure of $(X_{\mathcal B_\Lambda}, d_\delta)$ depends almost entirely on the topological structure of the infinite path space, rather than on the choice of geometry via the ultrametric $d_\delta$.
\end{rmk}

\section{Eigenvectors of Laplace-Beltrami operators and the wavelets of \cite{FGKP}}
\label{sec-wavelets-as-eigenfunctions}
In this section, we investigate the relation between the decomposition of $L^2(\Lambda^\infty, \mu_\delta)$ via the eigenspaces of the  Laplace-Beltrami operators $\Delta_s$ associated to the {spectral triples of Section \ref{sec:zeta-regular} for the }ultrametric Cantor set $(X_{\mathcal{B}_\Lambda}, d_{w_\delta})$ of Corollary~\ref{cor:ultrametric-Cantor}, and the wavelet decomposition of $L^2(\Lambda^\infty, M)$ given in Theorem 4.2 of \cite{FGKP}.
Theorem \ref{thm:JS-wavelets} below shows  that when  the measure $\mu_\delta$ induced by the Dixmier trace agrees with the measure $M$, the eigenspaces of the Laplace-Beltrami operators refine the wavelet decomposition of \cite{FGKP}.  In order to state and prove this Theorem, we first review the two orthogonal decompositions mentioned above.

According to Section~8.3 of \cite{pearson-bellissard} and Section~4 of \cite{julien-savinien}, for each $s\in \R$ the even spectral triple associated to the weighted stationary $k$-Bratteli diagram $(\mathcal{B}_\Lambda,w_\delta)$ of Corollary~\ref{cor:ultrametric-Cantor} gives rise to a Laplace-Beltrami operator $\Delta_s$ on $L^2(X_{\mathcal{B}_\Lambda}, \mu)$  as follows:
\[
\langle f, \Delta_s(g)\rangle =Q_s(f,g):=\frac{1}{2}\int_{\Upsilon}\text{Tr}(\vert D\vert^{-s}[D,\pi_\tau(f)]^\ast [D,\pi_\tau(g)]\, d\nu(\tau),
\]
where $\Upsilon$ is the set of choice functions and $\nu$ is the measure on $\Upsilon$ induced from the Dixmier measure $\mu_\delta$ on the infinite path space. Thanks to Section 8.1 of \cite{pearson-bellissard}, we know that $Q_s$ is a closable Dirichlet form for all $s\in \R$ and it has a dense domain that is generated by a set of characteristic functions on cylinder sets of $X_{\mathcal{B}_\Lambda}$. Also, by applying the work of \cite{pearson-bellissard} and \cite{julien-savinien} to our weighted stationary $k$-Bratteli diagram $\mathcal{B}_\Lambda$, we can obtain the explicit formula for $\Delta_s$ on characteristic functions as follows.

For a finite path $\eta = (\eta_i)_{i=1}^{|\eta|}$ {(where each  $\eta_i$ is an edge)} in $\mathcal{B}_\Lambda$, we write  $\chi_{[\eta]}$ for the characteristic function of the set $[\eta] \subseteq  X_{\mathcal{B}_\Lambda}$ of infinite paths of $\mathcal{B}_\Lambda$ whose initial segment is $\eta$, and $\eta(0, i)$ for $\eta_1 \cdots \eta_i$. We denote by $\eta(0, 0)$ the vertex $r(\eta)$. Also, for $\gamma \in F\mathcal{B}_\Lambda$, we set
 \[
 \frac{1}{F_\gamma} = \sum_{(e, e') \in \text{ext}_1 (\gamma)} \mu([\gamma e]) \mu([\gamma e']), 
 \]
  where $\text{ext}_1(\gamma)$ is the set of pairs $(e, e')$ of edges in $\mathcal{B}_\Lambda$ with $e \not= e'$ and $r(e) = r(e') = s(\gamma)$.
  
  {From Remark \ref{rmk:weight-diam}, we know that if Equation \eqref{eq:weight-diam} holds for the weighted stationary $k$-Bratteli diagram $(\mathcal B_\Lambda, w_\delta)$ associated to a higher-rank graph $\Lambda$, then $\text{ext}_1(\gamma)$ is nonempty for all $\gamma \in F\mathcal B_\Lambda$.  Since we need to invoke Equation \eqref{eq:weight-diam} in order to guarantee that the measures $\mu_\delta$ and $M$  agree on $X_{\mathcal B_\Lambda} \cong \Lambda^\infty$, we will also assume without loss of generality that $\text{ext}_1(\gamma) $ is always nonempty; equivalently, that $F_\gamma < \infty$.}
Then{, as in Section 4 of \cite{julien-savinien},} for each $s\in \R$, we have

\begin{align*}\Delta_s(\chi_{[\eta]}) &  = -\sum_{i=0}^{|\eta|-1} 2 F_{\eta(0,i)} w(\eta(0,i))^{s-2}\left(\mu([\eta(0,i)] \backslash [\eta(0, i+1)] ) \chi_{[\eta]} \right. \\
& \qquad \quad \left. - \mu([\eta]) \chi_{[\eta(0,i)]\backslash [\eta(0, i+1)]} \right).\end{align*}

Moreover, Theorem 4.3 of \cite{julien-savinien} tells us that each finite path $\gamma$ in $\mathcal{B}_\Lambda$ determines an eigenspace $E_\gamma$ for $\Delta_s$:
\begin{equation}\label{eq:eigenspace_E_gamma}
 E_\gamma = \text{span} \left\{ \frac{1}{\mu[\gamma e]} \chi_{[\gamma e]} - \frac{1}{\mu[\gamma e']} \chi_{[\gamma e']}: (e, e') \in \text{ext}_1(\gamma) \right\}.\end{equation}
The other eigenspaces of $\Delta_s$ are $E_{-1} = \text{span} \{\chi_{\Lambda^\infty} \}$ and 
\[E_0 = \text{span}\left\{ \frac{1}{\mu([v])} \chi_{[v]} -\frac{1}{\mu([v'])} \chi_{[v']}: v \not= v' \in \mathcal{V}_0 \right\}.\]
Observe that if $\gamma \not= \eta$, then $E_\gamma \perp E_\eta$.  
Also, the eigenspaces of $\Delta_s$ are independent of $s$, although the eigenvalues (as described in \cite{pearson-bellissard, julien-savinien}) depend on the choice of $s \in \R$.

We now review the construction of the wavelet decomposition (for a finite, strongly connected $k$-graph $\Lambda$)  of $L^2(\Lambda^\infty, M)$ which was introduced in Section 4 of \cite{FGKP}. 

For each $n \in \N$, write 
\[\mathscr{V}_n = \text{span} \{ \chi_{[\lambda]}: d(\lambda) = (n, \ldots, n)\},\]
and define 
\[\mathcal{W}_n = \mathscr{V}_{n+1} \cap \mathscr{V}_{n}^\perp.\] 
We know from Lemma~4.1 of \cite{FGKP} that 
\[ \{ \chi_{[\lambda]}: d(\lambda) = (n, \ldots, n) \text{ for some } n \in \N\}\]
densely spans $L^2(\Lambda^\infty, M)$.  Consequently, 
\begin{equation}
L^2(\Lambda^\infty, M)= \mathscr{V}_0 \oplus \bigoplus_{n\in \N} \mathcal{W}_n;
\label{eq:wavelet-decomp}
\end{equation}
Proposition \ref{prop:wavelets-reformulated} below shows that this orthogonal decomposition  is precisely the wavelet decomposition of Theorem 4.2 of \cite{FGKP}.

Both for the proof of Proposition \ref{prop:wavelets-reformulated} and our main result, Theorem \ref{thm:JS-wavelets}, it will be convenient to work with a specific basis for $\mathcal{W}_0$.  For each vertex $v$ in $\Lambda$, let 
\[ D_v =  v\Lambda^{(1,\ldots, 1)} .\] 
One can show (cf.~\cite{aHLRS} Lemma 2.1(a)) that $D_v$ is always nonempty.

Enumerate the elements of $D_v$ as $D_v = \{ \lambda_0, \ldots, \lambda_{\#(D_v) -1}\}.$
Observe that if $D_v = \{ \lambda\}$ is a 1-element set, then $[v] = [\lambda]$.  If $\#(D_v) > 1$, that is,  the cardinality of $D_v$ is greater than one, 
then for each  $1 \leq i \leq \#(D_v) -1$, we define 
\begin{equation}\label{eq:f_iv}
f^{i,v} = \frac{1}{M[\lambda_0]} \chi_{[\lambda_0]} - \frac{1}{M[\lambda_i]} \chi_{[\lambda_i]}.
\end{equation}
One easily checks that in $L^2(\Lambda^\infty, M)$, $\langle f^{i,v} , \chi_{[w]} \rangle = 0$ for all $i$ and all vertices $v, w$, and that \[ \{ f^{i, v}: v \in \Lambda^0, 1 \leq i\leq \#(D_v)-1\}\]
is a linearly independent set in $\mathcal{W}_0 =\mathscr{V}_1\cap \mathscr{V}_0^\perp  \subseteq L^2(\Lambda^\infty, M)$. 

To see that $\{ f^{i,v}\}_{i,v}$ is a basis for $\mathcal W_0$, we will show that $\dim \mathcal{W}_0 = \dim \text{span}\, \{f^{i,v}\}_{i,v}.$  First, note that for any $n \in \N$, the set $\{ \chi_{[\lambda]}: \lambda \in \Lambda^{(n, \ldots, n)}\}$
is a basis for $\mathscr{V}_n$.  Thus, \[ \dim \mathcal{W}_0 = \dim \mathscr{V}_1 - \dim \mathscr{V}_0 = \#(\Lambda^{(1, \ldots, 1)}) - \#(\Lambda^0).\]
On the other hand, 
	\[ \# \Big( \{ f^{i,v}: v \in \Lambda^0, 1 \leq i\leq \#(D_v)-1\}\Big) = \sum_{v \in \Lambda^0} (\#(D_v) -1 ) = \#(\Lambda^{(1, \ldots, 1)}) - \#(\Lambda^0),\]
so $\mathcal{W}_0 = \text{span} \, \{ f^{i,v}\}_{i,v}$ as claimed.

For each $\lambda \in \Lambda$, define the operator $S_\lambda \in B(L^2(\Lambda^\infty, M))$ by 

\begin{equation}\label{eq:S-lambda-cylinder}
S_\lambda \chi_{[\eta]} = \begin{cases}
\rho(\Lambda)^{d(\lambda)/2} \chi_{[\lambda \eta]}, & s(\lambda) = r(\eta) \\
0, & \text{ else.}
\end{cases}
\end{equation}

We can think of these operators as combined ``scaling and translation'' operators, since they change both the size and the range of a cylinder set $[\eta]$, and are intimately tied to the geometry of the $k$-graph $\Lambda$.

The operators $S_\lambda$ were introduced in Theorem 3.5 of \cite{FGKP}, and in fact they give rise to a representation of the $C^*$-algebra $C^*(\Lambda)$ of the $k$-graph $\Lambda$.  We discuss this in more detail in Section \ref{sec-Consani-Marcolli-spectral-triples-for-k-graphs} below; see Equations \eqref{eq:S_lambda} and \eqref{eq:S-lambda-star}.

The following Proposition reconciles our definition of $\mathcal{W}_n := \mathscr{V}_{n+1} \cap \mathscr{V}_n^\perp$ with the wavelet subspaces which were denoted $\mathcal{W}_{n, \Lambda}$ in  
 Theorem 4.2 of \cite{FGKP}.  By showing that each $\mathcal{W}_n$ can be obtained from $\mathcal{W}_0$ via certain ``scaling and translation'' operators $S_\lambda$, this justifies describing the orthogonal decomposition of $L^2(\Lambda^\infty, M)$ given in \eqref{eq:wavelet-decomp} as a wavelet decomposition.

\begin{prop}\label{prop:S_n}
For any $n \in \N$,
the set  
 	\[S_n = \{ S_\lambda f^{i, s(\lambda)}: d(\lambda) = (n, \ldots, n), 1 \leq i \leq \#(D_{s(\lambda)}) - 1\}\] 
 is a basis for $\mathcal{W}_n=\mathscr{V}_{n+1}\cap \mathscr{V}_n^{\perp}$.
\label{prop:wavelets-reformulated}
\end{prop}
\begin{proof}
The formula \eqref{eq:S-lambda-cylinder} shows that if $d(\lambda) = (n, \ldots, n),$ then $S_\lambda f^{i,s(\lambda)}$ is a linear combination of characteristic functions of cylinder sets of degree $(n+1, \ldots, n+1)$.  Thus, to see that $S_\lambda f^{i, s(\lambda)} \in \mathcal{W}_n$ for each such   $\lambda$ and each $1 \leq i \leq \#(D_{s(\lambda)}) -1$, we must check that $\langle S_\lambda f^{i, s(\lambda)}, \chi_{[\eta]} \rangle = 0$ whenever $d(\eta) = (n, \ldots, n)$.  We compute:
\begin{align*}
\frac{1}{\rho(\Lambda)^{d(\lambda)/2}} \langle S_\lambda f^{i, s(\lambda)}, \chi_{[\eta]} \rangle &= \frac{1}{M[\lambda_0]} \int_{\Lambda^\infty} \chi_{[\eta]} \chi_{[\lambda \lambda_0]} \, dM - \frac{1}{M[\lambda_i]}\int_{\Lambda^\infty} \chi_{[\eta]} \chi_{[\lambda \lambda_i]} \, dM \\
&= \begin{cases}
0, & \eta \not= \lambda \\
\frac{M[\lambda \lambda_0]}{M[\lambda_0]} - \frac{M[\lambda \lambda_i]}{M[\lambda_i]}, & \lambda = \eta.
\end{cases}
\end{align*}

Using the formula for $M$ given in \eqref{eq:kgraph_Measure}, we see that

\[ \frac{M[\lambda \lambda_0]}{M[\lambda_0]} - \frac{M[\lambda \lambda_i]}{M[\lambda_i]} = \rho(\Lambda)^{-d(\lambda)} - \rho(\Lambda)^{-d(\lambda)} = 0.\]
In other words, $ \langle S_\lambda f^{i, s(\lambda)}, \chi_{[\eta]}\rangle_M = 0$ always, so $S_\lambda f^{i, s(\lambda)} \perp \mathscr{V}_n$, and hence every element $S_\lambda f^{i, s(\lambda)} \in \mathcal{W}_n$.

Moreover, $S_n$ is easily seen to be a linearly independent set, since if $d(\lambda) = d(\lambda') = (n, \ldots, n)$ and $d(\lambda_i) = d(\lambda_i') = (1, \ldots, 1)$, 
\[ [\lambda \lambda_i] \cap [\lambda' \lambda_i'] = \delta_{\lambda, \lambda'} \delta_{\lambda_i,\lambda_i'} [\lambda \lambda_i].\]
Since  $\dim \mathcal{W}_n = \dim \mathscr{V}_{n+1} - \dim \mathscr{V}_n  = \#(\Lambda^{(n+1, \ldots, n+1)}) - \#(\Lambda^{(n, \ldots, n)}) $
and 
\[\#(S_n) = \sum_{\lambda \in \Lambda^{(n,\ldots, n)}} (\#(D_{s(\lambda)}) -1) = \#(\Lambda^{(n+1, \ldots, n+1)}) - \#(\Lambda^{(n, \ldots, n)})\]
we  have $\mathcal{W}_n = \text{span}\, S_n$ as claimed. 
\end{proof}

\subsection{Wavelets and eigenspaces for $\Delta_s$}\label{sec:k-wavelet}

In this section, we prove our Theorem relating the wavelet decomposition \eqref{eq:wavelet-decomp} with the eigenspaces $E_\gamma$ of the Laplace-Beltrami operators $\Delta_s$ in the case when $A := A_1 \cdots A_k$ is irreducible.
Recall from Corollary \ref{cor:dixmier-aHLRS}  that in this case, the  measure $\mu = \mu_\delta$ used in the following theorem agrees with the measure $M$ from Proposition 8.1 of \cite{aHLRS}, which was described in Equation \eqref{eq:M-measure} above. 

\begin{thm}\label{thm:wavelet-k}
Let $\Lambda$ be a finite, strongly connected $k$-graph with adjacency matrices $A_i$. Suppose that $A=A_1\cdots A_k$ is irreducible. For any weight $w_\delta$ on the associated Bratteli diagram $\mathcal{B}_\Lambda$ as in Proposition \ref{pr:delta-weight},
{such that Equation \eqref{eq:weight-diam} is valid for the ultrametric Cantor set $(\Lambda^\infty, d_{w_\delta})$,}
the eigenspaces of the associated Laplace-Beltrami operators $\Delta_s$ refine the wavelet decomposition of \eqref{eq:wavelet-decomp}: 
\[\mathscr{V}_0 = E_{-1} \oplus E_0 \quad \text{ and } \quad \mathcal{W}_n = \text{span}\, \{ E_\gamma: |\gamma| = n k + t,\, 0 \leq t \leq k-1\}.\]
\label{thm:JS-wavelets}
\end{thm}
\begin{proof}
First observe that under  the identification of $\Lambda^0 \subseteq \Lambda$ with $\mathcal{V}_0 \subseteq \mathcal{B}_\Lambda,$ we have $E_0 \subseteq \mathscr{V}_0$ and $E_{-1} \subseteq \mathscr{V}_0$, since the spanning vectors of both $E_0$ and $E_{-1}$ are linear combinations of $\chi_{[v]}$ for vertices $v$. Thus $E_{-1} \oplus E_0\subset \mathscr{V}_0$. For the other inclusion, we compute
\begin{align*} \left( 1 + \sum_{w \not= v \in \Lambda^0} \frac{\mu[w]}{\mu[v]} \right) \chi_{[v]} &= \chi_{\Lambda^\infty} - \sum_{w\not= v \in \Lambda^0}  \chi_{[w]}  + \sum_{w \not= v} \frac{\mu[w]}{ \mu[v]}  \chi_{[v]} \\
&= \chi_{\Lambda^\infty} -\sum_{w \not= v} \mu[w] \left( \frac{1}{\mu[w]} \chi_{[w]} - \frac{1}{\mu[v]} \chi_{[v]} \right)  .
\end{align*}
By rescaling, we see that $\chi_{[v]} \in E_{-1} \oplus E_0$, and hence $\mathscr{V}_0 = E_{-1} \oplus E_0$ as claimed.

To examine the claim about $\mathcal{W}_n$, let $\eta\in F\mathcal{B}_\Lambda$ with $|\eta|= n k + t$.  In other words, $\eta$ represents an element of degree $(\overbrace{
n+1, \ldots, n+1}^t,n, \ldots, n)$ in the associated $k$-graph. Choose a typical generating element $f_\eta$ of $E_\eta$ as in \eqref{eq:eigenspace_E_gamma},
\[f_\eta = \frac{1}{\mu[\eta e]} \chi_{[\eta e]} - \frac{1}{\mu[\eta e']} \chi_{[\eta e']},\]
where $(e,e')\in \text{ext}_1(\eta)$.
Write $\eta = \eta_n \eta_t$, where $d(\eta_n) = (n, \ldots, n)$ and $d(\eta_t) = (\overbrace{1, \ldots, 1}^t, 0, \ldots, 0)$.  Enumerate the paths in $r(\eta_t) \Lambda^{(1, \ldots, 1)}$ as 
\[ \{\lambda_0, \ldots, \lambda_m, \lambda_{m+1}, \ldots, \lambda_{m+\ell}, \lambda_{m+\ell+ 1}, \ldots, \lambda_{m+\ell + p}\}\]
where the paths $\lambda_i$ for $0 \leq i \leq m$ are the extensions of $\eta_t e$ and the paths $\lambda_i$ for $m+1 \leq i \leq m+\ell$ are the extensions of $\eta_t e'$.  Then 
\begin{equation}
f_\eta = \frac{1}{\mu[\eta e]} \sum_{i=0}^m \chi_{[\eta_n \lambda_i]} - \frac{1}{\mu[\eta e']} \sum_{i=m+1}^{m+\ell} \chi_{[\eta_n \lambda_i]}.\label{eq:f-eta-as-squares}
\end{equation}
 Using \eqref{eq:f_iv} and \eqref{eq:S-lambda-cylinder}, we obtatin
\[ S_{\eta_n} f^{i, r(\eta_t)} = \rho(\Lambda)^{(n/2, \ldots, n/2)} \left( \frac{1}{\mu[\lambda_0]} \chi_{[\eta_n \lambda_0]} - \frac{1}{\mu[\lambda_i]} \chi_{[\eta_n \lambda_i]} \right),\]
and hence
\begin{equation}
\label{eq:f-eta-in-W-n}
\begin{split}
S_{\eta_n} & \left( \sum_{i=1}^m \frac{-\mu[\lambda_i]}{\mu[\eta e]} f^{i, r(\eta_t)}  + \sum_{i=m+1}^{m+\ell} \frac{\mu[\lambda_i]}{\mu[\eta e']} f^{i, r(\eta_t)} \right) \\
& = \rho(\Lambda)^{(n/2, \ldots, n/2)}\left(  \frac{1}{\mu[\eta e]}\sum_{i=1}^m \chi_{[\eta_n \lambda_i]} - \frac{1}{\mu[\eta_n e']} \sum_{i=m+1}^{m+\ell} \chi_{[\eta_n \lambda_i]} \right. \\
& \qquad \qquad \qquad \left. + \frac{1}{\mu[\lambda_0]} { \chi_{[\eta_n\lambda_0]} } \left( \sum_{i=1}^m \frac{-\mu[\lambda_i]}{\mu[\eta e] } + \sum_{i=m+1}^{m+\ell} \frac{\mu[\lambda_i]}{\mu[\eta e'] } \right) \right)\\
&= \rho(\Lambda)^{(n/2, \ldots, n/2)} \left(  f_\eta + \frac{1}{\mu[\lambda_0]} \chi_{[\eta_n \lambda_0]} \left( \sum_{i=0}^m \frac{-\mu[\lambda_i]}{\mu[\eta e] } + \sum_{i=m+1}^{m+\ell} \frac{\mu[\lambda_i]}{\mu[\eta e'] } \right) \right).
\end{split}
\end{equation}
Since the paths $\lambda_i$, for $0 \leq i \leq m$, constitute the extensions of $\eta_t e$ with the same degree $(1,\ldots,1)$, we have $\sum_{i=0}^m \mu[\lambda_i] = \mu[\eta_t e]$.  Similarly, $\sum_{j=m+1}^{m+\ell} \mu[\lambda_j] = \mu[\eta_t e']$.  Moreover, 
\[ \frac{\mu[\eta_t e]}{\mu[\eta e]} = \rho(\Lambda)^{d(\eta e) - d(\eta_t e)}= \rho(\Lambda)^{d(\eta_n)} = \frac{\mu[\eta_t e']}{\mu[\eta e']}.\]
In other words, the coefficient of $\chi_{[\eta_n \lambda_0]}$ in Equation \eqref{eq:f-eta-in-W-n} is zero, and so $f_\eta \in \mathcal{W}_n$.

If our ``preferred path'' $\lambda_0$ is not an extension of either $e$ or $e'$, Equation \eqref{eq:f-eta-as-squares} and Equation \eqref{eq:f-eta-in-W-n} hold in a modified form without the zeroth term,  and we again have $f_\eta \in \mathcal W_n$.   In other words,
\[ E_\eta \subseteq \mathcal{W}_n \ \text{ whenever } |\eta| = n k + t.\]

To see that $\mathcal{W}_n= \bigoplus_{t=0}^{k-1} \bigoplus_{|\eta| = nk + t} E_\eta$, we again use a dimension argument.  
If $|\eta| = nk+t$, we know from \cite{julien-savinien} Theorem 4.3 that $\dim E_\eta = \#( s(\eta) \Lambda^{e_{t+1}}) -1$.  Since we have a bijection between 
\[ \bigcup_{|\eta| = nk + t} s(\eta) \Lambda^{e_{t+1}} \quad \text{ and } \quad \Lambda^{d(\eta) + e_{t+1}},\] 
\begin{align*}
\dim \left(  \bigoplus_{t=0}^{k-1} \bigoplus_{|\eta| = nk + t} E_\eta \right)& = \sum_{t=1}^k \#\Big( \Lambda^{(\overbrace{n+1, \ldots, n+1}^t, n, \ldots, n)}\Big) - \sum_{t=0}^{k-1} \#\Big( \Lambda^{(\overbrace{n+1, \ldots, n+1}^t, n, \ldots, n)}\Big) \\
&= \# (\Lambda^{(n+1, \ldots, n+1)}) - \# (\Lambda^{(n, \ldots, n)} )\\
& = \dim \mathcal{W}_n. 
\end{align*} 
\end{proof}

\begin{rmk}
\label{rmk:wavelets-eigenfcns-1graph}
Recall that a directed graph with adjacency matrix $A$ gives rise to both a stationary Bratteli diagram with adjacency matrix $A$, and a 1-graph -- namely, the category of its finite paths.  Moreover, for many 1-graphs the wavelets of \cite[Section 4]{FGKP} agree with the wavelets of \cite[Section 3]{marcolli-paolucci}.  (Marcolli and Paolucci only considered in \cite{marcolli-paolucci} strongly connected directed graphs whose adjacency matrix $A$ has entries from $\{0,1\}$; but for all such directed graphs, the wavelets of \cite[Section 4]{FGKP} agree with the wavelets of \cite[Section 3]{marcolli-paolucci}.) Thus, in this situation, Theorem \ref{thm:JS-wavelets} above implies that the  eigenspaces of the Laplace-Beltrami operators $\Delta_s$ associated to the stationary Bratteli diagram with adjacency matrix $A$, as in \cite{julien-savinien} Section 4, refine the graph wavelets from Section 3 of \cite{marcolli-paolucci}.
\end{rmk}

\subsection{Different constructions of wavelets}

 In this section, we describe two variations on the original wavelet decomposition of $L^2(\Lambda^\infty, M)$ from Section 4 of \cite{FGKP} which we discussed in the previous section.  The first of these variations has a more precise correspondence with the  eigenspaces $E_\gamma$ of the Laplace-Beltrami operators, and the second connects these eigenspaces with the $J$-shaped wavelets which were introduced in \cite{FGKP2} by four of the authors of the current paper.

To construct our first variation of the wavelet decomposition \eqref{eq:wavelet-decomp},
for each $v \in \Lambda^0$, equip $\C^{ v \Lambda^{e_1}}$ with the inner product 
\begin{equation}
 \langle \vec{v}, \vec{w} \rangle = \sum_{e \in v\Lambda^{e_1}} \overline{v_e} w_e M([e]).\label{eq:inner-prod-JS}
 \end{equation}
Then, let $\{(c_1)^{\ell, v}\}_\ell${, where $1 \leq \ell \leq \#(v\Lambda^{e_1}) -1$,} denote an orthonormal basis for the orthogonal complement of $\text{span}\{(1, 1, \ldots, 1)\}$ in $\C^{v\Lambda^{e_1}}$  with respect to this inner product.  
For $1 \leq \ell \leq \#(v\Lambda^{e_1}) -1$, define  
\[ f_1^{\ell, v}  = \sum_{e \in v\Lambda^{e_1}} (c_1)^{\ell, v}_e \chi_e.\]
Since $\chi_v = \sum_{e \in v\Lambda^{e_1}} \chi_e$, we have 
\begin{align*}
\langle \chi_w, f_1^{\ell, v} \rangle &= \delta_{v,w} \sum_{e \in v\Lambda^{e_1}} {(c_1)^{\ell, v}_e} M([e]) \\
&= 0,
\end{align*}
since each vector $(c_1)^{\ell, v}$ is orthogonal to $(1, 1, \ldots, 1)$ in the inner product \eqref{eq:inner-prod-JS}.  In other words, the functions $\{f_1^{\ell, v}\}_{\ell, v}$ are orthogonal to $\mathscr{V}_0$.

We claim that the functions $\{f_1^{\ell, v}\}_{\ell, v}$ form an orthonormal set.  To that end, 
 note that if $v \not= w$ then $\langle f_1^{\ell, v}, f_1^{\tilde \ell, w} \rangle = 0$, as there are no common extensions of $e \in v\Lambda^{e_1}$ and $\tilde e \in w\Lambda^{e_1}$.  

Now, observing that $\chi_e \chi_{\tilde e} = \delta_{e, \tilde e} \chi_e$, we compute 
\begin{align*}
\langle f_1^{\ell, v}, f_1^{\tilde \ell, v} \rangle &= \sum_{e \in v\Lambda^{e_1}} (c_1)^{\ell, v}_e (c_1)^{\tilde \ell, v}_e M([e]) \\
&= \delta_{\ell, \tilde \ell}
\end{align*} 
since we chose the vectors $\{(c_1)^{\ell, v}\}_\ell$ to be orthonormal with respect to the inner product \eqref{eq:inner-prod-JS}.

{We define our first family of ``mother wavelets'' by} 
\[\mathcal{W}_{0, 1} = \text{span} \{ f_1^{\ell, v}: v \in \Lambda^0, 1 \leq \ell \leq \#( v\Lambda^{e_1})-1 \}.\]

Now, for each edge $e \in \Lambda^{e_1}$, define an inner product on $\C^{s(e) \Lambda^{e_2}}$  by 
\begin{equation}
\label{eq:inner-prod-JS-2}
\langle \vec{v},  \vec{w} \rangle = \sum_{f \in s(e) \Lambda^{e_2}} \overline{v_f} w_f M([e f]).
\end{equation}
Let $\{(c_2)^{\ell, e}\}_{\ell}${, where $1 \leq \ell \leq \#(s(e) \Lambda^{e_2}) -1$,} be an orthonormal basis for the orthogonal complement of $\text{span} \{(1, \ldots, 1)\}$ with respect to this inner product.
For each edge $e \in \Lambda^{e_1}$ and each  $1 \leq \ell \leq \#(s(e) \Lambda^{e_2}) -1$, define 
\[ f_2^{\ell, e} := \sum_{f\in s(e) \Lambda^{e_2}} (c_2)^{\ell, e}_f \chi_{ef};\]
we claim that  $\{ f_2^{\ell, e}: e \in \Lambda^{e_1}, \ 1 \leq \ell \leq \#(s(e) \Lambda^{e_2})-1\}$ is an orthonormal set in $L^2(\Lambda^\infty, M)$ which is also orthogonal to $\mathscr{V}_0 \oplus \mathcal{W}_{0, 1}$.

To see this, we perform similar calculations to the ones above to show that 
\[ \langle f_2^{\ell, e}, f_1^{\ell, v} \rangle = 0\]
 (this follows because $\chi_e = \sum_{f \in s(e) \Lambda^{e_2}} \chi_{ef}$ and the vectors $(c_2)^{\ell, e}$ were chosen to be orthogonal to $(1, 1, \ldots, 1)$); that
\[ \langle f_2^{\ell, e}, f_2^{\tilde \ell, \tilde e} \rangle = \delta_{\ell, \tilde \ell} \delta_{e, \tilde e}\]
(again, this follows because the vectors $(c_2)^{\ell, e}$ were chosen to be an orthonormal set with respect to the inner product \eqref{eq:inner-prod-JS-2}); and that 
\[ \langle f_2^{\ell, \tilde e}, \chi_v \rangle =  \langle f^{\ell, \tilde e}_2, \sum_{{e \in v\Lambda^{e_1}}} \chi_e \rangle = \delta_{v, r(\tilde e)} \langle f^{\ell, \tilde e}_2, \chi_{\tilde e} \rangle = 0. \]

Thus, we can define {a second family $\mathcal{W}_{0,2}$ of mother wavelets:}
\[ \mathcal{W}_{0, 2} = \text{span} \{ f^{\ell, e}_2: e\in \Lambda^{e_1}, \ 1 \leq \ell \leq \#(s(e)\Lambda^{e_2})-1 \}.\]

In general, for any $1 \leq j \leq k$, and for any $\eta \in \Lambda^{{ e_1+  \cdots + e_{j-1}}}$, we define an inner product on $\C^{s(\eta) \Lambda^{e_j}}$  by 
\begin{equation}
\label{eq:inner-prod-JS-j}
\langle \vec{v}, \vec{w} \rangle = \sum_{e \in s(\eta) \Lambda^{e_j}} \overline{v_e} w_e M([\eta e]),
\end{equation}
pick an orthonormal basis $\{ (c_j)^{\ell, \eta}\}_\ell$ for the complement of $\text{span}\{ (1, \ldots, 1)\}$ in this inner product,
and we define 
\[ f_j^{\ell, \eta} = \sum_{e \in s(\eta) \Lambda^{e_j}} (c_j)^{\ell, \eta}_e \chi_{\eta e}.\]
Then, one checks (as above) that $\{f_j^{\ell, \eta}\}_{\ell, \eta}$ forms an orthonormal set {in $L^2(\Lambda^\infty, M)$}, so we set  
\[\mathcal{W}_{0, j} = \text{span} \{ f^{\ell, \eta}_j: \eta \in \Lambda^{{e_1 + \cdots +  e_{j-1}}}, \ 1 \leq \ell \leq  \# (s(\eta) \Lambda^{e_j}) -1\}.\]

To see that $f^{\ell, \eta}_j$ is orthogonal to $\mathcal{W}_{0, i}$ for $i < j$, pick a basis element $f^{\tilde \ell, \nu}_i $ of $ \mathcal{W}_{0, i}$  with $\nu\in \Lambda^{e_1+\ldots +e_{i-1}}$ and write 
\[ f^{\tilde \ell, \nu}_i = \sum_{e \in s(\nu) \Lambda^{e_i}} (c_i)^{\tilde \ell, \nu}_e \chi_{\nu e} =  \sum_{e \in s(\nu) \Lambda^{e_i}} \sum_{\beta \in s(e)  \Lambda^{e_{i+1}+ \cdots + e_j}} (c_i)^{\tilde \ell, \nu}_e \chi_{\nu e \beta}.\]
This  sum will collapse when we take the inner product  of $f_i^{\tilde{\ell},\nu}$ with $f^{\ell, \eta}_j$:   there is at most one choice of $e$ such that $\eta$ extends $\nu e$.  Moreover, if we write each $\beta$ above as $\tilde \beta f$ where $d(f) = e_j$, then there is also at most one choice of $\tilde \beta$ so that $\nu e \tilde \beta = \eta$.  Consequently, when we take the inner product with $f^{\ell, \eta}_j$, the sum above reduces to a sum over elements $ f \in s(\eta) \Lambda^{e_j}$.  To be precise, 
\begin{align*}
\langle f^{\ell, \eta}_j, f^{\tilde{\ell}, \nu}_i \rangle &= \delta_{\eta, \nu e \tilde{\beta}} \sum_{f \in s(\eta) \Lambda^{e_j}} \overline{(c_j)^{\ell, \eta}_f} (c_i)^{\tilde \ell, \nu}_e M([\eta f])  \\
 &= \delta_{\eta, \nu e \tilde{\beta}} (c_i)^{\tilde \ell, \nu}_e \sum_{f\in s(\eta) \Lambda^{e_j}} \overline{(c_j)_f^{\ell, \eta}} M([\eta f]) \\
 &= 0,
\end{align*}
because of our choice of $(c_j)^{\ell, \eta}$.  Thus $\mathcal{W}_{0,j}$ is orthogonal to $\mathscr{V}_0\oplus \mathcal{W}_{0,1}\oplus \dots \oplus \mathcal{W}_{0,j-1}$.

Now, for any $q \in \N$ and any $1 \leq j \leq k$, we can define 
\begin{equation}\label{eq:subspace-refine-wavelet}
\mathcal{W}_{q, j} = \text{span} \{ S_\lambda f^{\ell, \eta}_j: d(\lambda) = (q, q, \ldots, q)\}.
\end{equation}
As indicated above, we think of the spaces $\mathcal{W}_{0,j}$ as the ``mother wavelets'' and the operators $S_\lambda$ where $d(\lambda) = (q, \ldots, q)$ as the ``scaling and translation'' operators we apply to our mother wavelets to get an orthonormal basis for $L^2(\Lambda^\infty, M)$. {In fact, these ``mother wavelets'' $\mathcal W_{0,j}$ are a refinement of the original mother wavelet family $\mathcal W_0$ from \cite{FGKP}.}  The following Theorem makes this precise.
\begin{thm}
\label{thm:refined-JS-wavelets}
Let $\Lambda$ be a finite, strongly connected $k$-graph with adjacency matrices $A_i$, and let $\{S_\lambda: \lambda \in \Lambda\} \subseteq B(L^2(\Lambda^\infty, M))$ be the operators of Equation \eqref{eq:S-lambda-cylinder}.  For any $q \in \N, 1 \leq j \leq k$,  the set 
$ \{ S_\lambda f^{\ell, \eta}_j: d(\lambda) = (q, q, \ldots, q)\} \subseteq L^2(\Lambda^\infty, M)$ is  orthonormal.  Setting 
\[ \mathcal{W}_{q, j} = \text{span} \{ S_\lambda f^{\ell, \eta}_j: d(\lambda) = (q, q, \ldots, q) \},\]
we have $\mathcal{W}_{q, j} \perp \mathcal{W}_{q', j'}$ whenever $qk + j \not= q' k + j'$.  Moreover, 
\[ L^2(\Lambda^\infty, M) = \mathscr{V}_0 \oplus \bigoplus_{q\in \N} \bigoplus_{1 \leq j \leq k} \mathcal{W}_{q, j}.\]

If {we also assume that $A := A_1 \cdots A_k$ is irreducible and Equation \eqref{eq:weight-diam} holds for the  ultrametric Cantor sets ($X_{\mathcal B_\Lambda}, d_\delta)$,} then denoting by $E_\gamma$ the eigenspaces of the Laplace-Beltrami operators as in Theorem 4.3 of \cite{julien-savinien}, also described in Equation \eqref{eq:eigenspace_E_gamma}, we have  
\[ \mathcal{W}_{q, j} = \text{span } \{ E_\gamma: |\gamma| = qk + j-1\}.\]
{It follows that $\mathcal{W}_q = \bigoplus_{j=1}^k \mathcal{W}_{q, j}$.}
\end{thm}
\begin{proof}
The orthonormality checks  proceed analogously to the above computations of orthonormality for the spaces $\mathcal{W}_{0,j}$ (also see the proof of \cite{FGKP} Theorem 4.2).  To see that every element in $L^2(\Lambda^\infty, M)$ can indeed be written as a linear combination of elements of $\mathscr{V}_0$ and $\mathcal{W}_{q,j}$, we again use a dimension counting argument.  Note that for every path $\eta$ in the Bratteli diagram of length at most $qk + j$, 
\[\chi_{[\eta]} \in \text{span} \{\chi_{[\nu]}: |\nu| = qk+j\}\]
and moreover, $\{\chi_{[\nu]}: |\nu| = qk + j\}$ is an orthogonal set in $L^2(\Lambda^\infty, M)$.  Denote by $n_{qk+j} := \#( \{ \nu: |\nu| = qk+j\} )$.

On the other hand, we note that $\{ f^{\ell, \eta}_j: d(\eta) = e_1 + \cdots + e_{j-1}, 1 \leq \ell \leq \# (s(\eta) \Lambda^{e_j})-1\}$ contains 
\[\sum_{d(\eta)=e_1+\dots+e_{j-1}}(\#( s(\eta)\Lambda^{e_j}) -1)=n_j - n_{j-1}\]
elements.  Similarly, 
$\{S_\lambda f^{\ell, \eta}_j: d(\lambda) = (q, \ldots, q), d(\eta) = e_1 + \cdots + e_{j-1}, 1 \leq \ell \leq \#(s(\eta) \Lambda^{e_j})-1 \}$ contains 
\[\sum_{d(\lambda)=(q,\dots,q), d(\eta)=e_1+\cdots +e_{j-1}, s(\lambda)=r(\eta)}(\#( s(\eta)\Lambda^{e_j}) -1) =n_{qk+j}-n_{qk+j-1} \]
elements -- 
in other words, 
\[ \# \Big( S_\lambda f^{\ell, \eta}_j: d(\lambda) = (q, \ldots, q), d(\eta) = e_1 + \cdots +  e_{j-1}, 1 \leq \ell \leq \#(s(\eta) \Lambda^{e_j})-1 \Big) = n_{qk + j} - n_{qk+j-1}.\]
Thus, there are 
\[ n_0 + (n_1 - n_0) + (n_2 - n_1) + \cdots + (n_{qk + j} - n_{qk+j-1}) = n_{qk+j}\]
 elements in an orthogonal basis of $\mathscr{V}_0 \oplus \bigoplus \{\mathcal{W}_{r, i}: rk + i \leq qk + j\}$.  Since all of these basis elements are linear combinations of characteristic functions $\chi_{[\nu]}$ with $|\nu| \leq qk +j$, we have 
\[ \mathscr{V}_0 \oplus \bigoplus \{ \mathcal{W}_{r,i}: rk+i \leq qk+ j \} \subseteq \text{span} \{ \chi_{[\nu]}: |\nu| = qk+j\} .\]
Since both of these spaces have dimension $n_{qk+j}$, they must agree.  In particular,
 any characteristic function $\chi_{[\eta]}$, and hence any function in  $L^2(\Lambda^\infty, M)$, lies in $ \mathcal{V}_0 \oplus \bigoplus_{q=0}^\infty \bigoplus_{j=1}^k \mathcal{W}_{q,j}$ as claimed.

For the  assertion relating the eigenspaces $E_\gamma$ of the Laplace--Beltrami operator with the spaces $\mathcal{W}_{q,j}$, where $|\gamma| = qk + j-1$, recall that a spanning set of $E_\gamma$ is given by the functions $f_{\gamma}$, where \[f_\gamma = \frac{1}{M([\gamma e])} \chi_{[\gamma e]} - \frac{1}{M([\gamma e'])} \chi_{[\gamma e']}\]
for edges $e, e' \in \text{ext}_1(\gamma)$, and $\langle f_{\gamma}, \chi_{[\gamma]} \rangle = \frac{1}{M([\gamma e])} M([\gamma e]) - \frac{1}{M([\gamma e'])} M([\gamma e']) = 0.$
{
Since $\chi_{[\gamma]} = \sum_{e \in s(\gamma) \Lambda^{e_j}} \chi_{[\gamma e]}$, our choice of the inner product \eqref{eq:inner-prod-JS-j} which we used to define the functions $f^{\ell, \gamma}_j$ spanning $\mathcal W_{0, j}$ means that 
\[ \langle f_\gamma, \chi_{[\gamma]} \rangle = 0 \ \forall \ (e, e') \in \text{ext}_1(\gamma), \text{ and consequently } f_\gamma \in \mathcal{W}_{0,j}.\]}
Thus, $E_\gamma \subseteq \mathcal{W}_{q, j}$.  Again, dimension considerations (as in the proof of Theorem \ref{thm:JS-wavelets} above, or \cite{julien-savinien} Theorem 4.3) tell us that since  $\dim E_\gamma = \# (s(\gamma) \Lambda^{e_j}) -1$, 
\[ \dim \mathcal{W}_{q, j} = n_{qk +j} - n_{qk + j -1} = \sum_{\gamma: |\gamma| = qk+j-1} \dim E_\gamma. \qquad \qquad \]
{The final assertion, that $\mathcal W_q = \bigoplus_{j=1}^k \mathcal{W}_{q,j}$, now follows from the equality $\mathcal W_q = \bigoplus_{t=0}^{k-1} \{ E_\gamma : | \gamma| = qk +t\}$ of Theorem \ref{thm:JS-wavelets}}.
\end{proof}

We now discuss our second variation on the original wavelet decomposition of $L^2(\Lambda^\infty, M)$.
In Theorem 5.2 of \cite{FGKP2}, four of the authors of the current paper showed how to modify the construction of the wavelets described in \eqref{eq:wavelet-decomp}.  Fix $J = (J_1, \ldots, J_k) \in \N^k$ with $J_i >0 $  for all $1 \leq i \leq k$. For each integer $j \geq 1$ we set 
\[ \mathcal{W}_{0}^J = \text{span} \{ \chi_{[\lambda]}: d(\lambda) = J \} \cap \mathscr{V}_0^\perp, \qquad \mathcal{W}_j^J = \{ S_\lambda f: f \in \mathcal{W}_0^J, d(\lambda) = j J\},\]
then we have 
\begin{equation}L^2(\Lambda^\infty, M) = \mathscr{V}_0 \oplus \bigoplus_{j=0}^\infty \mathcal{W}_j^J.\label{eq:J-wavelets}
\end{equation}
This wavelet decomposition can also be shown to agree with the eigenspaces of a Laplace-type operator associated to a spectral triple for the infinite path space of a Bratteli diagram, as in the previous sections.  Since the proofs exactly parallel our work in Sections \ref{sec:zeta-regular} and \ref{sec-wavelets-as-eigenfunctions} of the current paper, we merely sketch the construction here.

Given a finite, strongly connected $k$-graph $\Lambda$ and $J \in \N^k$ with $J_i > 0$ for all $i$, we construct an associated Bratteli diagram $\mathcal{B}_\Lambda^J = ((\mathcal{V}_\Lambda^J)^n, (\mathcal{E}_\Lambda^J)^n)$ with $r((\mathcal{E}_\Lambda^J)^n) = (\mathcal{V}_\Lambda^J)^{n-1}$, by setting
\[ (\mathcal{V}_\Lambda^J)^n = \Lambda^0 \;\;\text{for all $n \in \N$}\]
and having $J_1$ edges of color 1, then $J_2$ edges of color 2, and so on; after $J_k$ edges of color $k$, we repeat the $J_1$ edges of color 1 and continue in ``$J$-scaled rainbow order'' ad infinitum.  More precisely, given $n \in \N$, write 
\[n = q (J_1 + \cdots + J_k) + n_1  + n_2  + \cdots + n_k ,\]
 with $n_i \leq J_i$ for all $i$ and $n_i = 0$ unless $n_{i-1} = J_{i-1}$.  Let $t$ be the largest index such that $n_t \not= 0$ (or $t=k$ if $n_i =0$ for all $i$); then 
\[  (\mathcal{E}_\Lambda^J)^n = \Lambda^{e_t}.\]
As in Proposition \ref{pr:kgraph-bratteli-inf-path-spaces}, we have $\Lambda^\infty \cong X_{\mathcal{B}_\Lambda^J}$.

Assuming that Hypothesis \ref{hypoth} holds, one can show  as in Corollary~\ref{cor:ultrametric-Cantor}  that $X_{\mathcal{B}_\Lambda^J}$ is an ultrametric Cantor set with the metric $d_{\delta}$ induced by the weight $w_\delta$ of Equation \eqref{eq:w-delta}.  Thus, we have a Pearson-Bellissard type spectral triple for $X_{\mathcal{B}_\Lambda^J} \cong \Lambda^\infty$,  whose associated Dixmier trace  agrees with the measure $M$ given in Equation~\eqref{eq:kgraph_Measure} on $\Lambda^\infty$ if $A_1^{J_1} \cdots A_k^{J_k}$ is irreducible, as in Corollary \ref{cor:dixmier-aHLRS}.  Finally, the analysis employed in Theorem \ref{thm:JS-wavelets} will show that the eigenspaces $\{E_\lambda: \lambda \in F \mathcal{B}_\Lambda^J \}$ of the Laplace-type operator of the spectral triple  agree with the wavelet decomposition \eqref{eq:J-wavelets}:
\[ \mathcal{W}_j^J = \bigoplus \{ E_\lambda: j (J_1 + \cdots + J_k) \leq |\lambda| \leq (j+1) (J_1 + \cdots + J_k) -1\}.\]
Similarly, we have an analogue of the refined wavelet decomposition of Theorem \ref{thm:JS-wavelets}: setting $\mathcal{W}_{0, 1}^J = \text{span} \{ \chi_{[\lambda]}: |\lambda| = 1 \} \cap \mathscr{V}_0^\perp$ and inductively defining 
\[\mathcal{W}_{0, \ell}^J = \text{span} \{ \chi_{[\lambda]}: |\lambda| = \ell\} \cap \left(\mathcal{W}_{0, \ell-1}^J\right)^\perp \quad \text{ for } 1 < \ell \leq J_1 + \cdots+ J_k,\]
\[ \mathcal{W}_{q, \ell }^J = \text{span} \{ S_\lambda f: f \in \mathcal{W}_{0,\ell}^J, |\lambda| = q (J_1+ \cdots + J_k) \},\]
we have $L^2(\Lambda^\infty, M) = \mathscr{V}_0 \bigoplus_{q\in \N} \bigoplus_{\ell =1}^{J_1 + \cdots + J_k} \mathcal{W}_{q, \ell}^J$ and 
\[ \mathcal{W}_{q, \ell}^J = \text{span} \{E_\gamma: |\gamma| = q (J_1 + \cdots + J_k) + \ell-1 \}.\]

\begin{rmk} 
Given a finite, strongly connected $k$-graph $\Lambda$, in addition to the wavelets and eigenspaces associated to $\Lambda$ as described above, we can also construct (in a variety of ways) a directed graph $E^J_\Lambda$, such that $(E^J_\Lambda)^\infty \cong \Lambda^\infty$.  Then, applying Theorem \ref{thm:JS-wavelets} to these directed graphs $E^J_\Lambda$, we obtain (for each $J$) a pair of compatible orthogonal decompositions of $L^2((E^J_\Lambda)^\infty, M_J) \cong L^2(\Lambda^\infty, M)$ -- the wavelet decomposition of \cite{marcolli-paolucci} and the eigenspaces of the Laplace-Beltrami operators $\Delta_s$.
To our knowledge, such a realization of the infinite path space of a $k$-graph as the infinite path space of a directed graph has not previously appeared in the $k$-graph literature, although the ``collapsing'' or ``telescoping'' procedure we use below is commonly employed when studying Bratteli diagrams.

To be precise, fix $J \in \N^k$ with $J_i > 0$ for all $i$.  If $\Lambda$ is a  finite, strongly connected $k$-graph with adjacency matrices $A_1, \ldots, A_k$, let $A^J = A_1^{J_1} A_2^{J_2} \cdots A_k^{J_k}$.  Let $E_\Lambda^J$ be the ``collapsed'' 1-graph with adjacency matrix $A^J$. 
Note that $\Lambda$ and $E_\Lambda^J$ have the same vertex set, and that the edges in $E_\Lambda^J$ are in bijection with the  morphisms of degree $J$ in $\Lambda$. 

We first observe that $(E_\Lambda^J)^\infty \cong \Lambda^\infty$.  By Remark 2.2 of \cite{kp}, any infinite path $y \in \Lambda^\infty$ is completely determined by the sequence of finite paths $\{ y(n \cdot J, (n+1) \cdot J)\}_{n \in \N}$. Since each of these finite paths $y(n \cdot J, (n+1) \cdot J)$ has degree $J$, it corresponds to a unique edge $e_{y, n} \in E_\Lambda^J$. 
Thus, just as in the proof of Proposition \ref{pr:kgraph-bratteli-inf-path-spaces} above, $y \mapsto (e_{y, n})_{n \in \N}$ is a bijection $\Lambda^\infty \to (E_\Lambda^J)^\infty$.

To see that this bijection is a homeomorphism, recall from the proof of Lemma 5.1 of \cite{FGKP2} that
\begin{equation}
\label{eq:J-topology-basis}\{ [\lambda]: \lambda \in \Lambda, \ d(\lambda) = n \cdot J, \ n \in \N\}
\end{equation}
is a basis for the topology on $\Lambda^\infty$. Moreover, the bijection between edges in $E_\Lambda^J$ and morphisms of degree $J$ in $\Lambda$ implies that the bijection $y \mapsto (e_{y, n})_n$ of the previous paragraph gives rise to a bijection between the elements $[\lambda]$ of the basis \eqref{eq:J-topology-basis} and the  cylinder sets in $(E_\Lambda^J)^\infty$.  In other words, $\Lambda^\infty \cong (E_\Lambda^J)^\infty$ as claimed.

 Denote by $M_J$ the measure of Equation \eqref{eq:kgraph_Measure} on the infinite path space $(E_\Lambda^J)^\infty$; that is, $M_J([e_1 \cdots e_n]) = (\rho(\Lambda)^{-J})^n x^\Lambda_{s(e_n)}.$
If $\lambda_i \in \Lambda$ is the path of degree $J$ which corresponds to the edge $e_i$, then \[M([\lambda_1 \cdots \lambda_n]) =M_J([e_1 \cdots e_n]),\]
so the homeomorphism $\Lambda^\infty \cong (E_\Lambda^J)^\infty$ induces an isomorphism $L^2(\Lambda^\infty, M) \cong L^2((E_\Lambda^J)^\infty, M_J)$ of Hilbert spaces.

Whenever $A^J$ is irreducible, then, Section 3 of \cite{marcolli-paolucci} tells us how to construct a wavelet decomposition of $L^2((E_\Lambda^J)^\infty, M_J)$, which by Theorem \ref{thm:JS-wavelets} above is compatible with the eigenspaces of the Laplace-Beltrami operators $\Delta_s$ associated to the stationary Bratteli diagram with adjacency matrix $A^J$.  
\end{rmk}

\section{Consani-Marcolli spectral triple for strongly connected higher-rank graphs}
\label{sec-Consani-Marcolli-spectral-triples-for-k-graphs}

In Section 6 of \cite{consani-marcolli}, Consani and Marcolli construct a spectral triple for the Cuntz-Krieger algebra $\mathcal{O}_A$ associated to a matrix $A \in M_n(\N)$.
Recall from \cite{kprr} that if $E$ is the directed graph with adjacency matrix $A$, then $\mathcal{O}_A \cong C^*(E)$.
In this section, we generalize the construction of Consani and Marcolli to build spectral triples for higher-rank graph $C^*$-algebras $C^*(\Lambda)$.
For these spectral triples (described in Theorem~\ref{thm-Consani-Marcolli-spectral-triples-k-graphs} below), it is shown in Theorem \ref{thm:CM-Dirac-wavelets} that the eigenspaces of the Dirac operator are compatible with the wavelet decomposition from \cite{FGKP}. We also discuss in Remark \ref{rmk:CM-rectangle-wavelets} at the end of the section how to modify the construction of the spectral triple to make the eigenspaces of the Dirac operator compatible with the $J$-shape wavelets of \cite{FGKP2}.

 In our construction of the spectral triples of Theorem~\ref{thm-Consani-Marcolli-spectral-triples-k-graphs}, we were also inspired by several other spectral triples associated to Bratteli diagrams or fractal sets: namely, Christensen and Ivan's spectral triples \cite{christensen-ivan-AF} for AF algebras, Julien and Putnam's work \cite{JP} on spectral triples for subshifts, and the spectral triples for certain hyperbolic dynamical systems studied by Deeley, Goffeng, Mesland and Whittaker~\cite{DGMW}.

We begin by reviewing the definition of the $C^*$-algebra $C^*(\Lambda)$ of a higher-rank graph $\Lambda$.

\subsection{The $C^*$-algebra of a higher-rank graph}\label{subsection:k-graph}
Let $\Lambda$ be a finite $k$-graph with no sources. Kumjian and Pask defined $C^*(\Lambda)$ in \cite{kp} to be the universal $C^*$-algebra generated by a collection of partial isometries $\{s_\lambda\}_{\lambda \in \Lambda}$ satisfying the Cuntz-Krieger conditions:
\begin{itemize}
\item[(CK1)] $\{ s_v: v \in \Lambda^0\}$ is a family of mutually orthogonal projections;
\item[(CK2)] Whenever $s(\lambda) = r(\eta)$ we have $s_\lambda s_\eta = s_{\lambda \eta}$;
\item[(CK3)] For any $\lambda \in \Lambda, \ s_\lambda^* s_\lambda = s_{s(\lambda)}$;
\item[(CK4)] For all $v \in \Lambda^0$ and all $n \in \N^k$, $\sum_{\lambda \in v\Lambda^n} s_\lambda s_\lambda^* = s_v$.
\end{itemize}
Given $m_1, m_2 \in \N^k$, we write $m_1 \vee m_2$ for the coordinate-wise maximum of $m_1$ and $m_2$.  For a pair $(\lambda, \eta) \in \Lambda \times \Lambda$ with $r(\lambda) = r(\eta)$, we define 
\[ \Lambda^{min}(\lambda, \eta) = \{ ( \alpha, \beta) : \lambda \alpha = \eta \beta \text{ and } d(\lambda \alpha) = d(\lambda) \vee d(\eta)\}.\]
If $\Lambda$ is a 1-graph, then $|\Lambda^{min}(\lambda, \eta)| \in \{ 0, 1\};$ however, this need not true for higher-rank graphs with $k > 1$ (c.f.~the 2-graphs of \cite{LLNSW}  Example 7.7).

Condition (CK4) implies that for any $\lambda, \eta \in \Lambda$ we have 
\[ s_\lambda^* s_\eta = \sum_{(\alpha, \beta) \in \Lambda^{min}(\lambda, \eta)} s_\alpha s_\beta^*,\]
where we interpret empty sums as zero.  Consequently, $C^*(\Lambda)$ is the closed linear span of $\{ s_\lambda s_\eta^*\}_{\lambda, \eta \in \Lambda}$.  In what follows, denote by $\mathcal{A}_\Lambda$ the dense $*$-subalgebra of $C^*(\Lambda)$ spanned by $\{ s_\lambda s_\eta^*\}_{\lambda, \eta \in \Lambda}$.

\subsection{The Consani-Marcolli $k$-graph spectral triples}
  The Dirac operator and spectral triple defined in this section generalize those of \cite{consani-marcolli} Section 6.2 from the $C^*$-algebras of directed graphs to those of higher-rank graphs.
 
\begin{defn} Let $\Lambda$ be a finite, strongly connected $k$-graph.  Define $\mathcal{R}_{-1} \subset L^2(\Lambda^\infty, M)$ to be the linear subspace of constant functions on $\Lambda^\infty$. 
 For $s\in\N$, define  $\mathcal{R}_s \subset L^2(\Lambda^\infty, M)$ by
 \[
 \mathcal{R}_s=\text{span} \left\{ \chi_{[\eta]} : \ \eta \in \Lambda, \ \sup \{ d(\eta)_i: 1 \leq i \leq k\} \leq s \right\},
 \]
 where $d(\eta) = (d(\eta)_1, \ldots, d(\eta)_k) \in \N^k$.
 
 Let $\Xi_s$ be the orthogonal projection in $L^2(\Lambda^\infty, M)$ onto the subspace $\mathcal{R}_s$.
       For a pair $(s,r)\in \N\times ( \N \cup \{-1\}) $ with $s>r$, let
        \[
        \widehat{\Xi}_{s,r} =\Xi_s  - \Xi_{r}.
        \]
        Since $\mathcal{R}_r \subset \mathcal{R}_s$, $\widehat{\Xi}_{s,r}$ is the orthogonal projection onto the subspace $\mathcal{R}_s \cap ({\mathcal{R}_{r} })^{\perp}$.

      Given an increasing sequence $\alpha= \{ \alpha_q\}_{q\in \N}$  of positive real numbers with $\lim_{q \to \infty} \alpha_q = \infty $, we define an operator $D$ on $L^2(\Lambda^\infty, M)$ by
      
\begin{equation} \label{eq:Dirac}        
D:=\sum_{q\in \N} \alpha_q\;\widehat{\Xi}_{q, q-1}.
\end{equation}

\end{defn}
 
Let $\pi$ be the representation of $C^*(\Lambda)$ on $L^2(\Lambda^\infty, M)$ described in  Proposition 3.4 and Theorem 3.5 of \cite{FGKP} (also see Equations \eqref{eq:S_lambda} and \eqref{eq:S-lambda-star} below). The aim of this section is to prove that this representation makes $(\mathcal{A}_\Lambda, L^2(\Lambda^\infty, M), D)$ into a spectral triple, and to describe the eigenspaces of $D$ in terms of the wavelet decomposition of \cite{FGKP}.  To prove these statements, we begin by establishing that $D$ is a self-adjoint unbounded operator with compact resolvent. 

\begin{prop}\label{pr:dirac-self-adjoint}
The operator $D$ on $L^2(\Lambda^\infty, M)$ of Equation \eqref{eq:Dirac}  is unbounded and self-adjoint.
  \end{prop}
 
 \begin{proof} 
 The fact that $D$ is unbounded follows from the hypothesis that $\lim_{q \to \infty} \alpha_q = \infty$. 
Thus, to see that $D$ is self-adjoint we must first check that it is densely defined, and then show that $D$ and  $D^*$ have the same domain.  For the first assertion, recall from Lemma 4.1 of \cite{FGKP} that 
\[ \{[\eta]: d(\eta) = (n, \ldots, n)\text{ for some }n \in \N\} \]
 generates the topology on $\Lambda^\infty$, and hence
 \[
 \text{span} \{\chi_{[\eta]}:d(\eta)=(n,n,\dots,n),\; n\in \N\}
 \]
 is dense in $L^2(\Lambda^\infty, M)$.
  Given such a ``square'' cylinder set $[\eta]$ with $d(\eta) =(s, \ldots, s)$, since $\chi_{[\eta]} \in \mathcal{R}_s$, we can write $\chi_{[\eta]} = \sum_{r\leq s} \widehat{\Xi}_{r, r-1}(\chi_{[\eta]})$.  Then, 
\[ 
D (\chi_{[\eta]}) = \sum_{r\leq s} \alpha_r \widehat{\Xi}_{r,r-1} (\chi_{[\eta]}) ,
\]
which is a finite linear combination of vectors with finite $L^2$-norm, and hence is in $L^2(\Lambda^\infty, M)$. 
 In other words, for any finite linear combination $\xi$ of characteristic functions of square cylinder sets, $D\xi$ is in $L^2(\Lambda^\infty, M)$.  Thus $D$ is defined on (at least) the finite linear combinations of square cylinder sets, which form a dense subspace of $L^2(\Lambda^\infty, M)$.
 
Moreover, our definition of $D$ as a diagonal operator on $L^2(\Lambda^\infty, M)$ with real eigenvalues implies that  $D = D^*$ formally; since the operators $D$ and $D^*$ are given by the same diagonal formula, their domains also agree, and hence we do indeed have $D = D^*$ as unbounded operators.
 \end{proof}

\begin{prop} 
Let $D$ be the operator on $L^2(\Lambda^\infty, M)$ given in \eqref{eq:S-lambda-cylinder}.      For all complex numbers $\lambda \not\in \{ \alpha_n \}_{n \in \N}$, the resolvent $R_\lambda(D) := (D - \lambda)^{-1}$ is a compact operator on $L^2(\Lambda^\infty, M)$.
      \label{pr:cpt-resolvent}
    \end{prop}
\begin{proof}
By definition, $D$ is given by multiplication by $\alpha_q$ on $\mathcal{R}_{q}\cap \mathcal{R}_{(q-1)}^\perp$.  Consequently, for all $q 
\in \N$, $(D- \lambda)^{-1}$ is given by multiplication by $\frac{1}{\alpha_q - \lambda}$ on $\mathcal{R}_{q}\cap \mathcal{R}_{(q-1)}^\perp$.

Since $\lambda \not \in \{\alpha_n\}_{n\in \N}$ and $\lim_{n\to \infty} \alpha_n=\infty$, given $\epsilon > 0 ,$ we can choose $N$ so that for all $n \geq N$, $\frac{1}{|\lambda - \alpha_n|} < \epsilon$.  Fix $s\in \N$, then for any $f \in \mathcal{R}_s \cap \mathcal{R}_{s-1}^\perp$ of norm 1, 
\begin{align*}
 \| \left( \sum_{q=1}^N \frac{1}{\alpha_q - \lambda}\widehat \Xi_{q, q-1} (f) \right)& - (D-\lambda)^{-1}(f) \|  
= \| \sum_{q > N} \frac{1}{\alpha_q - \lambda} \widehat{\Xi}_{q, q-1}(f) \|  \\
&= \begin{cases}
\left| \frac{1}{\alpha_s - \lambda} \right| \, \|f\| & \text{ if } s > N \\
0 & \text{ if } s \leq N
\end{cases} \\
& < \epsilon,
\end{align*}
since $\|f \| = 1$ by hypothesis.  Since the subspaces $\{  \mathcal{R}_s \cap \mathcal{R}_{s-1}^\perp: s \in \N_0\}$ span $L^2(\Lambda^\infty, M)$,
it follows that $(D - \lambda)^{-1}$ is the norm limit of finite rank operators and hence is compact. 
\end{proof}

According to Proposition~3.4 and Theorem~3.5 of \cite{FGKP}, there is a separable representation $\pi$ of $C^*(\Lambda)$ on $L^2(\Lambda^\infty, M)$ when $\Lambda$ is a finite, strongly connected $k$-graph. We will prove in Theorem \ref{thm-Consani-Marcolli-spectral-triples-k-graphs} below that this representation makes $(\mathcal{A}_\Lambda, L^2(\Lambda^\infty, M), D)$ into a spectral triple; recall that  $\mathcal{A}_\Lambda$ is the dense $*$-subalgebra of $C^*(\Lambda)$ spanned by $\{s_\lambda s_\eta^*: \lambda, \eta \in \Lambda\}$.

{Before stating Theorem \ref{thm-Consani-Marcolli-spectral-triples-k-graphs}, we review the definition of the representation $\pi$.
For $p\in \N^k$ and $\lambda\in \Lambda$, let $\sigma^p$ and $\sigma_\lambda$ be the shift map and prefixing map given in Remark~\ref{rmk:S2_shift}(b).
If we let $S_\lambda:=\pi(s_\lambda)$, the image of the standard generator $s_\lambda$ of $C^*(\Lambda)$, then Theorem 3.5 of \cite{FGKP} tells us that $S_\lambda$ is given on characteristic functions of cylinder sets by
\begin{equation}\label{eq:S_lambda}
\begin{split}
S_\lambda\chi_{[\eta]} (x) &=\chi_{[\lambda]}(x)\rho(\Lambda)^{d(\lambda)/2}\chi_{[\eta]}(\sigma^{d(\lambda)}(x))\\
     &=\begin{cases} \rho(\Lambda)^{d(\lambda)/2}\quad \text{if $x=\lambda\eta y$ for some $y \in \Lambda^\infty$}\\ 
     0 \quad\quad\quad \text{otherwise}\end{cases}\\
     &= \rho(\Lambda)^{d(\lambda)/2} \chi_{[\lambda \eta]}(x).
\end{split}
\end{equation}

   Moreover, the  adjoint $S^*_\lambda$ of  $S_\lambda$  is given on characteristic functions of cylinder sets by
    \begin{equation}
\label{eq:S-lambda-star}    
    \begin{split}
    S^*_\lambda \chi_{[\eta]}(x) &=\chi_{[s(\lambda)]}(x) \rho(\Lambda)^{-d(\lambda)/2}\chi_{[\eta]}(\sigma_\lambda(x))\\
    &=\begin{cases} \rho(\Lambda)^{-d(\lambda)/2}\quad\text{if $\lambda x=\eta y$ for some $y\in \Lambda^\infty$}\\ 0 \quad\quad\quad\text{otherwise}\end{cases}\\
    &= \rho(\Lambda)^{-d(\lambda)/2}\sum_{(\zeta, \xi) \in \Lambda^{min}(\lambda, \eta)} \chi_{[\zeta]}(x).
   \end{split}\end{equation}
 }

 \begin{thm} \label{thm-Consani-Marcolli-spectral-triples-k-graphs} 
 Let $\Lambda$ be a finite, strongly connected $k$-graph, and denote by $\pi$ the representation of $C^*(\Lambda)$ on $L^2(\Lambda^\infty, M)$ given by Proposition 3.4 and Theorem 3.5 of \cite{FGKP}. Let $\mathcal{A}_\Lambda$ be the dense $\ast$-subalgebra of $C^*(\Lambda)$ given in Section~\ref{subsection:k-graph} and let $D$ be the operator given in \eqref{eq:Dirac}.
 If {there exists a constant $C\ge 0$ such that} the sequence $\alpha= \{\alpha_q\}_{q \in \N_0}$ satisfies 
 \[
| \alpha_{q+1} - \alpha_q | \leq C,\ \forall q \in \N_0, 
 \] 
 then the commutator $[D,\pi(a)]$ is a bounded operator on $L^2(\Lambda^\infty, M)$ for any $a \in \mathcal{A}_\Lambda$.
 
 Combined with the above results, this implies that the data $(\mathcal{A}_\Lambda,  L^2(\Lambda^\infty, M), D)$ gives a spectral triple for $C^*(\Lambda)$. 
 \end{thm}
\begin{proof}
To prove that $(\mathcal{A}_\Lambda, L^2(\Lambda^\infty, M), D)$ is a spectral triple we need to show that $D$ is self-adjoint, $(D^2+1)^{-1}$ is compact and $[D,\pi(a)]$ is bounded for all $a \in \mathcal{A}_\Lambda$. The first statement is the content of Proposition \ref{pr:dirac-self-adjoint}, and the second follows from Proposition \ref{pr:cpt-resolvent}, thanks to the fact that $\pm i \not \in \{ \alpha_n\}_{n\in \N}$ and hence $(D \pm i )^{-1}$ is compact. 
Thus, to complete the proof of the Theorem, we will now show that $[D, \pi(a)]$ is bounded for all finite linear combinations $a = \sum_{i \in F } c_i s_{\lambda_i} s_{\eta_i}^*\in \mathcal{A}_\Lambda$, {where $c_i\in \C$.}

Given $\lambda \in \Lambda$, write $\max_\lambda = \max_j \{ d(\lambda)_j\}$   and $\min_\lambda = \min_j \{d(\lambda)_j\}$. Then the formula \eqref{eq:S_lambda} implies immediately that, for any fixed $s\in \N$,
 the operator  $S_\lambda$ on $L^2(\Lambda^\infty, M)$ takes $\mathcal{R}_s$ to $\mathcal{R}_{s+\max_\lambda}$.
      
Moreover, Equation \eqref{eq:S-lambda-star} implies that the operator $S_\lambda^*$ on $L^2(\Lambda^\infty, M)$ takes $\mathcal{R}_s$ to $\mathcal{R}_{s-\min_\lambda}$ if $\min_\lambda \leq s$, and to $\mathcal{R}_0$ otherwise.

To see this, suppose $\chi_{[\eta]} \in \mathcal{R}_s$ and $d(\eta) = (n_1, \ldots, n_k)$. 
Then $S_\lambda^* \chi_{[\eta]}$ is a linear combination of cylinder sets $\chi_{[\zeta]}$ with 
\[d(\zeta)_i = \begin{cases} 0, & d(\lambda)_i  \geq d(\eta)_i \\
d(\eta)_i - d(\lambda)_i, & d(\lambda)_i < d(\eta)_i
\end{cases}\]
Consequently, we see that (as desired) 
     \begin{align*}
     \max\{ d(\zeta)_i \} & = \max\{0,  n_i - d(\lambda)_i: 1 \leq i \leq k \}  \leq s - \textstyle{\min_\lambda}. \end{align*}
If $s<\min_\lambda$, then $n_i - d(\lambda)_i \leq 0$ for all $i$, so  $S^*_\lambda \chi_{[\eta]}\in \mathcal{R}_0$ for all $\chi_{[\eta]}\in \mathcal{R}_s$.
      
  Similarly,  if $f \in \mathcal{R}_s^\perp$, then $S_\lambda f \in \mathcal{R}_{s + \min_\lambda}^\perp$.  
      This follows from the fact that $\langle S_\lambda f, g \rangle = \langle f, S_\lambda^*g\rangle $ and the fact that $S_\lambda^*$ takes $\mathcal{R}_r$ to $\mathcal{R}_{r -\min_\lambda}$.  Thus, if $\langle f, h \rangle = 0$ for all $h \in \mathcal{R}_s$, then in particular 
      \[ \langle f, S_\lambda^* g\rangle = 0 \ \forall \ g \in \mathcal{R}_{s + \min_\lambda}.\]
An analogous argument, using the fact that $S_\lambda$ takes $\mathcal{R}_s$ to $\mathcal{R}_{s+\max_\lambda}$, shows that  $S_\lambda^*$  takes $\mathcal{R}_s^\perp$ to $\mathcal{R}_{s-\max_\lambda}^\perp$ if $s \geq \max_\lambda$. 
 
Now fix $q \in \N, \ f\in \mathcal{R}_q \cap \mathcal{R}_{q-1}^{\perp}$,  and fix $\lambda, \mu \in \Lambda$ with $s(\lambda) = s(\mu)$. 
We use  the reasoning of the previous paragraphs to identify the subspaces $\mathcal{R}_s, \mathcal{R}_t^\perp$ which contain  $S_\lambda S_\mu^* f$.

  If $ \max_\mu > q$, then we cannot say that $S_\mu^* f$ is orthogonal to any $\mathcal{R}_t$; we cannot guarantee that $\langle f, S_\mu \xi \rangle = 0$ for any function $\xi$, since even if $\xi \in \mathcal{R}_{-1}$, $S_\mu \xi$ will lie in $\mathcal{R}_{\max_\mu-1}$, which space strictly contains $\mathcal{R}_{q-1}$. (In general, $S_\mu \xi$ will lie in  some subspace containing $\mathcal{R}_{\max_\mu-1}$.) 
Moreover, if $q < \min_\mu$, then $S_\mu^* f\in \mathcal{R}_0$.  Thus,  
  \[ q < \textstyle{\min_\mu} \Rightarrow S_\lambda S_\mu^* f \in \mathcal{R}_{\max_\lambda}; \quad \textstyle{\min_\mu} \leq q \leq  \max_\mu \Rightarrow S_\lambda S_\mu^* f \in\mathcal{R}_{q + \max_\lambda - \min_\mu}; \]
  \[ q > \textstyle{\max_\mu} \Rightarrow S_\lambda S_\mu^* f \in \mathcal{R}_{q + \max_\lambda - \min_\mu} \cap \mathcal{R}_{(q-1) + \min_\lambda - \max_\mu}^\perp.\]

For now, assume $q > \max_\mu$.  Writing $g= S_\lambda S_\mu^*  f$,
  we have 
 \begin{align*} 
 g & = \Big( \Xi_{q+\max_\lambda - \min_\mu }-\Xi_{(q-1)+\min_\lambda - \max_\mu} \Big) g \\
 &=\sum_{w=q+\min_\lambda - \max_\mu }^{q+\max_\lambda - \min_\mu }  \Big( \Xi_{w}-\Xi_{w-1} \Big) g
\end{align*}
and consequently 
  $$D g= \sum_{w=q+\min_\lambda - \max_\mu }^{q+\max_\lambda - \min_\mu } D \Big( \Big( \Xi_{w}-\Xi_{w-1} \Big) g \Big) = \sum_{w=q+\min_\lambda-\max_\mu}^{q+\max_\lambda-\min_\mu} \, \alpha_w\,  \Big( \Big( \Xi_{w}-\Xi_{w-1} \Big) g \Big). $$ 
 It now follows that (still assuming $q > \max_\mu$)
    \[
   [D, S_\lambda S_\mu^* ] f= DS_\lambda S_\mu^* f - S_\lambda S_\mu^* D f =  \sum_{w=q+\min_\lambda - \max_\mu }^{q+\max_\lambda - \min_\mu }  \, ( \alpha_w - \alpha_q) \Big( \Big( \Xi_{w}-\Xi_{w-1} \Big) S_\lambda S_\mu^* f \Big)  .
    \]
 Consequently,  since $|\alpha_w - \alpha_{w-1}| \leq C$ for all $w$, 
    \[
    \begin{split}
      \Vert  [D, S_\lambda S_\mu^* ] f \Vert &\leq   \sum_{w=q+\min_\lambda - \max_\mu }^{q+\max_\lambda - \min_\mu }  \, |\alpha_w-\alpha_q|\,  \Vert  S_\lambda S_\mu^* f \Vert  \\ 
        &\leq \|S_\lambda S_\mu^* f\|  \sum_{w= q+ \min_\lambda - \max_\mu}^{q + \max_\lambda - \min_\mu} C |w-q|\\
      & = \|S_\lambda S_\mu^*  f\|  C \sum_{t = \min_\lambda - \max_\mu }^{\max_\lambda - \min_\mu } | t|  .
     \end{split}
        \]
  Since $S_\lambda S_\mu^*$ is a partial isometry and hence norm-preserving,  
  whenever $f \in \mathcal R_q \cap \mathcal R_{q-1}^\perp$ for $q > \max_\mu$,  $\| [D, S_\lambda S_\mu^*] f\|$ is bounded above by a constant which depends only on $\lambda$ and $\mu$.

  If we have $\min_\mu \leq q \leq \max_\mu$, since we no longer know that $S_\lambda S_\mu^* f \in \mathcal{R}_t$ for any $t$, in calculating $\| [D, S_\lambda S_\mu^* f\ \|$ we have to begin our summation over $w$ at zero, rather than at $q +  \min_\lambda - \max_\mu$. 
  In this case, the final (in)equality above becomes 
  \[ \| [D, S_\lambda S_\mu^*] f \| \leq \sum_{t = 1}^{\max_\lambda - \min_\mu} C t \|S_\lambda S_\mu^* f\|  + \sum_{t=1}^q  C t \| S_\lambda S_\mu^* f \|.\]
 In this case, $q \leq \max_\mu$, so we obtain the norm bound 
  \begin{align*} \| [D, S_\lambda S_\mu^*] f \| & \leq \| S_\lambda S_\mu^* f\| C \left( \frac{(\max_\lambda - \min_\mu)(\max_\lambda - \min_\mu +1)}{2} + \frac{\max_\mu (\max_\mu +1)}{2} \right) .
  \end{align*}
  In other words, $\|[D, S_\lambda S_\mu^*] f\|$ is again bounded by a constant which only depends on $\lambda$ and $\mu$.  A similar argument shows that if $q < \min_\mu$, $\Vert [D,S_\lambda S^*_\mu] f\Vert$ is bounded by a constant which only depends on $\lambda$ and $\mu$.  Since $\{\mathcal R_q \}_q$ densely spans $L^2(\Lambda^\infty, M)$, it follows that $[D, S_\lambda S_\mu^* ]$ is a bounded operator for all $(\lambda, \mu) \in \Lambda \times \Lambda$ with $s(\lambda) = s(\mu)$.
   
By linearity, it follows  that $[D, \pi(a)]$ is bounded for all finite linear combinations $a = \sum_{i \in F} c_i s_{\lambda_i} s_{\eta_i}^*$ of the generators $s_\lambda s_\eta^*$ of $C^*(\Lambda)$. Since every element of the dense $*$-subalgebra $\mathcal{A}_\Lambda$ of $C^*(\Lambda)$ is given by such a finite linear combination, it follows that $(\mathcal{A}_\Lambda, L^2(\Lambda^\infty, M), D)$ is a spectral triple, as claimed.
 \end{proof}

 \begin{thm}
 \label{thm:CM-Dirac-wavelets}
 Let $(\mathcal{A}_\Lambda, L^2(\Lambda^\infty, M), D)$ be the spectral triple described in Theorem~\ref{thm-Consani-Marcolli-spectral-triples-k-graphs}.  Then the eigenspaces of the Dirac operator agree with the wavelet decomposition 
 \[L^2(\Lambda^\infty, M) = \mathscr{V}_0 \oplus \bigoplus_{q=0}^\infty \mathcal{W}_q\]
 of Theorem 4.5 of \cite{FGKP};  we have $\mathscr{V}_0 = \mathcal{R}_0 \oplus \mathcal{R}_{-1}$ and 
 \[ \mathcal{W}_q = \mathcal{R}_{q+1}  \cap \mathcal{R}_{q}^\perp.\]
 \end{thm}
 \begin{proof}
 By definition, $\mathcal{R}_0 = E_0$ and $\mathcal{R}_{-1} = E_{-1}$; hence the fact that $E_0 \oplus E_{-1} \cong \mathscr{V}_0$, as we established in Theorem \ref{thm:wavelet-k}, gives the first assertion.  
 
 For the second, 
recall that $\mathcal{W}_q = \text{span}\,\{ S_\lambda f: f \in \mathcal{W}_0, \ d(\lambda) = (q, q, \ldots, q)\}$.  Consequently, since $\max\{d(\lambda)_i\} = \min \{ d(\lambda)_i\}$ for all such $\lambda$, each such $S_\lambda$ takes $\mathcal{R}_s \cap \mathcal{R}_{s-1}^\perp$ to $\mathcal{R}_{s+q} \cap \mathcal{R}_{s+q-1}^\perp$.  To prove the Proposition, then, it suffices to show that 
 \[ \mathcal{W}_0 = \mathcal{R}_1 \cap \mathcal{R}_0^\perp.\]

 However, again, we recall that $\mathcal{W}_0$ was constructed precisely to be the span of a family $\{f^{m,v}\}$ of functions which were orthogonal to $\mathscr{V}_0 \supseteq \mathcal{R}_0$; and every function $f^{m,v}$ was a linear combination of characteristic functions $\chi_\eta$ with $d(\eta) = (1, \ldots, 1)$.  Since  $\mathcal{R}_1$ consists of linear combinations of such characteristic functions, it follows that $\mathcal{W}_0 \subseteq \mathcal{R}_1 \cap \mathcal{R}_0^\perp$. 
 
 Since $\mathcal{W}_0$ has dimension 
 \[
 {\sum_{v \in \Lambda^0} \#( v \Lambda^{(1, \ldots, 1)})-1 = \#(\Lambda^{(1, \ldots, 1)}) - \#(\Lambda^0), }
 \]

  and $\mathcal{R}_s$ is densely spanned by $\{ \chi_{[\lambda]}: d(\lambda) = (s, \ldots, s) \}$,  the dimension counting argument used in the proof of Theorem~\ref{thm:wavelet-k} gives the desired result.
 \end{proof}
 
 \begin{rmk} \label{rmk:CM-rectangle-wavelets}
 Fix $J \in \N^k$ with $J_i > 0 $ for all $i$.  We described in Section 5 of \cite{FGKP2} how to construct wavelets with ``fundamental domain'' $J$ -- the original construction in Section 4 of \cite{FGKP} used $J=(1, \ldots, 1)$.  By defining 
 \[ \widetilde{\mathcal{R}}_s = \text{span} \{ \chi_{[\eta]}: d(\eta) \leq sJ\}\]
 we can construct a Dirac operator $\widetilde{D}$ on $L^2(\Lambda^\infty, M)$ which gives rise to a spectral triple $(\mathcal{A}_\Lambda, L^2(\Lambda^\infty, M), \widetilde D)$ whose eigenspaces agree with the wavelet decomposition given in Theorem 5.2 of \cite{FGKP2}.  We omit the details here as they are completely analogous to the proofs of Theorems \ref{thm-Consani-Marcolli-spectral-triples-k-graphs} and \ref{thm:CM-Dirac-wavelets} above.
 \end{rmk}

\bibliographystyle{amsplain}
\bibliography{eagbib}

\end{document}